\documentclass[11pt]{article}

\usepackage{amsmath,amsthm,amssymb}
\usepackage[T1]{fontenc}
\usepackage{enumerate}
\usepackage{bbm}
\usepackage{todonotes}
\usepackage{mathrsfs}
\usepackage{xcolor}
\usepackage{hyperref}
\usepackage[nocompress]{cite}
\usepackage{tikz}
\usepackage[capitalize]{cleveref}
\usepackage[a4paper,top=1.2 in,bottom=1.5 in]{geometry}
\usepackage{nicefrac}

\providecommand{\noopsort}[1]{} 

\allowdisplaybreaks

\theoremstyle{plain}
\newtheorem{Th}{Theorem}[section]
\newtheorem{Prop}[Th]{Proposition}
\newtheorem{Lemma}[Th]{Lemma}
\newtheorem{Cor}[Th]{Corollary}

\theoremstyle{definition}
\newtheorem{Remark}[Th]{Remark}
\newtheorem{Def}[Th]{Definition}
\newtheorem{Example}[Th]{Example}
\newtheorem{Question}[Th]{Question}

\theoremstyle{plain}
\newtheorem*{namedthm}{\namedthmname}
\newcounter{namedthm}
\makeatletter
	\newenvironment{named}[2]
	{\def\namedthmname{#1}
	\refstepcounter{namedthm}
	\namedthm[#2]\def\@currentlabel{#1}}
	{\endnamedthm}
\makeatother

\definecolor{Scarlet}{rgb}{0.78, 0.11, 0.0}
\hypersetup{	citecolor = Scarlet,colorlinks,
			linkcolor = Scarlet,
			urlcolor = Scarlet}

\newcommand{\un}{\underline}
\newcommand{\ca}{{\cal A}}
\newcommand{\cb}{{\cal B}}
\newcommand{\cf}{{\cal F}}
\newcommand{\ck}{{\cal K}}
\newcommand{\cm}{{\cal M}}
\newcommand{\xbm}{(X,{\cal B},\mu)}
\newcommand{\xbmt}{(X,{\cal B},\mu,T)}
\newcommand{\ot}{\otimes}
\newcommand{\ov}{\overline}
\newcommand{\Q}{\mathbb{Q}}
\newcommand{\R}{{\mathbb{R}}}

\newcommand{\C}{{\mathbb{C}}}
\newcommand{\Z}{{\mathbb{Z}}}
\newcommand{\N}{{\mathbb{N}}}

\newcommand{\vep}{\varepsilon}
\newcommand{\1}{\mathbbm{1}}
\newcommand{\E}{\mathbb{E}}
\renewcommand{\leq}{\leqslant}
\renewcommand{\geq}{\geqslant}
\newcommand{\mob}{\boldsymbol{\mu}}

\newcommand{\tot}{\boldsymbol{\varphi}}
\newcommand{\divides}{\mid}

\newcommand{\sB}{\mathscr{B}}
\newcommand{\sC}{\mathscr{C}}
\newcommand{\bdelta}{\boldsymbol{\delta}}
\DeclareMathOperator{\lcm}{lcm}
\newcommand{\Nk}{(N_k)_{k\geq 1}}
\newcommand{\Mk}{(M_k)_{k\geq 1}}

\newcommand{\Automaton}{M}
\newcommand{\alphB}{\cb}
\newcommand{\sqfree}{Q}
\newcommand{\abundantnumbers}{\mathscr{A}}
\newcommand{\deficientnumbers}{\mathscr{D}}
\newcommand{\perfectnumbers}{\mathscr{P}}

\numberwithin{equation}{section}

\title{Rationally almost periodic sequences, polynomial multiple recurrence and symbolic dynamics}

\author{V.\ Bergelson\thanks{The first author gratefully acknowledges the support of the NSF under grant DMS-1500575.} \and J.\ Ku\l aga-Przymus\thanks{Research supported by Narodowe Centrum Nauki UMO-2014/15/B/ST1/03736.} \and M.\ Lema{\'n}czyk\thanks{Research supported by Narodowe Centrum Nauki UMO-2014/15/B/ST1/03736 and the EU grant ``AOS'', FP7-PEOPLE-2012-IRSES, No 318910.} \and F.\ K.\ Richter}

\date{}

\begin{document}
\bibliographystyle{siam}

\thispagestyle{empty}
\maketitle

\begin{abstract}
A set $R\subset \N$ is called \emph{rational} if it is well-approximable by finite unions of arithmetic progressions, meaning that for every $\epsilon>0$ there exists a set $B=\bigcup_{i=1}^ra_i\N+b_i$, where $a_1,\ldots,a_r,b_1,\ldots,b_r\in\N$,
such that 
$$
\overline{d}(R\triangle B):=\limsup_{N\to\infty}\frac{|(R\triangle B)\cap\{1,\ldots,N\}|}{N}< \epsilon.
$$
Examples of rational sets include many classical sets of number-theoretical origin such as the set of squarefree numbers, the set of abundant numbers, or sets of the form $\Phi_x:=\{n\in\N: \frac{\tot(n)}{n}<x\}$, where $x\in[0,1]$ and $\tot$ is Euler's totient function.
We investigate the combinatorial and dynamical properties of rational sets and obtain new results in ergodic Ramsey theory.
Among other things, we show that if $R\subset \N$ is a rational set with $\overline{d}(R)>0$, then the following are equivalent:
\begin{enumerate}[(a)~~]
\item\label{itm:pmrar-a}
$R$ is divisible, i.e. $\overline{d}(R\cap u\N)>0$ for all $u\in\N$.
\item\label{itm:pmrar-b}
$R$ is an averaging set of polynomial single recurrence.
\item\label{itm:pmrar-c}
$R$ is an averaging set of polynomial multiple recurrence.
\end{enumerate}
As an application, we show that if $R\subset \N$ is rational and divisible, then for any set $E\subset \N$ with $\overline{d}(E)>0$
and any polynomials $p_i\in\Q[t]$, $i=1,\ldots,\ell$,
which satisfy
$p_i(\Z)\subset\Z$ and $p_i(0)=0$ for all $i\in\{1,\ldots,\ell\}$,
there exists $\beta>0$ such that the set
$$
\left\{n\in R:\overline{d}\Big(
E\cap (E-p_1(n))\cap \ldots\cap(E-p_\ell(n))
\Big)>\beta \right\}
$$
has positive lower density.

Ramsey-theoretical applications naturally lead 
to problems in symbolic dynamics, which involve \emph{rationally almost periodic sequences} (sequences whose level-sets are rational).
We prove that if $\ca$ is a finite alphabet, $\eta\in\ca^\N$ is rationally almost periodic, $S$ denotes the left-shift on $\ca^\Z$ and 
$$X:=\{y\in \ca^\Z : \text{each finite word appearing in $y$ appears in }\eta\},$$
then $\eta$ is a generic point for an $S$-invariant probability measure $\nu$ on $X$ such that the measure preserving system $(X,\nu,S)$ is ergodic and has rational discrete spectrum.
\end{abstract}
\small

\newpage

\tableofcontents
\normalsize

\section{Introduction}
\label{sec_intro}

The celebrated Szemer{\'e}di theorem
on arithmetic progressions \cite{MR0369312}
states that any set $S\subset\N$
having positive upper density
$\overline{d}(S)=\limsup_{N\to\infty}
\frac{|S\cap\{1,\ldots,N\}|}{N}>0$ contains
arbitrarily long arithmetic progressions. 
A (one-dimensional special case of a) polynomial generalization of Szemer{\'e}di's theorem
proved in \cite{MR1325795} states that for any
$S\subset \N$ with $\overline{d}(S)>0$
and any polynomials $p_i\in\Q[t]$, $i=1,\ldots,\ell$,
which satisfy
$p_i(\Z)\subset\Z$ and $p_i(0)=0$ for all $i\in\{1,\ldots,\ell\}$,
the set $S$ contains (many) polynomial progressions
of the form $\{a,a+p_1(n),\ldots,a+p_\ell(n)\}$.
The proof of the polynomial extension of Szemer{\'e}di's
theorem given in \cite{MR1325795} is obtained with the help
of an ergodic approach introduced by Furstenberg
(see \cite{MR0498471,Furstenberg81}). In particular, the formulated above one-dimensional polynomial Szemer{\'e}di theorem follows from the fact that for any probability space $\xbm$, any invertible measure preserving transformation $T\colon X\to X$, 
any $A\in\cb$ with $\mu(A)>0$
and any $\ell$ polynomials $p_i\in\Q[t]$ satisfying
$p_i(\Z)\subset\Z$ and $p_i(0)=0$, $i\in\{1,\ldots,\ell\}$,
there exist arbitrarily large $n\in\N$
such that
$\mu\big(A\cap T^{-p_1(n)}A\cap\ldots\cap T^{-p_\ell(n)}A\big)>0$.
As a matter of fact, one can show\footnote{We remark
that the original proof in \cite{MR1325795} established
only
\begin{equation*}
\label{eq:poly-sz-liminf}
\liminf_{N\to\infty}\frac{1}{N}\sum_{n=1}^N
\mu\Big(A\cap T^{-p_1(n)}A\cap\ldots\cap T^{-p_\ell(n)}A\Big)>0,
\end{equation*}
whereas the existence of the limit in \eqref{eq:poly-sz} was obtained later, see \cite{MR2191208,MR2151605}.}
that
\begin{equation}
\label{eq:poly-sz}
\lim_{N\to\infty}\frac{1}{N}\sum_{n=1}^N
\mu\Big(A\cap T^{-p_1(n)}A\cap\ldots\cap T^{-p_\ell(n)}A\Big)>0.
\end{equation}


One of the goals of this paper is to refine \eqref{eq:poly-sz} by considering multiple ergodic averages of the from
\begin{equation}\label{eq:rest-poly-sz}
\lim_{N\to\infty}\frac{1}{|R\cap[1,N]|}\sum_{n=1}^N\1_R(n)\mu\Big(A\cap T^{-p_1(n)}A\cap\ldots\cap T^{-p_\ell(n)}A\Big),
\end{equation}
for certain sets $R$ of arithmetic origin called \emph{rational sets}, which were introduced in \cite{MR1954690} (see \cref{def_rational} below).
We show that for any rational set $R$ the limit in \eqref{eq:rest-poly-sz} exists. Furthermore, we give necessary and sufficient conditions on $R$ for this limit to be positive.
This, in turn, allows us to obtain new refinements of the polynomial Szemer{\'e}di theorem, some of which we state at the end of this introduction.

To present the main results of our paper we need to
introduce some definitions first.

\begin{Def}[Rationally almost periodic sequences and rational sets]
\label{def_rational}
Let $\ca$ be a finite set.
We endow the space $\ca^{\N}$ with the Besicovitch pseudo-metric $d_B$ (cf. \cite{MR1512943,MR0068029}),
\begin{equation}
\label{eqn:B-metric-on-0-1}
d_B(x,y):=
\limsup_{N\to\infty}
\frac{|\{1\leq n\leq N: x(n)\neq y(n)\}|}{N}.
\end{equation}

A sequence $x\in \ca^{\N}$ is called
\emph{(Besicovitch) rationally almost periodic} or, for short, RAP
if for every $\vep>0$ there exists a periodic sequence
$y\in\ca^{\N}$
such that $d_B(x,y)<\vep$. (A more general definition of the Besicovitch pseudo-metric $d_B$ and of rationally almost periodic sequences will be introduced in Subsection \ref{sec_convergence} (see page \pageref{eqn:B-metric-on-C}) and in Subsection
\ref{sec_rational-subshifts} (see \cref{d:rational-subshifts}).)

A set $R\subset \N$ is called \emph{rational} if
the sequence $\1_R$ (viewed as a sequence in $\{0,1\}^\N$) is RAP, see \cite[Definition 2.1]{MR1954690}.
\end{Def}

Here are some examples of rational sets:
\begin{itemize}
\item
The set $\sqfree$ of \emph{squarefree numbers} (see \cite[Lemma 2.7]{MR1954690}).
\item 
The set $\abundantnumbers$ of \emph{abundant numbers}\footnote{Let $\sigma(n)=\sum_{d\mid n} d$
denote the classical \emph{sum of divisors function}.
The set of \emph{abundant numbers} and the set of \emph{deficient numbers} are defined, respectively, as $\abundantnumbers:=\{n\in\N: \sigma(n)>2n\}$ and $\deficientnumbers:=\{n\in\N: \sigma(n)<2n\}$.
(The classical set of \emph{perfect numbers} is defined as $\perfectnumbers:=\{n\in\N: \sigma(n)=2n\}$.)
\label{ftn:4}} and the set $\deficientnumbers$ of \emph{deficient numbers} (see \cref{c_rat}).
\item
For any $x\in[0,1]$, the set $\Phi_x:=\{n\in\N: \frac{\tot(n)}{n}<x\}$, where $\tot$ is Euler's totient function (also see \cref{c_rat}).
\end{itemize}
The above examples are special cases of {\em sets of multiples} and sets of {\em$\sB$-free numbers}. For $\sB\subset\N\setminus\{1\}$ the corresponding sets of multiples and $\sB$-free numbers are defined as $\cm_\sB:=\bigcup_{b\in \sB}b\N$ and $\cf_\sB:=\N\setminus \cm_\sB$, respectively. The abundant numbers (as well as the union of the abundant and perfect numbers) form a set of multiples, the deficient numbers yield an example of a $\sB$-free set and $\Phi_x$ is a set of multiples.
In \cref{seq_applications-B-free} we show that for $\sB\subset\N\setminus\{1\}$ the set $\cf_{\sB}$ is a rational set if and only if the {\em density} $d(\cf_\sB):=\lim_{N\to\infty}\frac1N|\cf_\sB\cap[1,N]|$ exists (see \cref{c_rat-0}).

A natural way of obtaining rational sets is via level-sets of RAP sequences: if $\ca=\{a_1,a_2,\ldots,a_r\}$ is a finite set and $x\in\ca^\N$ is a RAP sequence then the sets $\{n\in\N:x(n)=a_1\}, \ldots,\{n\in\N:x(n)=a_r\}$ are rational. As a matter of fact, $x\in\ca^\Z$ is RAP if and only if all its level-sets are rational. 
Examples of RAP sequences include \emph{regular Toeplitz sequences},
or, more generally, \emph{Weyl rationally almost periodic sequences} (for definitions see \cref{sec:drs.e}). In particular, \emph{paperfolding sequences}\footnote{Given an infinite binary sequence $i\in\{0,1\}^\N$, we inductively define the \emph{paperfolding sequence} $t\in\{0,1\}^\N$ with \emph{``folding instructions''} $i(1),i(2),i(3),\ldots$ as follows: set $t(1):=i(1)$ and, whenever $t(n)$ has already been defined for $n\in\{1,2,\ldots,2^k-1\}$, we define $t(n)$ for $n\in\{2^k,2^k+1,\ldots,2^{k+1}-1\}$ as $t(2^k):=i(k)$ and $t(n):=t(2^{k+1}-n)$ for $2^k<n<2^{k+1}$.
For more information on paperfolding sequences see \cite{MR1170439,MR684028}.
\label{ftn:paperfolding}
}
as well as \emph{automatic sequences coming from synchronized automata} are RAP sequences (see \cref{sec:drs.e} and \cref{SAS} for definitions and more details).


\begin{Def}[cf. {\cite[Definition 1.5]{MR1412598}}]
\label{def_ave-set-of-rec-2}
We say that $R\subset\N$ is an
\emph{averaging set of polynomial multiple
recurrence}
if for any invertible measure preserving system $\xbmt$,
$A\in\mathcal{B}$ with $\mu(A)>0$,
$\ell\in\N$ and any polynomials
$p_i\in\Q[t]$, $i=1,\ldots,\ell$, with $p_i(\Z)\subset\Z$ and $p_i(0)=0$ for all $i\in\{1,\ldots,\ell\}$,
the limit in \eqref{eq:rest-poly-sz} exists 
and is positive.
If $\ell=1$ then we speak of an \emph{averaging set of polynomial single recurrence}.
\end{Def}

An averaging set of (single or multiple) polynomial recurrence $R\subset \N$ must also be a {\em set of recurrence}, i.e. for each measure preserving system
$\xbmt$ and each $A\in\cb$ with $\mu(A)>0$ there exists $n\in R$ such that $\mu(A\cap T^{-n}A)>0$.
If we assume that the density $d(R)=\lim_{N\to\infty}\frac1N|R\cap[1,N]|$ exists and is positive
then it follows -- by considering cyclic rotations on finitely many points --
that the density of $R\cap u\N$ also exists and is positive
for any positive integer $u$. This divisibility property is a rather trivial but necessary condition for a positive density set to be ``good'' for averaging recurrence.
This leads to the following definition.

\begin{Def}
\label{def_divisibility-property}
Let $R\subset\N$. We say that $R$ is \emph{divisible}
if $d(R\cap u\N)$ exists and is positive for all $u\in\N$.
\end{Def}
Note that for rational sets the existence of $d(R)$ and $d(R\cap u\N)$ is
automatic
(cf.\ \cref{l:asympdenswords} below).
Therefore, to verify divisibility, it suffices to check the positivity of $d(R\cap u\N)$ for all $u\in\N$.

One of our main theorems asserts that for rational sets
divisibility is not only a necessary
but also sufficient condition for averaging recurrence:

\begin{Th}
\label{thm_poly-multi-rec-along-rat}
Let $R\subset \N$ be a rational set and assume $d(R)>0$.
The following are equivalent:
\begin{enumerate}[(a)~~]
\item\label{itm:pmrar-a}
$R$ is divisible.
\item\label{itm:pmrar-b}
$R$ is an averaging set of polynomial single recurrence.
\item\label{itm:pmrar-c}
$R$ is an averaging set of polynomial multiple recurrence.
\end{enumerate}
\end{Th}

It was proved in \cite{MR1954690} that every self-shift of the set $\sqfree$ of squarefree numbers, i.e., any set of the form $\sqfree-r$ for $r\in \sqfree$, is divisible and hence satisfies the hypothesis of \cref{thm_poly-multi-rec-along-rat}. Moreover, it follows from \cite{MR1954690} that a shift $\sqfree-r$ for $r\in\N$ is divisible if and only if $r\in \sqfree$. The following theorem establishes a result of similar nature for sets of $\sB$-free numbers.
\begin{Th}
\label{th:b-free-sumset-combi}
Let $\sB\subset\N\setminus\{1\}$ and
assume that $d(\cf_\sB)$
exists and is positive.
Then there exists a set $D\subset \cf_\sB$ with
$d(\cf_\sB\setminus D)=0$
such that the set $\cf_\sB-r$
is an averaging set of polynomial multiple recurrence if and only if $r\in D$.
\end{Th}

\begin{Remark}
A detailed discussion
of criteria for the existence of $d(\cf_\sB)$
will be provided in \cref{seq_applications-B-free}, see
\cref{def:bes} and \cref{daer}.
In \cref{sec_revisit} we obtain a version
of \cref{th:b-free-sumset-combi} for the case when
$d(\cf_\sB)$ does not necessarily exist, see
\cref{c_ave-rec-for-B-free-along-subseq}.
\end{Remark}

In \cref{seq_applications-B-free} we also show that in \cref{th:b-free-sumset-combi} one has $D=\cf_\sB$ if and only if the set $\sB$ is \emph{taut} (see \cref{d_taut} and \cref{cor:vb-iff}).

\cref{thm_poly-multi-rec-along-rat} motivates closer interest in RAP sequences
as an independent object.
In \cref{sec_rds} we take a dynamical approach to study RAP sequences more closely. To formulate our results in this direction, let us first recall some basic notions of symbolic dynamics.

As before, let $\ca$ be a finite set (alphabet) and let $S\colon \ca^\Z\to\ca^\Z$ denote the left-shift on $\ca^\Z$, i.e., $Sx=y$ where $x\in\ca^\Z$ and $y(n)=x(n+1)$ for all $n\in\Z$.
For $x\in \ca^\Z$ (or $x\in \ca^\N$) and $n<m$ we call
$x[n,m]=(x(n),x(n+1),\ldots,x(m))$ a
\emph{word appearing in $x$}.
Given $\eta\in \ca^\N$, let
\begin{eqnarray*}
X_\eta &:=&\{x\in\ca^\Z : (\forall n<m)(\exists k\in\N)\;\; x[n,m]=\eta[k,k+m-n-1]\}
\\
&=&\{x\in \ca^\Z : \text{each word appearing in $x$ appears in }\eta\}.
\end{eqnarray*}
Clearly, $X_\eta$ is a closed and $S$-invariant subset of $\ca^\Z$ (usually referred to as the \emph{subshift determined by $\eta$}).\footnote{When $\eta$ is (topologically) recurrent, that is, any finite word appearing in $\eta$ reappears infinitely often, then there is $\widetilde{\eta}\in\ca^{\Z}$ such that $\widetilde{\eta}[1,\infty)=\eta$ and
$X_\eta=\overline{\{S^k\widetilde{\eta} : k\in\Z\}}$.
(Cf.\ \cite{Downar}, pp.\ 189-190.) Note, however, that not all RAP sequences are recurrent.}
A sequence $\eta\in \ca^\N$ 
is called \emph{generic} for an $S$-invariant Borel probability measure $\mu$ on $\ca^{\Z}$ if
$$
\lim_{N\to\infty}\frac{1}{N}\sum_{n=0}^{N-1} f(S^n \tilde\eta)= \int_{\ca^{\Z}} f\, d\mu
$$
for all continuous functions $f\in C(\ca^\Z)$, where $\tilde\eta\in\ca^\Z$ denotes any two sided sequence extending $\eta\in\ca^\N$. Note that the above definition does not depend on the choice of the two sided extension $\tilde\eta$ of $\eta$.

For a RAP sequence $\eta\in\ca^\N$ we call the corresponding symbolic dynamical system $(X_\eta,S)$ a {\em rational subshift}.
We show in \cref{sec_rds}
that any rational sequence $\eta$ is generic for an ergodic measure $\nu$ such that $(X_\eta, \nu,S)$ has rational discrete spectrum (i.e.\ the span of all eigenfunctions of $T$ is dense in $L^2\xbm$ and all the corresponding eigenvalues are roots of unity),
a result which we believe is of independent interest:

\begin{Th}
\label{thm_orb-clos-of-rat}
Let $\eta\in\ca^\N$ be RAP.
Then there exists an $S$-invariant
Borel probability measure $\nu$ on $X_\eta$
such that $\eta$ is generic for
$\nu$
and the measure preserving system
$(X_\eta,\nu, S)$ is ergodic
and has rational discrete spectrum.
\end{Th}

In light of \cref{thm_orb-clos-of-rat}, the following result (obtained in \cref{sec_revisit}) can be viewed as a ``dynamical'' generalization of \cref{thm_poly-multi-rec-along-rat}.

\begin{Th}
\label{thm_poly-multi-rec-along-disc-rat-spec}
Let $R\subset \N$ with $d(R)>0$ and suppose $\eta:=\1_R$
is generic for a Borel probability measure $\nu$ on $X_\eta\subset\{0,1\}^\Z$
such that $(X_\eta,\nu, S)$ has rational discrete spectrum.
Then there exists an increasing sequence of natural numbers $\Nk$ such that the following are equivalent:
\begin{enumerate}[(A)~~]
\item\label{itm:pmradrs-aa}
$R$ is divisible along $(N_k)_{k\geq1}$, that is, for all $u\in\N$
$$
d^{(N_k)}(R\cap u\N):=\lim_{k\to\infty}\frac{|R\cap u\N
\cap \{1,\ldots,N_k\}|}{N_k}>0.
$$
\item\label{itm:pmradrs-cc}
$R$ is an averaging set of polynomial multiple recurrence along
$\Nk$, that is,
for all invertible measure preserving systems
$\xbmt$, $\ell\in\N$, $A\in\mathcal{B}$ with $\mu(A)>0$
and for all polynomials $p_i\in\Q[t]$, $i=1,\ldots,\ell$,
with $p_i(\Z)\subset\Z$ and $p_i(0)=0$ for $i\in\{1,\ldots,\ell\}$, one has
$$
\lim_{k\to\infty}\frac{1}{N_k}\sum_{n=1}^{N_k}
\1_R(n)\mu\Big(A\cap T^{-p_1(n)}A\cap\ldots\cap T^{-p_\ell(n)}A\Big)
>0.
$$
\end{enumerate}
\end{Th}

In Sections \ref{seq_applications-B-free} and \ref{sec_revisit}
we give various examples of (classes of) rational sets for which Theorems
\ref{thm_poly-multi-rec-along-rat} and \ref{thm_poly-multi-rec-along-disc-rat-spec} hold. 


With the help of Furstenberg's correspondence principle (see \cref{prop:fc}) we have the following combinatorial corollary of \cref{thm_poly-multi-rec-along-rat}.

\begin{Th}
\label{c:fc-rat-without-intersectivitiy}
Let $R\subset \N$ be rational and divisible.
Then for any set $E\subset \N$ with $\overline{d}(E)>0$
and any polynomials $p_i\in\Q[t]$, $i=1,\ldots,\ell$,
which satisfy
$p_i(\Z)\subset\Z$ and $p_i(0)=0$ for all $i\in\{1,\ldots,\ell\}$,
there exists $\beta>0$ such that the set
$$
\left\{n\in R:\overline{d}\Big(
E\cap (E-p_1(n))\cap \ldots\cap(E-p_\ell(n))
\Big)>\beta \right\}
$$
has positive lower density.
\end{Th}

We note that \cref{th:b-free-sumset-combi}
also yields combinatorial corollaries in the spirit of
\cref{c:fc-rat-without-intersectivitiy}, which are formulated and proved
in \cref{seq_applications-combinatorics}.

We conclude this introduction with stating an amplified version of
\cref{c:fc-rat-without-intersectivitiy}, a proof of which is also contained
in \cref{seq_applications-combinatorics}.

\begin{Th}
\label{c:fc-rat}
Let $R\subset \N$ be rational and divisible.
Then for any $E\subset \N$ with $\overline{d}(E)>0$
and any polynomials $p_i\in\Q[t]$, $i=1,\ldots,\ell$,
which satisfy
$p_i(\Z)\subset\Z$ and $p_i(0)=0$, for all $i\in\{1,\ldots,\ell\}$,
there exists a subset $R'\subset R$ satisfying $\overline{d}(R')>0$
and such that for any finite subset
$F\subset R'$, we have
$$
\overline{d}\left(\bigcap_{n\in F}\Big(
E\cap \big(E-p_1(n)\big)\cap \ldots\cap\big(E-p_\ell(n)\big)
\Big)\right)>0.
$$
\end{Th}

\paragraph{Structure of the paper:}
\cref{sec_recurrence} is divided into four subsections:
In Subsection \ref{sec_convergence} we show that RAP sequences are good weights for polynomial multiple convergence. In Subsection \ref{sec_single-rec}, we prove the equivalence \eqref{itm:pmrar-a} $\Leftrightarrow$ \eqref{itm:pmrar-b} of \cref{thm_poly-multi-rec-along-rat}. In Subsection \ref{sec_multi-rec} we give a proof of the equivalence \eqref{itm:pmrar-a} $\Leftrightarrow$ \eqref{itm:pmrar-c}. Finally, in Subsection \ref{seq_applications-B-free}, we provide more examples of rational sets, and discuss some of their properties.
This includes a discourse on $\sB$-free numbers and a proof of \cref{th:b-free-sumset-combi}.

In \cref{sec_rds} we define rational subshifts and study their dynamical properties. In particular, \cref{sec_rds} contains a proof of \cref{thm_orb-clos-of-rat}.

In \cref{sec_revisit}
 we give a proof of a strenghtening of \cref{thm_poly-multi-rec-along-rat} and in \cref{seq_applications-combinatorics} we provide various combinatorial applications of it via Furstenberg's correspondence principle.

In \cref{SAS} we prove that systems generated by Weyl rationally almost periodic sequences (see page \pageref{WRAP} for the definition) satisfy Sarnak's conjecture.

Finally, in the appendix we establish a uniform version of the polynomial multiple recurrence theorem obtained in \cite{MR1325795}, which is needed for the proof of \cref{thm_poly-multi-rec-along-rat}.

\paragraph{Acknowledgement:} We thank the anonymous referee for many helpful comments.

\section{Rationality and recurrence}
\label{sec_recurrence}

\subsection{Rational sequences are good weights for polynomial multiple convergence}
\label{sec_convergence}
The purpose of this subsection is to show that for rational sets
$R$ with $d(R)>0$, the limit in \eqref{eq:rest-poly-sz} always exists.

First, we make the following observation:
If $d(R)$ exists and is positive
then the limit in \eqref{eq:rest-poly-sz}
exists and is positive if and only if the limit
\begin{equation}\label{eqn:ac-2.2}
\lim_{N\to\infty}\frac{1}{N }\sum_{n=1}^N
\1_R(n)\mu\big(A\cap T^{-p_1(n)}A\cap\ldots\cap T^{-p_\ell(n)}A\big)
\end{equation}
exists and is positive.
Since throughout this paper we mostly consider sets $R$ for
which $d(R)$ exists (except in \cref{sec_revisit}) and is positive,
it suffices to study the ergodic averages given by \eqref{eqn:ac-2.2}
instead of \eqref{eq:rest-poly-sz}.

For the special case where $\ell=1$ and $p_1(t)=t$,
the existence of the limit in \eqref{eqn:ac-2.2}
follows from the work of Bellow and Losert in \cite{MR773063}.
To better describe what is known in this case,
we need to introduce first the following extended form of \cref{def_rational}.

\begin{Def}
\label{def:eqn:B-metric-on-C}
Given $x,y\colon\N\to\C$, we define
\begin{equation}\label{eqn:B-metric-on-C}
d_B(x,y):=\limsup_{N\to\infty}\frac{1}{N}\sum_{n=1}^N |x(n)-y(n)|.
\end{equation}
A sequence $x\colon\N\to\C$ is called
\emph{Besicovitch almost periodic} (BAP) \cite{MR1512943,MR0068029}
if for every $\vep>0$
there exists a trigonometric polynomial $P(t)=\sum_{j=1}^M c_j e^{2\pi i \lambda_j t}$ with $c_1,\ldots,c_M\in\C$ and $\lambda_1,\ldots,\lambda_M\in\R$ such that
$d_B(x,P)=d_B((x(n))_{n\in\N},(P(n))_{n\in\N})<\vep$.
If, for each $\vep>0$, one can choose $\lambda_1,\ldots,\lambda_M\in\Q$ -- which is equivalent to
the assertion that the sequence $(P(n))$ is periodic -- then
we call $x$ \emph{(Besicovitch) rationally almost periodic}, or RAP.
In particular, RAP sequences are a special type of BAP sequences.
\end{Def}

It is shown in \cite[Section 3]{MR773063} that for any bounded
BAP sequence $x\colon\N\to\C$,
the ergodic averages
$$
\lim_{N\to\infty}\frac{1}{N}\sum_{n=1}^N x(n) T^n f
$$
converge almost everywhere for any function $f\in L^1\xbm$.
From this, the existence of the limit in \eqref{eqn:ac-2.2}
for $\ell=1$ and $p_1(t)=t$ follows immediately.

\begin{Def}
\label{def:goodweights}
A sequence $x\in\{0,1\}^\N$ is called a \emph{good weight for polynomial multiple convergence} if for every invertible
measure preserving system $\xbmt$,
for all $f_1,\ldots,f_\ell \in L^\infty(X,\mu)$
and for all polynomials $p_i\in\Q[t]$, $p_i(\Z)\subset\Z$, $i\in\{1,\ldots,\ell\}$, the limit
\begin{equation}\label{eqn:ac-3}
\lim_{N\to\infty}\frac{1}{N }\sum_{n=1}^N
x(n)\prod_{i=1}^\ell  T^{p_i(n)}f_i
\end{equation}
exists in $L^2\xbm$.
\end{Def}

The following proposition shows that 
the limit in \eqref{eqn:ac-2.2} exists in general.

\begin{Prop}
\label{prop_conv-of-erg-ave}
Let $x\in\{0,1\}^\N$ be RAP.
Then $x$ is a good weight for polynomial multiple convergence.
\end{Prop}


\begin{proof}
It follows from the results of Host, Kra~\cite{MR2191208} and
Leibman~\cite{MR2151605} that the sequence
$$
\frac1N\sum_{n=1}^N T^{q_1(n)}f_1\cdot\ldots\cdot T^{q_\ell(n)}f_\ell,\;N\geq1,
$$
converges in $L^2$, for any $q_i\in\Q[t]$, $q_i(\Z)\subset\Z$, $i=1,\ldots,\ell$. In particular, given arbitrary $a\in \N,b\in\Z$ the averages
\begin{equation*}
\label{eq:gwfpmr-1}
\frac1N\sum_{n=1}^N T^{p_1(an+b)}f_1\cdot\ldots\cdot T^{p_\ell(an+b)}f_\ell,
\end{equation*}
converge in $L^2$ as $N\to\infty$.
Equivalently, the limit
\begin{equation}
\label{eq:gwfpmr-2}
\lim_{N\to\infty}\frac1{N}\sum_{n=1}^N \1_{a\N+b}(n)
T^{p_1(n)}f_1\cdot\ldots \cdot T^{p_\ell(n)}f_\ell
\end{equation}
exists.
Observe that any periodic sequence
can be written as a finite linear combination of
infinite arithmetic progressions $\1_{a\N+b}$.
Therefore, it follows from
\eqref{eq:gwfpmr-2} that
for any periodic sequence $y\in\{0,1\}^\N$ the limit
$$
\lim_{N\to\infty}\frac1{N}\sum_{n=1}^Ny(n) T^{p_1(n)}f_1\cdot\ldots \cdot T^{p_\ell(n)}f_\ell
$$
also exists in $L^2$.

Since any RAP sequence $x$
can be approximated by periodic sequences,
we can find periodic sequences $y_m$, $m\in\N$,
satisfying $d_B(y_m,x)\to 0$ as $m\to\infty$. Define
$$
L_{m}:=\lim_{N\to\infty}\frac1{N}\sum_{n=1}^Ny_m(n) T^{p_1(n)}f_1\cdot\ldots \cdot T^{p_\ell(n)}f_\ell.
$$
Then
$$
\| L_{m_1}-L_{m_2}\|_{L^2}\leq
d_B\left({y_{m_1}},{y_{m_2}}\right) \Vert f_1\Vert_{L^\infty}\cdot\ldots\cdot \Vert f_\ell\Vert_{L^\infty},$$
which shows that $(L_m)$ is a Cauchy sequence, whence the limit
$L:=\lim_{m\to\infty}L_m$ exists.
Moreover,
$$
\limsup_{N\to\infty}
\left\| L_{m} - \frac1{N}\sum_{n=1}^Nx(n) T^{p_1(n)}f_1\cdot\ldots \cdot T^{p_\ell(n)}f_\ell\right\|_{L^2}
$$
can be bounded from above by
$d_B\left(x,{y_{m}}\right) \Vert f_1\Vert_{L^\infty}\cdot\ldots\cdot \Vert f_\ell\Vert_{L^\infty}$,
which converges to zero as $m\to\infty$.
Therefore, the limit in~\eqref{eqn:ac-3} exists and equals $L$.
\end{proof}

\subsection{Averaging single recurrence}
\label{sec_single-rec}

In this subsection we provide a proof of the equivalence
\eqref{itm:pmrar-a} $\Leftrightarrow$ \eqref{itm:pmrar-b} in \cref{thm_poly-multi-rec-along-rat}. Of course, this equivalence is
a special case of the more general equivalence
\eqref{itm:pmrar-a} $\Leftrightarrow$ \eqref{itm:pmrar-c}.
We include a separate proof of this simpler case because, on the one hand,
this proof is more elementary and self-contained
and, on the other hand, it contains in embryonic form the ideas needed for the proof of the general case.
Let us state the non-trivial implication, namely
\eqref{itm:pmrar-a} $\Rightarrow$ \eqref{itm:pmrar-b}, as an independent theorem.

\begin{Th}\label{p5newnewA}
Assume that $R\subset \N$ is rational and divisible.
Then $R$ is an averaging set of polynomial single recurrence.
\end{Th}

The proof of \cref{p5newnewA} is comprised of two parts.
First, we prove the assertion for totally ergodic systems. Recall that $\xbmt$ is called \emph{totally ergodic}
if $T^m$ is ergodic for all $m\in\N$.
Equivalently, $T$ is ergodic and the spectrum of the unitary operator associated with $T$ contains no non-trivial roots of unity.

\cref{l30} below, which is the second ingredient in the proof of \cref{p5newnewA}, allows us to reduce the case of general ergodic systems to those which are totally ergodic.
This is done by replacing $T^{p(n)}$ with $T^{p(un)}$ for a highly divisible natural number $u$. Since $p(0)=0$, this allows us to identify $T^{p(un)}$ with $T^{q(n)}$ for some other polynomial $q$.
This procedure annihilates the rational part of the spectrum in the sense that will be made precise below.

In the following, we use $\ck_{rat}$ to denote
the \emph{rational Kronecker factor} of $\xbmt$, which is defined
as the smallest sub-$\sigma$-algebra of $\cb$
for which all eigenfunctions with roots of unity as eigenvalues are measurable.
Equivalently, the rational Kronecker factor is the largest factor of $T$ which has rational discrete spectrum. It is also a characteristic factor for
ergodic averages along polynomials.
This means that for any function $f\in L^2$ and any polynomial $p\in \Q[t]$, $p(\Z)\subset\Z$, one has
\begin{equation}
\label{eq_tot-erg-char-factor}
\lim_{N\to\infty}
\left\|
\frac{1}{N}\sum_{n=1}^N \Big(T^{p(n)}f-
T^{p(n)}\E(f|\ck_{rat})\Big)\right\|_{L^2}=0,
\end{equation}
where $\E(f|\ck_{rat})$ denotes the \emph{conditional expectation of $f$ with respect to $\ck_{rat}$}, i.e., the unique function in $L^2\xbm$ such that $\E(f|\ck_{rat})$ is $\ck_{rat}$-measurable and $\int_A\E(f|\ck_{rat}) d\mu =\int_A fd\mu$ for all $A\in\ck_{rat}$.
A proof of \eqref{eq_tot-erg-char-factor} can be found in \cite[Section 2]{Bergelson96}.

\begin{Lemma}\label{l30}
Let $\xbmt$ be an invertible measure preserving system and let $R\subset\N$ with $d(R)>0$. Also, let $p\in \Q[t]$ satisfy $p(\Z)\subset\Z$ and $p(0)=0$.
Assume that for each real-valued $g\in L^2\xbm$ with
$\E(g|\ck_{rat})=\int g\, d\mu >0$ there exists some $\delta>0$ such that
\begin{equation}\label{ten}
\overline{d}\big(D_{\delta}(g)\cap u\N\big)>0,\quad\forall u\in\N,
\end{equation}
where $D_{\delta}(g):=\big\{n\in R: \langle  T^{p(n)}g,g\rangle>\delta\big\}$.
Then for all non-negative $f\in L^2(X,\cb,\mu)$ with $\int_X f\, d\mu>0$, we have
\begin{equation}\label{rr:-3}
\limsup_{N\to\infty}\frac{1}{N}\sum_{n=1}^{N} \1_R(n)\langle  T^{p(n)}f,f\rangle~>~0.
\end{equation}
\end{Lemma}

\begin{proof} Fix $f\in L^2\xbm$, $f\geq0$ with $\int_X f\,d\mu>0$.
Then the function $$g^{(1)}:=f-\E(f|\ck_{rat})+\int_X f\, d\mu$$ is real-valued and satisfies $\E(g^{(1)}|\ck_{rat})=\int_X g^{(1)}\, d\mu >0$. Therefore, we can find some
$\delta>0$ such that \eqref{ten} holds for $g=g^{(1)}$.
Pick $0<\epsilon < \sqrt{\delta}$.
Let $\ck_u$ stand for the factor of $\ck_{rat}$ that is generated by eigenfunctions corresponding to eigenvalues which are roots of unity of degree at most $u$. Note that $\ck_{1!}\subset \ck_{2!}\subset
\ck_{3!}\subset \ldots$ and
$$
\ck_{rat}=\bigvee_{m\in\N} \ck_{m!}.
$$
Hence, using Doobs' martingale convergence theorem (see \cite[Section 3.4]{MR833286}),
we can find $m\in\N$ such that $\|\E(f|\ck_{rat})-\E(f|\ck_{m!})\|_{L^2}< \epsilon$.
Take $u=m!$. Define
\begin{align*}
g^{(2)}&= \E(f|\ck_u)-\int_X f d\mu,\\
g^{(3)}&=\E(f|\ck_{rat})-\E(f|\ck_u),
\end{align*}
so that $f=g^{(1)}+g^{(2)}+g^{(3)}$.
A simple calculation shows that
\begin{equation*}
\langle T^n g^{(i)},g^{(j)}\rangle=0,\quad\text{for all $i,j\in\{1,2,3\}$ with $i\neq j$ and for all $n\in \Z$}.
\end{equation*}
It follows that
\begin{align*}
\langle T^{p(n)}f,f\rangle&=\langle T^{p(n)}(g^{(1)}+g^{(2)}+g^{(3)}),g^{(1)}+g^{(2)}+g^{(3)}\rangle\\
&=\langle T^{p(n)}g^{(1)},g^{(1)}\rangle + \langle T^{p(n)}g^{(2)},g^{(2)}\rangle + \langle T^{p(n)}g^{(3)},g^{(3)}\rangle.
\end{align*}
Then, using $T^{p(n)}g^{(2)}=g^{(2)}$ for all $n\in u\N$ and $\|g^{(3)}\|_{L^2}<\vep$,
we get that for every $n\in D_{\delta}(g^{(1)})\cap u\N$,
\begin{align*}
\langle T^{p(n)}f,f\rangle&=\langle T^{p(n)}g^{(1)},g^{(1)}\rangle +\langle T^{p(n)}g^{(2)},g^{(2)}\rangle +\langle T^{p(n)}g^{(3)},g^{(3)}\rangle \\
&\geq \langle T^{p(n)}g^{(1)},g^{(1)}\rangle +\langle g^{(2)},g^{(2)}\rangle - \vep^2\\
&\geq \delta-\vep^2>0.
\end{align*}
To complete the proof, it suffices to notice that
\begin{align*}
\limsup_{N\to\infty}\frac{1}{N}\sum_{n=1}^N\1_R(n)\langle T^{p(n)}f,f \rangle
&\geq
\limsup_{N\to\infty}\frac{1}{N}\sum_{n \in D_{\delta}(g^{(1)})\cap u\N\cap\{1,\ldots,N\}} \1_{R}(n)\langle T^{p(n)}f,f\rangle\\
&\geq(\delta-\vep^2) ~\overline{d}\big(D_{\delta}(g^{(1)})\cap u\N\big)
\end{align*}
and apply~\eqref{ten} for $g^{(1)}$.
\end{proof}

\begin{proof}[Proof of \cref{p5newnewA}]
Suppose $R\subset\N$ with $d(R)>0$ is both rational and divisible.
We want to show that $R$ is a set of averaging polynomial single recurrence, i.e.\ we want to show that for
any invertible measure preserving system $\xbmt$,
$p\in\Q[t]$, $p(\Z)\subset\Z$, with $p(0)=0$ and $A\in\cb$ with $\mu(A)>0$, the limit
\begin{equation}
\label{eq_poly-rec-ave-1}
\lim_{N\to\infty}\frac{1}{N}\sum_{n=1}^N \1_R(n) \mu(A\cap T^{-p(n)}A)
\end{equation}
is positive. Note that the
limit in \eqref{eq_poly-rec-ave-1} exists by Proposition~\ref{prop_conv-of-erg-ave}.

In view of \cref{l30}, to show that \eqref{eq_poly-rec-ave-1}
is positive it suffices to show that~\eqref{ten} holds
for all real-valued $g\in L^2\xbm$ with $\E(g|\ck_{rat})=\int_X g\, d\mu >0$.
However, for any such $g$, it follows from
\eqref{eq_tot-erg-char-factor} that
\[
\lim_{N\to \infty}\frac1{N}\sum_{n=1}^N\langle T^{q(n)}g,g\rangle=\Big(\int_X g\, d\mu\Big)^2
\]
for all polynomials $q\in \Q[t]$, $q(\Z)\subset\Z$.
In particular, we can pick $q(n)=p(u(an+b))$  and obtain
\begin{equation*}\label{rr:-1}
\lim_{N\to\infty}\frac1{N}\sum_{n=1}^N\langle T^{p(u(an+b))}g,g\rangle=\Big(\int_X g\, d\mu\Big)^2 \text{ for all }a,u\in\N,~b\in\N\cup\{0\}.
\end{equation*}
This can be rewritten as
\begin{equation}\label{rr:-2}
\lim_{N\to\infty}\frac1{N}\sum_{n=1}^N\1_{a\N+b}(n)\langle T^{p(un)}g,g\rangle=\frac{1}{a}\Big(\int_X g\, d\mu\Big)^2 \text{ for all }a,u\in\N,~b\in\N\cup\{0\}.
\end{equation}
Now, if $E\subset \N$ is a finite union of infinite arithmetic
progressions then $\1_E$ can be written as a finite linear combination of functions of the form $\1_{a\N+b}$ and it follows from \eqref{rr:-2} that for any such set $E$, we have
$$
\lim_{N\to \infty}\frac{1}{N} \sum_{n=1}^N\1_E(n)\langle T^{p(un)}g,g \rangle=d(E)\Big(\int_X g\, d\mu\Big)^2.
$$
Finally, since $R\cap u\N$ is rational for all $u\in\N$
and every rational set can be approximated by
finite unions of infinite arithmetic progressions, we deduce that
\begin{equation}\label{rr:-2.2}
\lim_{N\to \infty}\frac{1}{N} \sum_{n=1}^N\1_{R\cap u\N}(n)\langle T^{p(un)}g,g \rangle=d(R\cap u\N)\Big(\int_X g\, d\mu\Big)^2,
\qquad\forall u\in\N.
\end{equation}
Choose $\delta>0$ so that $\delta(1+\|g\|_{L^2}^2)<\big(\int_X g\, d\mu \big)^2$.
It is now an immediate consequence of \eqref{rr:-2.2} that
$$
\overline{d}\Big(\big\{n\in R\cap u\N: \langle  T^{p(n)}g,g\rangle>\delta\big\}\Big)\geq
d(R\cap u\N)\delta.
$$
From this it follows that~\eqref{ten} holds.
\end{proof}

\subsection{Averaging multiple recurrence}
\label{sec_multi-rec}
In this subsection we prove
\eqref{itm:pmrar-a} $\Rightarrow$ \eqref{itm:pmrar-c} in
\cref{thm_poly-multi-rec-along-rat}.
Since \eqref{itm:pmrar-c} $\Rightarrow$ \eqref{itm:pmrar-a}
is trivial, this will complete the proof of \cref{thm_poly-multi-rec-along-rat}.
Let us state the implication that we want to prove as a separate theorem.

\begin{Th}\label{p5new}
Assume $R\subset \N$ is rational and divisible.
Then $R$ is an averaging set of polynomial multiple recurrence.
\end{Th}

For the proof of \cref{p5new}, we rely on a series of known results.
We recall first some fundamental properties of
nilsystems.

Let $G$ be a nilpotent Lie group and let $\Gamma$ be
a uniform and discrete subgroup of $G$.
The compact manifold $X:=G/\Gamma$ is called a
\emph{nilmanifold}.
$G$ acts naturally on $X$.
More precisely, if $g,y\in G$ and $x=y\Gamma\in X$
then $T_gx$ is defined as $(gy)\Gamma$.
For a fixed $g\in G$ the topological dynamical system
$(X,T_g)$ is called a \emph{nilsystem}.
Every nilmanifold $X=G/\Gamma$
possesses a unique $G$-invariant probability measure $\mu_X$,
called the \emph{Haar measure} of $X$.

A bounded function $\phi\colon \N\rightarrow \C$ is called a
\emph{basic nilsequence} if there exist a nilmanifold $X=G/\Gamma$, a point
$x\in X$, an element $g\in G$ and a continuous function $f\in C(X)$ such that
$\phi(n)=f(T_g^n x)$ for all $n\in\N$.
Here, $T_g^nx$ coincides with $T_{g^n}x$.
A function $\psi\colon \N\rightarrow \C$ is called a \emph{nilsequence}
if for each $\vep>0$ there exists a basic nilsequence $(\phi(n))$
such that $|\psi(n)-\phi(n)|<\vep$ for all $n\in\N$.

An important tool in the proof of \cref{p5new}
is a theorem of Leibman that allows us to replace
multiple ergodic averages along polynomials with Birkhoff sums of nilsequences.

\begin{Th}[cf. {\cite[Theorem 4.1]{MR2643713}} and
{\cite[Proposition 3.14]{MR2122919}}]\label{thm_4.1} Assume that $\xbmt$ is an invertible measure preserving system, let $f_1,\ldots,f_\ell\in L^{\infty}\xbm$,
$p_1,\ldots,p_\ell\in \Q[t]$ ($p_i(\Z)\subset\Z$ for $i=1,\ldots,\ell$) and set
$\varphi(n):=\int_X T^{p_1(n)}f_1\cdot\ldots \cdot T^{p_\ell(n)}f_\ell\, d\mu$, $n\in\Z$.
Then there exists a nilsequence $(\psi(n))$ such that
$$
\limsup_{N-M\to\infty}\frac{1}{N-M}\sum_{n=M}^{N-1}
|\varphi(n)-\psi(n)|~=~0.
$$
In particular, $d_B(\varphi,\psi)=0$.
\end{Th}

If $(x_n)_{n\geq 1}$ is a sequence of points from a nilmanifold $X=G/\Gamma$
such that
\[
\lim_{N\rightarrow\infty}\frac{1}{N}\sum_{n=1}^N f(x_n) = \int_X f~d\mu_X
\]
for all continuous functions $f\in C(X)$, then we call such a sequence \emph{uniformly distributed}.
If $(x_n)_{n\geq 1}$ has the property that $(x_{an+b})_{n\in\N}$
is uniformly distributed for all $a,b\in\N$, then we call
this sequence \emph{totally equidistributed}.
It is well known that for any nilsystem $(X,T_g)$
the following are equivalent (see, for instance, \cite{AGH63,Parry69}
in the case of
connected $G$ and \cite{MR2122919} in the general case):
\begin{itemize}
\item
The sequence $(T_g^n x)_{n\in\N}$ is totally equidistributed for all $x\in X$;
\item
The system $(X,\mu_X,T_g)$ is totally ergodic.
\end{itemize}

Any nilmanifold has finitely many connected components (and each such component is a sub-nilmanifold). Moreover,
 since any ergodic nilrotation $T_g$ permutes these components in a cyclical fashion,
we deduce that for some $u\in\N$ the nilrotation $T_{g^u}$ fixes
each connected component. The next proposition asserts that in this case
the action of $T_{g^u}$ on each of these connected components is totally ergodic.

\begin{Prop}[see {\cite[Proposition 2.1]{MR2415080}}]
\label{prop_2.1}
Let $X=G/\Gamma$ be a nilmanifold,  $g\in G$
and assume that the nilrotation $T_g$ is ergodic.
Fix $x\in X$ and let $Y$ denote the connected component of
$X$ containing $x$. Then there exists $u\in\N$ such that
$Y$ is $T_{g^u}$-invariant and
$(Y,\mu_Y,T_{g^u})$ is totally ergodic.
\end{Prop}

The next lemma is important for the proof of
\cref{thm_poly-multi-rec-along-rat} and asserts that
linear sequences coming from totally ergodic nilrotations
(or equivalently, totally equidistributed sequences)
do not correlate with RAP sequences.

\begin{Lemma}\label{l1}
Suppose $R\subset\N$ is rational and $T_g$
is a totally ergodic nilrotation on a nilmanifold $X=G/\Gamma$.
Then, for all $x\in X$ and $f\in C(X)$, we have
$$
\lim_{N\to\infty}\frac1{N}\sum_{n=1}^N
\1_{R}(n)f(T_g^{n}x)=d(R)\int_Xf\,d\mu_X.
$$
\end{Lemma}
\begin{proof}
Since $T_g$ is totally ergodic, we deduce that the sequence
$(T_g^n x)_{n\in\N}$ is totally equidistributed for each $x\in X$.
Therefore,  for all $a\in\N$ and $b\in\N\cup\{0\}$, we obtain
\begin{equation*}\label{rr:3-1}
\lim_{N\to\infty}\frac1{N}\sum_{n=1}^N
f(T_g^{an+b}x)=\int_Xf\,d\mu_X.
\end{equation*}
This can be rewritten as
\begin{equation}\label{rr:3-2}
\lim_{N\to\infty}\frac1{N}\sum_{n=1}^N
\1_{a\N+b}(n)f(T_g^{n}x)=\frac{1}{a}\int_Xf\,d\mu_X.
\end{equation}
If $E\subset \N$ is a finite union of infinite arithmetic
progressions then $\1_E$ can be written as a finite linear combination of functions of the form $\1_{a\N+b}$. It now follows directly from \eqref{rr:3-2} that for any such set $E$, one has
\begin{equation}\label{rr:3-3}
\lim_{N\to\infty}\frac1{N}\sum_{n=1}^N
\1_E(n)f(T_g^{n}x)=d(E)\int_Xf\,d\mu_X.
\end{equation}
Finally, since $R$ is rational, it can be approximated in the $d_B$
pseudo-metric by finite unions of infinite arithmetic progressions and so,
using \eqref{rr:3-3}, we obtain
$$
\lim_{N\to\infty}\frac1{N}\sum_{n=1}^N
\1_{R}(n)f(T_g^{n}x)=d(R)\int_Xf\,d\mu_X.
$$
\end{proof}

\begin{proof}[Proof of \cref{p5new}]
Let $\xbmt$ be an invertible measure preserving system and assume that $R\subset\N$ is rational and divisible.
Take any $A\in\cb$ with $\mu(A)>0$ and let $p_1,\ldots,p_\ell\in\Q[t]$ with $p_i(\Z)\subset\Z$, $p_i(0)=0$, $i=1,\ldots,\ell$, be arbitrary.
We will show that
\begin{equation}\label{rr20}
\lim_{N\to\infty}\frac{1}{N}\sum_{n=1}^N
\1_R(n)\varphi(n)>0,
\end{equation}
where $\varphi(n)=\mu\big(A\cap T^{-p_1(n)}A\cap\ldots\cap T^{-p_\ell(n)}A\big)$. This, in view of \eqref{eqn:ac-2.2}, suffices to conclude that $R$ is an averaging set of polynomial multiple recurrence.
The existence of the limit in \eqref{rr20}
follows immediately from \cref{prop_conv-of-erg-ave}, hence it
only remains to show its positivity.

By \cref{Cor-uniformity}, there exists $\delta>0$ such that
\begin{equation}\label{rr20-1}
\lim_{N\to\infty}\frac{1}{N}\sum_{n=1}^N \varphi(un)>\delta
\text{ for all }u\in\N.
\end{equation}
Using Theorem \ref{thm_4.1}, we can find a nilsequence $(\psi(n))$ such that \eqref{rr20} holds if and only if
\begin{equation}\label{rr21}
\lim_{N\to\infty}\frac{1}{N}\sum_{n=1}^N \1_R(n)\psi(n)>0.
\end{equation}
Moreover, since $d_B(\varphi,\psi)=0$, it follows from \eqref{rr20-1} that
\begin{equation}\label{rr20-2}
\lim_{N\to\infty}\frac{1}{N}\sum_{n=1}^N\psi(un)>\delta \text{ for all }u\in\N.
\end{equation}

By definition, every nilsequence can be uniformly approximated by basic nilsequences. For us this means that there exist
a nilpotent Lie group $G$, a uniform and discrete subgroup $\Gamma$, $x\in X=G/\Gamma$ and $f\in C(X)$ such that $\vert \psi(n)- f(T_g^n x)\vert \leq \delta/4 $ for all $n\in\N$. We can assume without loss
of generality that $T_g$ is ergodic and, since $\varphi(n)\in[0,1]$ and $d_B(\varphi,\psi)=0$, that $0\leq f \leq 1$.
It follows from \eqref{rr20-2} and the inequalities $\vert \psi(n)- f(T_g^n x)\vert \leq \delta/4$, $n\in\N$, that
\begin{equation}\label{rr20-3}
\lim_{N\to\infty}\frac{1}{N}\sum_{n=1}^N f(T_g^{un} x)>\frac{3\delta}{4},\quad\text{for all $u\in\N$.}
\end{equation}

Using \cref{prop_2.1}, we can find $u\in\N$ and a sub-nilmanifold $Y\subset X$ containing $x$
such that $(Y,\mu_Y,T_{g^u})$ is totally ergodic.
In the following, we identify $f$ with $f|_Y$.
Since $R$ is rational, it is straightforward  that
the set $R/u:=\{n\in\N: nu\in R\}$ is also rational. Thus, we can invoke
\cref{l1} and obtain
\begin{equation}\label{rr21-1}
\lim_{N\to\infty}\frac1{N}\sum_{n=1}^N
\1_{R/u}(n)f(T_{g^u}^{n}x)=d(R/u)\int_Y f\,d\mu_Y.
\end{equation}

Finally, combining \eqref{rr21-1} with \eqref{rr20-3} (together with the ergodic theorem) and the fact that $\vert \psi(un)- f(T_{g^u}^n x)\vert \leq \delta/4$ for all $n\in\N$, we obtain
\begin{eqnarray*}
\lim_{N\to\infty}\frac1{N}\sum_{n=1}^N \1_{R}(n)\psi(n)
&\geq&
\lim_{N\to\infty}\frac1{N}\sum_{n=1}^N\1_{R\cap u\N}(n)\psi(n)
\\
&=&
\frac{1}{u}\left(
\lim_{N\to\infty}\frac1{N}\sum_{n=1}^N \1_{R/u}(n)\psi(un)\right)
\\
&\geq&
\frac{1}{u}\left(
\lim_{N\to\infty}\frac1{N}\sum_{n=1}^N \1_{R/u}(n) f(T_{g^u}^{n} x)-
\frac{\delta }{4}d(R/u)\right)
\\
&\geq&  \frac{1}{u}\left(\frac{3\delta }{4}d(R/u) - \frac{\delta }{4}d(R/u)\right)~>~0.
\end{eqnarray*}
This completes the proof.
\end{proof}

We would like to pose the following question describing
one possible way of extending \cref{thm_poly-multi-rec-along-rat}
to a more general version involving several commuting measure preserving transformations.

\begin{Question}
Assume $R\subset \N$ is rational and $d(R)>0$.
Are the following equivalent?
\begin{description}
\item[$(\alpha)$~]\label{itm:pmrar-a-Q}
$R$ is divisible.
\item[$(\beta)$~]\label{itm:pmrar-b-Q}
For all probability spaces $(X,\mathcal{B},\mu)$, all $\ell\in\N$, all $\ell$-tuples of
commuting invertible measure preserving transformations $T_1,\ldots,T_\ell $ on $(X,\mathcal{B},\mu)$, all $A\in\mathcal{B}$ with $\mu(A)>0$
and for all polynomials $p_i\in\Q[t]$, $i=1,\ldots,\ell$,
with $p_i(\Z)\subset\Z$ and $p_i(0)=0$, one has
$$
\lim_{N\to\infty}\frac{1}{N}\sum_{n=1}^N
\1_R(n)\mu\Big(A\cap T_1^{-p_1(n)}A\cap\ldots\cap T_\ell^{-p_\ell(n)}A\Big)
>0.
$$
\end{description}
\end{Question}

\subsection{Inner regular sets, W-rational sets and $\sB$-free numbers}
\label{seq_applications-B-free}

The set $\sqfree$ of squarefree numbers
is rational (see \cref{c_rat-0} below) but it is not divisible, as
$\sqfree\cap p^2\N=\emptyset$ for all primes $p$.
In particular,  $\sqfree$ is not a set of recurrence.
However, as it was mentioned in \cref{sec_intro}, it follows from results obtained in \cite{MR1954690} that $\sqfree-r$ is divisible (and hence -- by virtue of \cref{thm_poly-multi-rec-along-rat} --
an averaging set of polynomial multiple recurrence) if and only if $r\in \sqfree$.

This raises the question whether every rational set can be shifted to become
divisible. In general, the answer to this question is negative.
For example, one can show that for a carefully chosen increasing sequence $a_0,a_1,a_2,\ldots\in\N$, the set $S=\N\setminus\bigcup_{n\geq 0}(a_n\N+ n)$ is rational. On the other hand, for any integer $n\geq 0$ one has $(S-n)\cap a_n\N=\emptyset$ (cf. \cite[Theorem 11.6]{MR564927} and \cite[Theorem 2.20]{MR2259058}).


We will introduce now a rather natural family of rational sets with the property that for any set in this family there is a shift that is divisible.
\begin{Def}\label{def:wrap-0-1}
We define the \emph{Weyl pseudo-metric} $d_W$ on $\{0,1\}^{\N}$ as
\begin{equation}
\label{eq:d_W-1}
d_W(x,y)=\limsup_{N\to\infty}\sup_{\ell\geq1}\frac1N\left|\{\ell\leq n\leq \ell+N: x(n)\neq y(n)\}\right|.
\end{equation}
A set $R\subset\N$ is called \emph{W-rational} if $\1_R\in\{0,1\}^{\N}$ can be approximated by periodic sequences in the $d_W$ pseudo-metric.
\end{Def}

Note that every W-rational set is a rational set.

In Subsection \ref{sec:drs.e} below we will extend the definition of the $d_W$ pseudo-metric from $\{0,1\}^{\N}$ to $\ca^{\N}$ for arbitrary finite subsets $\ca$ and we also introduce the related notion of Weyl rationally almost periodic sequences (see page \pageref{WRAP}).

\begin{Prop}\label{p:olg}
Suppose $D\subset \N$ is W-rational and $d(D)>0$. There exists a shift of $D$ which is divisible.
\end{Prop}
\begin{proof} Assume that no translation of $D$ is divisible. Hence, for each $n\geq0$ there exists $w_n\geq1$ such that if $C_n:=\{s\in \N : 
n+sw_n\in D^c\}$ then
\begin{equation}\label{olg1}
d(C_n)=1.\end{equation}
Fix $K\geq1$. Then by~\eqref{olg1}, also
$$
d\left(\{s\geq 0 : n+ sw_1\cdot\ldots\cdot w_K\in D^c\text{ for each }n=0,1,\ldots, K-1\}\right)=1.$$
It follows that for every $K\geq 1$ there exists $s\geq1$ such that $\1_D(n+sw_1\cdot\ldots\cdot w_K)=0$ for each $n=0,1\ldots,K-1$.
In other words, in the sequence $\1_D$ there appear arbitrarily long blocks of consecutive zeros. This implies that the only periodic sequence that approximates $\1_D$ in the $d_W$ pseudo-metric is the sequence $(0,0,0,\ldots)$, which contradicts $d(D)>0$.
\end{proof}

In order to give more examples of rational sets that possess shifts that are divisible,
we will now recall the notion of inner regular sets (see \cite[Definition 2.3]{MR1954690}). A subset $R\subset \N$ is called {\em inner regular} if for each $\vep>0$ there exists $m\geq 1$ such that for each $a\in\N\cup\{0\}$ the intersection $R\cap(m\Z+a)$ is either empty or has lower density $>(1-\vep)/m$.
It follows immediately that every inner regular set is rational. Also, it is shown in \cite[Lemma 2.7]{MR1954690}
that the set of squarefree numbers $\sqfree$ is inner regular.

\begin{Prop}\label{joma}Assume that $\emptyset\neq R\subset \N$ is inner regular. Then, for each $r\in R$, the set $R-r$ is divisible.
\end{Prop}
\begin{proof}
Suppose $u\geq 1$ is arbitrary.
Fix $\vep>0$ with $\vep<1/u$. We can find $m\geq 1$ so that for every $a\in\N\cup\{0\}$ the intersection
$R\cap (m\Z+a)$ is either empty or has lower density greater than $(1-\vep)/m$.
Since $r\in R$, the intersection $R\cap (m\Z+r)$ is not empty.
This means that the set $\{k\in\N:mk+r\in R\}=\{k\in\N:mk\in R-r\}$ has lower density greater than $1-\vep$.
Since $\vep<1/u$, the set $\{k\in\N:mk\in R-r\}\cap u\N$ has positive lower density. This means the set $(R-r)\cap um\N$ has positive lower density and the assertion follows.\end{proof}

We move the discussion now to sets of $\sB$-free numbers. The purpose of the remainder of this section is to prove a general form of \cref{th:b-free-sumset-combi} formulated in \cref{sec_intro}.

Given $\sB\subset \N\setminus\{1\}$, we consider its \emph{set of multiples}
$\cm_\sB:=\bigcup_{b\in\sB}b\N$ and the corresponding set of \emph{$\sB$-free numbers}
$\cf_\sB:=\N\setminus \cm_\sB$, i.e., the set of integers without a divisor in $\sB$. Without loss of generality, we can assume
that $\sB$ is {\em primitive}, that is, no $b$ divides $b'$ for distinct $b,b'\in\sB$. Indeed,
for a general set $\sB$ one can find a primitive subset
$\sB_0\subset\sB$ such that $\cm_\sB=\cm_{\sB_0}$ and
$\cf_\sB=\cf_{\sB_0}$ (cf.\ \cite[Chapter 0]{MR1414678}).
Note that if we take $\sB=\{p^2: p~\text{is prime}\}$, then the set of
$\sB$-free numbers equals the set $\sqfree$ of squarefree numbers.

Sets of $\sB$-free numbers make good candidates for rational sets.
Unfortunately, not every set of $\sB$-free numbers is a rational set,
since the density $d(\cf_\sB)$ of $\cf_\sB$
need not exist.
An example of a set $\sB$
for which the density of $\cm_\sB$ and $\cf_\sB$ does not exist was given by Besicovitch in \cite{MR1512943}.
This leads to the following definition.
\begin{Def}[cf. \cite{MR1414678}]
\label{def:bes}
 We say that $\sB$ is \emph{Besicovitch} if $d(\cm_{\sB})$ exists. (This is equivalent to the existence of $d(\cf_\sB)$.)
\end{Def}
Davenport and Erd{\H o}s proved that the \emph{logarithmic density}
$$\bdelta(\cm_\sB):=\lim_{N\to\infty}\frac{1}{\log N}\sum_{n=1}^N \frac{1}{n}\1_{\cm_\sB}(n)$$
exists for all $\sB=\{b_1,b_2,\ldots\}\subset \N\setminus\{1\}$. This, of course, implies that
the logarithmic density $\bdelta(\cf_\sB)$ exists.
For $m\geq 1$, consider the sets $\mathscr B(m)=\{b_1,b_2,\ldots,b_m\}$ and let $\cm_{\sB(m)}$ and
$\cf_{\mathscr B(m)}$ denote the corresponding set of multiples of $\mathscr{B}(m)$ and set of $\mathscr{B}(m)$-free numbers respectively.

\begin{Th}[see \cite{Davenport1936,MR0043835}]\label{daer}
For each $\sB\subset \N\setminus\{1\}$,  the logarithmic density $\bdelta(\cm_\sB)$ of $\cm_\mathscr{B}$ exists. Moreover,
$$
\bdelta(\cm_\sB)=\underline{d}(\cm_\sB)=\lim_{m\to\infty}d(\cm_{\sB(m)}).
$$
In particular, if $\sB$ is Besicovitch then $d(\cm_{\sB})=\lim_{m\to\infty}d(\cm_{\mathscr B(m)})$.
Analogous results hold for $\cf_\sB$ instead of $\cm_\sB$.
\end{Th}

From \cref{daer}, we obtain two useful corollaries.
\begin{Cor}\label{c_rat-0}
Let $\sB\subset\N\setminus\{1\}$. Then $\cm_\sB$ and $\cf_\sB$ are rational if and only if $\sB$ is Besicovitch.
\end{Cor}
\begin{proof}
Note that for any $m\geq 1$, the sequence
$\1_{\cm_{\sB(m)}}$ is periodic.
Hence if $\sB\subset\N\setminus\{1\}$ is Besicovitch, then by \cref{daer} the sequence $\1_{\cm_{\sB}}$ can be approximated in the $d_B$-pseudo-metric by $\1_{\cm_{\sB(m)}}$ as $m\to\infty$, which proves that $\cm_\sB$ is rational. An analogous argument applies to $\cf_\sB$.

On the other hand, if $\cm_\sB$ is rational then the density of $\cm_\sB$ exists and hence, by definition, the set $\sB$ is Besicovitch.
\end{proof}

In the following,
let $\abundantnumbers$ denote the set of \emph{abundant numbers},
$\perfectnumbers$ the set of \emph{perfect numbers} and
$\deficientnumbers$ the set of \emph{deficient numbers} (for definitions, see \cref{ftn:4} on page \pageref{ftn:4}).

\begin{Cor}\label{c_rat}
Let $A\subset\N\setminus\{1\}$. Suppose $A$ satisfies the following two conditions:
\begin{enumerate}[(1)~~]
\item\label{property:c_rat-0}
 $d(A)$ exists, and
 \item\label{property:c_rat-1}
 $n A\subset A$ for all $n\in\N$.
 \end{enumerate}
Then $A$ is a rational set.
In particular, the set of abundant numbers
$\abundantnumbers$, the set of deficient numbers $\deficientnumbers$ and, for any $x\in[0,1]$, the set $\Phi_x:=\{n\in\N: \frac{\tot(n)}{n}<x\}$ are rational sets.
\end{Cor}

\begin{proof}
Set $\sB:=A$. It follows from property \eqref{property:c_rat-1} that $\cm_\sB=A$. Also, $\sB$ is Besicovitch because $d(\cm_\sB)=d(A)$ exists according to property \eqref{property:c_rat-0}. Hence, in view of \cref{c_rat-0}, the set $A=\cm_\sB$ is rational.

We now turn our attention to the set of abundant numbers. First, note that
$n\abundantnumbers\subset \abundantnumbers$ for all $n\in\N$. Also, the fact that $d(\abundantnumbers)$ exists was proven by Davenport in \cite{Dav33}. Therefore $\abundantnumbers$ is rational.
Moreover, since $\N=\abundantnumbers~\dot{\cup}~\perfectnumbers~\dot{\cup}~ \deficientnumbers$ and $d(\perfectnumbers)=0$ (cf. \cite{MR0073618}), we conclude that $\deficientnumbers$ is also rational.

Finally, for any $x\in[0,1]$, the set $\Phi_x:=\{n\in\N: \frac{\tot(n)}{n}<x\}$ satisfies $n\Phi_x\subset\Phi_x$ and it was first shown in \cite{MR1544950} that $d(\Phi_x)$ exists. Hence $\Phi_x$ is rational.
\end{proof}

As mentioned above, a shift of the set of squarefree numbers, $\sqfree-r$, is an averaging set of polynomial multiple
recurrence if and only if $r\in \sqfree$.
Our next goal is to show that a similar result holds
for other sets of $\sB$-free numbers.
Note that if $r\notin \cf_\sB$, i.e.\ $r\in\cm_{\sB}$,  then
$\cf_\sB-r$  is not a set of recurrence. Indeed, if it were a set of recurrence then, by considering the cyclic rotation on $r$ points, for some $u\geq1$ we would have
$ur\in\cf_{\sB}-r$ and therefore $(u+1)r\in\cf_{\sB}$, which is a contradiction. Hence,
$r\in \cf_\sB$ is a necessary condition for $\cf_\sB-r$ to be
good for recurrence.
As for the other direction,
we have the following result.

\begin{Th}
\label{c_ave-rec-for-B-free}
Suppose $\sB\subset\N\setminus\{1\}$ is Besicovitch. Then
`almost every' self-shift of $\cf_\sB$
is an averaging set of polynomial multiple recurrence.
More precisely, there exists a set $D\subset \cf_\sB$
with $d(\cf_\sB\setminus D)=0$ such that for all $r\in\N$
the following are equivalent:
\begin{itemize}
\item
$r\in D$;
\item
$\cf_\sB-r$ is divisible;
\item
$\cf_\sB-r$ is an averaging set of polynomial multiple recurrence.
\end{itemize}
\end{Th}

Note that \cref{th:b-free-sumset-combi} is now an immediate
consequence of \cref{c_ave-rec-for-B-free}.
We give a proof of \cref{c_ave-rec-for-B-free}
at the end of this subsection.
Let us remark that in most cases one can actually take $D=\cf_\sB$.
To distinguish between sets of $\sB$-free numbers for which
$D=\cf_\sB$ and for which
$D\subsetneq \cf_\sB$, we introduce the following notions.

\begin{Def}[cf. \cite{MR1414678}]
\label{d_taut}
Let $\sB\subset\N\setminus\{1\}$.
We call $\sB$ \emph{Behrend} if $\bdelta(\cm_\sB)=1$ (this is equivalent to the existence of the density of $\cm_\sB$ with $d(\cm_\sB)=1$).
We call $\sB$ \emph{taut} if for every $b\in \mathscr{B}$,
one has $\bdelta(\cm_\sB)> \bdelta(\cm_{\sB\setminus\{b\}})$.
\end{Def}

If $\sB$ is Behrend then $D=\emptyset$, because in this case
$\cf_{\sB}$ (and each of its translations) has zero density. Behrend sets are not taut (see Lemma~\ref{beh2} below) and
it will be clear from the proof of
\cref{c_ave-rec-for-B-free} that in the statement
of \cref{c_ave-rec-for-B-free} one can take $D=\cf_\sB$
if and only if $\sB$ is taut (see \cref{cor:vb-iff} below).

The remainder of this subsection is dedicated to proving
\cref{c_ave-rec-for-B-free}. We start with a series of
lemmas.

\begin{Lemma}[Corollary 0.14 in \cite{MR1414678}]\label{beh2}
$\mathscr{A}\cup \sB$ is Behrend if and only if at least one of $\mathscr{A}$ and $\sB$ is Behrend.
In particular, Behrend sets are not taut.
\end{Lemma}

\begin{Lemma}[Corollary 0.19 in \cite{MR1414678}]\label{beh3}
$\mathscr{B}$ is taut if and only if it is primitive and does not contain a set of the form $c\mathscr{A}$, where $c\in\N$ and $\mathscr{A}\subset \N\setminus\{1\}$ is Behrend.
\end{Lemma}

\begin{Lemma}[Cf.\ the proof of Lemma 6.5 in~\cite{Ba-Ka-Ku-Le}]\label{postepy}
Let $\sC\subset \N$. For any $u\in\N$ and $a\in\N\cup\{0\}$ the logarithmic densities of $\mathcal{M}_\mathscr{C} \cap (u\N+a)$ and $\mathcal{F}_\mathscr{C} \cap (u\N+a)$ exist and satisfy
\begin{align*}
\bdelta(\mathcal{M}_\mathscr{C} \cap (u\N+a))=\un{d}(\mathcal{M}_\mathscr{C} \cap (u\N+a))&=\lim_{m\to \infty}d(\mathcal{M}_{\sC(m)}\cap (u\N+a)),\\
\bdelta(\mathcal{F}_\mathscr{C} \cap (u\N+a))=\ov{d}(\mathcal{F}_\mathscr{C} \cap (u\N+a))&=\lim_{m\to \infty}d(\mathcal{F}_{\sC(m)}\cap (u\N+a)).
\end{align*}
\end{Lemma}
\begin{proof}
The assertion concerning $\mathcal{M}_\mathscr{C} \cap (u\N+a)$ was covered in the proof of Lemma 6.5 in~\cite{Ba-Ka-Ku-Le}. The remaining part follows immediately, as $\cf_\sC=\N\setminus \cm_\sC$.
\end{proof}

\begin{Lemma}\label{l36}
Let $\mathscr{C}\subset \N\setminus\{1\}$ and let $a\in\N\cup\{0\}$. If $u\in\N$ is coprime to each element of $\mathscr{C}$ then
$$
\bdelta(\cf_\sC\cap (u\N+a))=\frac{1}{u}\cdot\bdelta(\cf_\sC).
$$
\end{Lemma}
\begin{proof}
Suppose $\sC=\{b_1,b_2,\ldots\}$ and let
$\sC(m):=\{b_1,\ldots,b_m\}$.
The assertion of the lemma is clearly equivalent to
$
\bdelta(\cm_\sC\cap(u\N+a))=\frac{1}{u}\cdot\bdelta(\cm_\sC).
$
Since $u$ is coprime to each element of $\sC$, it follows by the Chinese Remainder Theorem that for any $m\geq 1$ and $r\in\N\cup\{0\}$ there exists $r'\in\N\cup\{0\}$ such that
$$
 (\lcm(b_1,\dots ,b_m)\N+r)\cap (u\N+a)= u\cdot\lcm(b_1,\dots,b_m)\N +r'.
$$
In particular,
$$
d\big((\lcm(b_1,\dots ,b_m)\N+r)\cap (u\N+a)\big)=\frac{1}{u\cdot\lcm(b_1,\dots,b_m)}.
$$
It follows that
\begin{equation}\label{toto}
d( \cm_{\sC(m)}\cap (u\N+a))=\frac{1}{u}\cdot d(\cm_{\sC(m)})
\end{equation}
since $\cm_{\sC(m)}$ is periodic with period $\lcm(b_1,\dots,b_m)$.
Using \cref{postepy}, \eqref{toto} and~\cref{daer}, we obtain
$$
\bdelta( \cm_\sC\cap (u\N+a))=\lim_{m\to \infty}d(\mathcal{M}_{\sC(m)}\cap (u\N+a))=\lim_{m\to \infty}\frac{1}{u}\cdot d(\cm_{\sC(m)})=\frac{1}{u}\cdot \bdelta(\cm_\sC),
$$
which completes the proof.
\end{proof}

\begin{Lemma}
\label{l:divisibility-of-taut-case}
Suppose $\mathscr{C}\subset\N\setminus\{1\}$ is  taut.
If $a\in\cf_{\mathscr{C}}$ then for every $u\in\N$, one has
$\bdelta\big((\cf_{\mathscr{C}}-a)/u\big)>0$.
\end{Lemma}

\begin{proof}
Define
\[
\mathscr{C}'(a):=\Big\{\frac{b}{\gcd(b,a)}: b\in \mathscr{C}\Big\}.
\]
Notice that
\begin{equation}\label{ooo1}
\gcd(a,c)=1 \text{ for each } c\in\mathscr{C}'(a).
\end{equation}
Moreover, $\cm_{\sC'(a)}\supset \cm_{\sC}$, whence
\begin{equation}\label{ooo2}  \cf_{\sC'(a)}\subset \cf_{\sC}.\end{equation}

Since $\gcd(b,a)$ takes only finitely many values as $b\in\sC$ varies, we have
$$
\sC'(a)=\bigcup_{d\divides a}\Big\{  \frac{b}{d} : b\in\mathscr{C},\gcd(b,a)=d\Big\}.
$$
Suppose that $1\in \mathscr{C}'(a)$. Then for some $b\in \sC$, gcd$(b,a)=b$, whence $a\in \cm_{\sC}$, a contradiction. It follows that
\begin{equation}\label{ooo3}
1\notin \sC'(a).
\end{equation}
Suppose that $\sC'(a)$ is Behrend. Then, by \cref{beh2}, for some  $d_0\divides a$, the set $\mathscr{A}:=\Big\{  \frac{b}{d_0} :b\in\sC, \gcd(b,a)=d_0\Big\}$ is Behrend and we have $d_0\mathscr{A}\subset \sC$. However, this and~\eqref{ooo3}, in view of \cref{beh3}, contradict the tautness of $\sC$. Therefore, $\sC'(a)$ cannot be Behrend, i.e.\ $c:=\bdelta(\cf_{\sC'(a)})>0$. We will use this constant to prove that
$\bdelta\big((\cf_\sC-a)/u\big)>c$ for all $u\in\N$.

By \cref{postepy} and~\eqref{ooo2}, we have
\begin{multline*}
\bdelta((\cf_\sC-a)/u)=u\cdot \bdelta((\cf_\sC-a)\cap u\N)\\
\geq u\cdot \bdelta((\cf_{\sC'(a)}-a)\cap u\N)=u\cdot \bdelta(\cf_{\sC'(a)}\cap (u\N+a)).
\end{multline*}
Hence, it suffices to show that
$$
\bdelta(\cf_{\sC'(a)}\cap (u\N+a))\geq \frac1u \cdot\bdelta\big(\cf_{\mathscr{C}'(a)}\big).
$$
In order to verify this last claim, let us divide $\mathscr{C}'(a)$ into two pieces:
\begin{align*}
\mathscr{C}'_1(a,u)&:=\{c\in\mathscr{C}'(a): \gcd(u,c)>1\},\\
\mathscr{C}'_2(a,u)&:=\{c\in\mathscr{C}'(a): \gcd(u,c)=1\}.
\end{align*}
We claim that
\begin{equation}\label{ooo4}
\cf_{\sC'_1(a,u)}\cap (u\N+a)=u\N+a.
\end{equation}
Indeed,  take any $c\in \sC'_1(a,u)$. Then gcd$(u,c)>1$ and since~\eqref{ooo1} holds, gcd$(u,c)$ does not divide $a$. Hence $c\N\cap(u\N+a)=\emptyset$. It follows that
$\cm_{\sC'_1(a,u)}\cap(u\N+a)=\emptyset$ and~\eqref{ooo4} follows.

Therefore, using \eqref{ooo4} and additionally \cref{l36}, we obtain
$$
\bdelta(\cf_{\sC'(a)}\cap (u\N+a))=\bdelta\big(\cf_{\mathscr{C}_2'(a)}\cap (u\N+a)\big)=\frac{1}{u}\cdot\bdelta\big(\cf_{\mathscr{C}'_2(a)}\big)\geq \frac1u \cdot\bdelta\big(\cf_{\mathscr{C}'(a)}\big)
$$
and the result follows.
\end{proof}

Before we  present the proof of \cref{c_ave-rec-for-B-free},
one more theorem needs to be quoted.

\begin{Th}[{\cite[Theorem 4.5 and the proof of Lemma 4.11]{Ba-Ka-Ku-Le}}]\label{thm1}
Let $\sB\subset \N\setminus\{1\}$. Then there exists a taut set $\sC\subset \N\setminus\{1\}$ such that $\cf_{\sC}\subset \cf_{\sB}$ and $\bdelta(\cf_{\sC})=\bdelta(\cf_{\sB})$. Moreover,
if $\sB$ is Besicovitch, then $\sC$ is Besicovitch.
\end{Th}

\begin{proof}[Proof of \cref{c_ave-rec-for-B-free}]
Let $\sB\subset \N\setminus\{1\}$ be Besicovitch. If $\sB$ is Behrend then
$\cf_\sB$ has zero density and so no shift of
$\cf_\sB$ is divisible or good for averaging polynomial recurrence.
In this case we can put $D=\emptyset$ and we are done.
Thus, let us assume that $\sB$ is not Behrend.
In view of \cref{thm_poly-multi-rec-along-rat}, it suffices to
find a set $D\subset \cf_\sB$ with $d(\cf_\sB\setminus D)=0$ and
such that $\cf_\sB-r$ is divisible if and only if $r\in D$.
Pick $\sC\subset \N\setminus\{1\}$ taut with $\cf_{\sC}\subset \cf_{\sB}$ and $d(\cf_{\sC})=d(\cf_{\sB})$; the existence of $\sC$ is guaranteed by \cref{thm1}.

We make the claim that one can choose $D:=\cf_\sC$.
In particular, if $\sB$ is taut then one can choose $D=\cf_\sB$.

To verify this claim, we invoke \cref{l:divisibility-of-taut-case},
which tells us that
$\bdelta\big((\cf_{\mathscr{C}}-r)/u\big)>0$ if and only if $r\in\cf_\sC$.
Since $\sC$ is Besicovitch, we can replace logarithmic density
with density and conclude that $\cf_\sC-r$ is divisible if and only
if $r\in\cf_\sC$. Finally, to finish the proof, we observe that
$d(\cf_\sB\setminus \cf_\sC)=0$ and therefore
$\cf_\sB-r$ is divisible if and only if $r\in\cf_\sC$.
\end{proof}

\begin{Th}[Corollary of the proof of \cref{c_ave-rec-for-B-free}]
\label{cor:vb-iff}
In the statement of \cref{c_ave-rec-for-B-free} one has $D=\cf_\sB$ if and only if $\sB$ is taut.
\end{Th}

\begin{Remark}
Let $\sB\subset \N\setminus\{1\}$ be Besicovitch and taut (hence $d(\cf_\sB)>0$).
Here is the summary of equivalent conditions that we obtained in this section.
\begin{enumerate}[(a)]
\item
$a\in \cf_\sB$,\label{R:a}
\item
$d(\mathcal{M}_\sB\cup a\N) > d(\mathcal{M}_\sB)$,\label{R:b}
\item
$\cf_\sB-a$ is divisible,\label{R:c}
\item
$(\cf_\sB-a)\cap u\N\neq \emptyset$ for all $u\in \N$,\label{R:d}
\item
$\cf_\sB-a$ is an averaging set of polynomial multiple recurrence,\label{R:e}
\item
$\cf_\sB-a$ is an averaging set of polynomial single recurrence,\label{R:f}
\end{enumerate}
The following diagram describes the logical connections between the above statements.
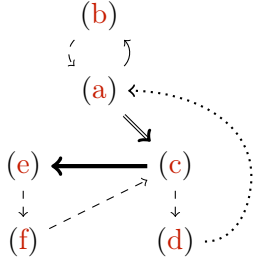
\begin{figure}[h]
\begin{center}
\begin{tikzpicture}[baseline=(current  bounding  box.center)]
\node (A) at (0,0) {$\eqref{R:a}$};
\node (B) at (0,1) {$\eqref{R:b}$};
\node (C) at (1,-1) {$\eqref{R:c}$};
\node (D) at (1,-2) {$\eqref{R:d}$};
\node (E) at (-1,-1) {$\eqref{R:e}$};
\node (F) at (-1,-2) {$\eqref{R:f}$};
\node (G) at (-1,-3) {};
\node (left) at (-2,-1.5) {};
\node (right) at (2,-1) {};
\draw[->,dashed]
	(B) edge[bend right=45] (A)
	(E) edge node[auto] {} (F)
	(C) edge node[auto] {} (D)
	(F) edge node[auto] {} (C);
\draw[->]	
	(A) edge[bend right=45] node[auto] {} (B);
\draw[->, ultra thick]
	(C) edge node[auto] {} (E);
\draw[->]	
	(A) edge[double] node[auto] {} (C);
\draw[->,dotted,thick]	
	(D.east) to[out=0,in=-90] (right) to[out=90,in=0] (A.east);
\end{tikzpicture}
\end{center}
\caption{Dashed arrows: trivial implications which hold for any $\sB\subset \N\setminus\{1\}$. The dotted arrow: this implication was explained in the paragraph before \cref{c_ave-rec-for-B-free}. The short plain arrow: implication follows from \cite{Ba-Ka-Ku-Le}. The thick arrow: this implication follows from \cref{thm_poly-multi-rec-along-rat}. The double arrow: implication proved in \cref{l:divisibility-of-taut-case}.}
\label{figure11}
\end{figure}
\end{Remark}

\section{Rational dynamical systems}
\label{sec_rds}

The purpose of \cref{sec_rds} is to give a proof of (slightly more general versions of) Theorems \ref{thm_orb-clos-of-rat} and \ref{thm_poly-multi-rec-along-disc-rat-spec}. In Subsection \ref{sec:drs.e} we define rational and W-rational subshifts and we give a variety of examples.
In Subsection \ref{sec_rational-subshifts} we extend the notion of rational subshifts to the notion of rational subshifts along increasing subsequences. Finally, in Subsections \ref{sec_inv-meas} and \ref{sec_revisit} we formulate and prove extensions of Theorems \ref{thm_orb-clos-of-rat} and \ref{thm_poly-multi-rec-along-disc-rat-spec}.

\subsection{Definition of rational subshifts. Examples}
\label{sec:drs.e}
In this section we define and give examples of symbolic dynamical systems
determined by RAP sequences. We will refer to such systems as \emph{rational subshifts}.

Consider the product space $\ca^\Z$, where $\ca$ is a finite set ({\em alphabet}). We endow $\ca$ with the discrete metric $\rho$ and $\ca^\Z$ with the product topology induced by $(\ca,\rho)$; in particular $\ca^\Z$ is compact and metrizable.
Let $S \colon \ca^\Z\to \ca^\Z$ be the left shift, i.e.\
$S((x(n))_{n\in\Z})=(y(n))_{n\in\Z}$, where $y(n)=x(n+1)$ for each $n\in\Z$.

Recall that for any closed and $S$-invariant subset $X\subset \ca^\Z$, the system $(X,S)$ is called a \emph{subshift} of $(\ca^\Z,S)$.
Recall that for $x\in \ca^\Z$ (or $x\in \ca^\N$) and $n<m$, $x[n,m]=(x(n),x(n+1),\ldots,x(m))$ is said to be a word appearing in $x$.

Given $\eta\in \ca^\N$, the set
$$X_\eta:=\{x\in\ca^\Z : (\forall n<m)(\exists k\in\N)\;\; x[n,m]=\eta[k,k+m-n-1]\}$$
is closed and $S$-invariant. It is the \emph{subshift determined by $\eta$}.

Recall, in \cref{def_rational} we introduced
the pseudo-metric of Besicovitch,
\begin{equation}
\label{eqn:B-metric-on-0-1-sec}
d_B(x,y):=\limsup_{N\to\infty}\frac1N\sum_{n=1}^N\rho(x(n),y(n))
\end{equation}
and defined {\em rationally almost periodic} (RAP) sequences, which are sequences that can be approximated in the $d_B$-pseudo-metric by periodic sequences.

\begin{Def}\label{d:rational-subshifts-0}
A subshift $(X,S)$ of $(\ca^\Z,S)$ is called
{\em rational} if there exists a RAP sequence
$\eta\in\ca^\N$ such that $X=X_\eta$.
\end{Def}



We now present some examples of rational subshifts.

\paragraph{The squarefree subshift.}
Consider the set $\sqfree$ of squarefree numbers and let
$X_{\sqfree}:=X_{\1_{\sqfree}}\subset\{0,1\}^\Z$.
The resulting topological dynamical system $(X_\sqfree,S)$ is called the \emph{squarefree subshift}
and has been studied in \cite{MR3055764,Peckner12,sarnak-lectures}.
Naturally, many combinatorial properties of $\sqfree$ are
encoded in the dynamics of $(X_{\sqfree},S)$, which
further motivates the study of this system.
This line of investigation is related to Sarnak's conjecture, see \cite{sarnak-lectures,2014arXiv1410} and \cref{secSar}.

\paragraph{$\sB$-free subshifts.}
For $\sB\subset\N\setminus\{1\}$ let $\cf_\sB$ denote the set of $\sB$-free numbers and let
$X_{\cf_\sB}:=X_{\1_{\cf_\sB}}\subset\{0,1\}^\Z$.
The system $(X_{\cf_\sB},S)$ is called the \emph{$\sB$-free subshift}. Such subhifts have been studied recently in
\cite{Ab-Le-Ru,Ba-Ka-Ku-Le,MR3356811}. If the set
$\sB$ is Besicovitch (see \cref{def:bes})
then it follows from \cref{c_rat-0} that $(X_{\cf_\sB},S)$
is a rational subshift. Note that the squarefree subshift $(X_Q,S)$ is an example of a $\sB$-free subshift.

\paragraph{Toeplitz systems.}
Following \cite{MR0255766}, a sequence $\eta \in\ca^{\N}$ is called \emph{Toeplitz},
if for each $n\geq0$ there is $p\geq1$ such that
\begin{equation}\label{toep01}\mbox{
$\eta(n)=\eta(n+s p)$ for all $s\in\N$.}\end{equation}
In this case the subshift $(X_\eta,S)$ is called a \emph{Toeplitz system}.
In \cite{MR2180227}, Downarowicz characterized Toeplitz dynamical systems as being exactly all symbolic,
minimal and almost 1-1 extensions of odometers.

If additionally
\begin{equation}\label{toep02}
\limsup_{p\to\infty}~\overline{d}\big(\{n\in\N:\eta(n)=\eta(n+s p)~\text{for all}~s\in\N\}\big)=1
\end{equation}
then the Toeplitz sequence $\eta$ is called {\em regular}.
It follows from \eqref{toep02} that any regular Toeplitz sequence is RAP (in fact, it is straightforward to check that a Toeplitz sequence is regular if and only if it is RAP).
Therefore Toeplitz systems coming from regular Toeplitz sequences are rational subshifts.

\paragraph{Weyl almost periodic sequences and Weyl rational subshifts.}
We recall the definition of the \emph{Weyl pseudo-metric} $d_W$ (see \cref{def:wrap-0-1}),
$$
d_W(x,y)=\limsup_{N\to\infty}\sup_{\ell\geq1}\frac1N\left|\{\ell\leq n\leq \ell+N: x(n)\neq y(n)\}\right|.
$$
Any $d_W$-limit of periodic sequences is called {\em Weyl rationally almost periodic sequence} (WRAP)\label{WRAP}.
The subshift $(X_\eta,S)$ determined by a WRAP sequence $\eta\in\ca^\N$ is called \emph{W-rational}.\footnote{W-rational subshifts are a special kind of \emph{Weyl almost periodic systems}, see \cite{MR951081,MR980516}.}
Each WRAP sequence is RAP and hence any W-rational subshift is a rational subshift. Note that the reverse implication does not hold. For example, the indicator function $\1_\sqfree$ of the squarefree numbers $\sqfree$ is a sequence that is RAP but nor WRAP.
It should also be mentioned that any regular Toeplitz sequence is WRAP, which can be shown easily using the definition of regular Toeplitz sequences.

\paragraph{Sequences generated by synchronized automata.}
Paperfolding sequences, which were introduced in \cref{sec_intro} (see footnote \ref{ftn:paperfolding}), provide examples of rational sequences generated by so-called synchronized automata.  

Let $k\in\N$, let $\alphB:=\{0,1,\ldots,k-1\}$, let $\ca$ be a finite alphabet, let $Q:=\{q_0,\ldots,q_r\}$ be a finite set, let $\tau:Q\to\ca$ and let $\delta\colon Q\times \alphB \to Q$.
The quintuple $\Automaton=(Q,\alphB,\delta,q_0,\tau)$ is called a \emph{complete deterministic finite automaton} with \emph{set of states} $Q$, \emph{input alphabet} $\alphB$, \emph{output alphabet} $\ca$, \emph{transition function} $\delta$, \emph{initial state} $q_0$ and \emph{output mapping} $\tau$.
Let $\alphB^*$ denote the collection of all finite words in letters from $\alphB$. There is a natural way of extending $\delta\colon Q\times \alphB \to Q$ to $\delta\colon Q\times \alphB^* \to Q$: for the empty word $\epsilon\in\alphB^*$ we define $\delta(q,\epsilon):=q$, $q\in Q$, and for a non-empty word $w=w_1\ldots w_n \in\alphB^*$ we define recursively $\delta(q,w_1\ldots w_n):=\delta(\delta(q,w_1\ldots w_{n-1}),w_n)$, $q\in Q$. 
This way, we can associate to each word $w\in\alphB^*$ an element $a\in\ca$ via $a=\tau(\delta(q_0,w))$.
For more details on deterministic finite automata see \cite[Section 4.1]{MR1997038}.


Given $n\in \N$, we consider its expansion in base $k$, i.e.
$$
n=\sum_{j\geq0}\vep_jk^j, \mbox{where $\vep_j\in\alphB$, $j\geq0$}.
$$
In this representation $\vep_j=0$ for all but finitely many $j\geq 0$. Let $j_n$ be the largest index such that $\vep_{j_n}\neq0$.
We then set $[n]_k:=(\vep_{j_n},\vep_{j_n-1},\ldots,\vep_0)$; note that $[n]_k\in\alphB^*$ for all $n\in\N$.
\begin{Def}
Following \cite{MR2590264}, we say that a sequence $a\in\ca^\N$ is {\em automatic} if there exists a complete deterministic finite automaton $\Automaton$ as above such that $a(n)=\tau(\delta(q_0,[n]_k))$ for all $n\in\N$.
\end{Def}

\begin{Def}\label{def_synchro}
Following \cite{MR2567477}, Part~4, an automaton $\Automaton=(Q,\alphB,\delta,q_0,\tau)$ is called {\em synchronized} if there exists a word $w\in \alphB^\ast$ such that $\delta(q,w)=\delta(q_0,w)$ for all $q\in Q$. In this case, the word $w$ is called a \emph{synchronizing word}.
\end{Def}

While not every automatic sequence is RAP\footnote{The Thue-Morse sequence (which was independently discovered by Thue \cite{Thue1,Thue2} and Morse \cite{MR1507944,MR0000745}) is known to be automatic (see \cite{MR1997038}), but it is not RAP as it is a generic point for a measure such that the corresponding dynamical system has no purely discrete spectrum\cite{Keane}, see \cref{thm_orb-clos-of-rat}.}, any automatic sequence coming from a synchronized automaton is not only RAP but also WRAP.
This result, which we state as a proposition below, has been shown implicitly in \cite{De-Dr-Mu}. For the sake of completeness, we give a proof of it in \cref{SAS}.

\begin{Prop}[see \cite{De-Dr-Mu}] \label{synch-ratio}
Each automatic sequence generated by a synchronized automaton is WRAP.
\end{Prop}


As RAP sequences can be approximated by periodic sequences, one may be tempted to believe that the dynamics of rational subshifts is similar to the dynamics of certain low-complexity systems such as translations on compact groups.
However, rational subshifts exhibit a wide variety of dynamical properties. For instance:
\begin{itemize}
\item
Rational subshifts can have positive topological entropy (for the definition of topological entropy see for instance \cite[Section 6.3]{MR833286}); for example, rational $\mathscr{B}$-free subshifts can have positive topological entropy, see \cite{Ab-Le-Ru,Ba-Ka-Ku-Le,Peckner12}. Moreover, they can have many invariant measures \cite{MR3356811}.
\item
Rational subshifts can be proximal\footnote{A dynamical system $(Y,T)$ is {\em proximal} if any pair $y,z\in Y$ is proximal, i.e.\ there exists a sequence $(n_k)$ such that $d(T^{n_k}y,T^{n_k}z)\to0$. 
In particular, in any such system $y$ and $Ty$ are proximal and it follows that proximal systems have exactly one fixed point to which all other points are proximal.}; in fact, the
squarefree subshift is an example of a proximal rational subshift \cite{sarnak-lectures}.
\item
Rational subshifts can be topologically mixing (see \cref{topmixing} for a proof).
\item
If $(X,S)$ is a rational subshift of positive entropy, then there is 
$y\in X$ which is not RAP. As a matter of fact, no generic point for a measure of positive entropy is RAP (this follows from \cref{t:Benji} below).
\end{itemize}


On the other hand, W-rational subshifts have much more regular properties than general rational subshifts. This is illustrated by the following two propositions.

\begin{Prop}[\cite{MR951081,MR0255766}]\label{prop-3.5}
If $x$ is WRAP then $(X_x,S)$ is uniquely ergodic and has zero topological entropy.\footnote{It is shown in \cite{MR951081} that for subshifts on finite alphabet the entropy function is $d_W$-continuous.} In particular, any non-trivial $W$-rational subshift is not proximal.
\end{Prop}

\begin{Prop}[cf. {\cite[Lemma 4]{MR980516}}]\label{l:flo}
Let $x\in \ca^\N$ be WRAP and suppose
$z\in X_x$. Let $z|_{\N}\in\ca^\N$ denote the restriction of $z\in\ca^\Z$
to $\N$. Then $z|_\N$ is also WRAP.
\end{Prop}

\begin{proof}
Let $z\in X_x$ and let $\varepsilon>0$.
Pick any periodic sequence $y\in \ca^\N$ with
$d_W(x,y)\leq \varepsilon/2$. Let $M$ denote the period of $y$. Let $\ell_K$ be such that $x(n+\ell_K)=z(n)$ for $n\leq K$. We can assume without loss of generality that there exists $0\leq i_0<M$
with $\ell_{K}\equiv i_0 \bmod M$ for all $K\in\N$.

Let $N_0$ be such that for all $N\geq N_0$, we have
\begin{equation}
\label{eqn:2}
\sup_{\ell\geq 0}\frac1N \sum_{n=1}^N\rho(x(n+\ell ),y(n+\ell ))\leq \varepsilon.
\end{equation}
Fix $N\geq N_0$ and $\ell\in\N$, and let $K\geq N+\ell$. Then, by the choice of $\ell_K$, we have
\begin{equation}\label{EQ1}
z(n+\ell )=x(n+\ell +\ell _{K})\text{ for all $n\leq N$}.
\end{equation}
Moreover, since $y$ is $M$-periodic and since $\ell_{K}\equiv i_0\bmod M$,
\begin{equation}\label{EQ2}
S^{i_0}y(n+\ell )=y(n+\ell +\ell _{K}) \text{ for each }n\in\N.
\end{equation}
Using \eqref{EQ1}, \eqref{EQ2} and \eqref{eqn:2}, we conclude that
\begin{multline*}
\frac1N \sum_{n=1}^N\rho(z(n+\ell ),S^{i_0}y(n+\ell ))=\frac1N \sum_{n=1}^{N}\rho(x(n+\ell +\ell _{K}),S^{i_0}y(n+\ell ))\\
=\frac1N \sum_{n=1}^{N}\rho(x(n+\ell +\ell _{K}),y(n+\ell +\ell _{K}))\leq\varepsilon
\end{multline*}
and the result follows.
\end{proof}

\subsection{Rationality along subsequences}
\label{sec_rational-subshifts}

In this short subsection we introduce and discuss a useful generalization of RAP sequences.

\begin{Def}\label{d:rational-subshifts}
Let $\Nk$ be an increasing sequence of natural numbers and
assume that $\ca$ is a finite set endowed with the discrete metric $\rho$.
\begin{itemize}
\item
For $x,y\in\ca^{\N}$, we define
(cf.\ \eqref{eqn:B-metric-on-0-1},
\eqref{eqn:B-metric-on-C} and \eqref{eqn:B-metric-on-0-1-sec}),
$$
d_B^{(N_k)}(x,y):=\limsup_{k\to\infty}\frac{1}{N_k}\sum_{n=1}^{N_k}\rho(x(n),y(n)).
$$
We say that $x\in \ca^{\N}$ is
\emph{rationally almost periodic along $\Nk$} (RAP along $\Nk$) if
$x$ is a $d_B^{(N_k)}$-limit of periodic sequences. A subset $R\subset\N$ is called {\em rational along} $\Nk$, if $\1_R$ is RAP along $\Nk$.
\item
Let $(X,S)$ be a subshift of $(\ca^\Z,S)$.
We call the topological dynamical system $(X,S)$
\emph{rational along $\Nk$} if $X=X_\eta$ for some $\eta\in\ca^\N$  that is RAP along
$\Nk$.
If $N_k=k$ for all $k$ then, clearly,  $(X,S)$ a
rational subshift (see \cref{d:rational-subshifts-0}).
\end{itemize}
\end{Def}

\begin{Remark}\label{N1}
Let $R\subset \N$ and suppose that $R$ is rational along $\Nk$. For $u\geq 1$, let $R/u:=\{n\in\N : nu\in R\}$. Then $R/u$ is rational along $(N_k/u)_{k\geq 1}$ for any $u\geq 1$.
\end{Remark}
Clearly, any sequence that is RAP is also RAP along $\Nk$.
The following example shows that, in general, the converse does not hold. 

\begin{Example}\label{ex1}
Given an increasing sequence $(b_k)\subset\N$, we define
$$
x=x^{(b_k)}:=\underbrace{010101\ldots01}_{b_1}\underbrace{101010\ldots10}_{b_2}\underbrace{010101\ldots01}_{b_3}\ldots
$$
Note that if $(b_k)_{k\in\N}$ is increasing sufficiently fast then $x$ is RAP along $(N_{2k})_{k\geq 1}$, where $N_k:=b_1+\ldots +b_k$, $k\geq 1$. We will now show that for no choice of increasing $(b_k)$, the sequence $x$ is RAP. Suppose that there is $(b_k)_{k\in\N}$ that yields a RAP sequence $x$. Then the sequence
$$
x':=\underbrace{0\ldots0}_{b_1/2}\underbrace{1\ldots1}_{b_2/2}\underbrace{0\ldots0}_{b_3/2}\ldots
$$
must also be RAP and therefore the density $d(A)$ of
$A:=\{n\in\N : x'(n)=1\}$ exists. Notice that
$$
|\{1\leq n\leq N_{2k}/2 : x'(n)=1\}|> 1/2 \cdot N_{2k}/2,
$$
so, in particular $d(A)\geq 1/2>0$. Since $(b_k)$ is increasing, $d(A)=1/2$. Let $0<\vep<1/4$ and let $y$ be a periodic sequence with $d_B(x',y)<\vep$. Then
$|\{1\leq n\leq R : y(n)=1\}|=(\frac14\pm \vep)R$, where $R$ is a period of $y$. We will now estimate $d_B(x',y)$ from below. Fix $1\leq i\leq R$. If $y(i)=0$ then
$$
\frac{1}{N} |\{1\leq n\leq N : x'(Rn+i)\neq y(i)\}|=\frac{1}{N} |\{1\leq n\leq N : x'(Rn+i)=1\}| \to \frac1R\cdot\frac12 ,
$$
and similarly, if $y(i)=1$ then
$$
\frac{1}{N} |\{1\leq n\leq N : x'(Rn+i)\neq y(i)\}|=\frac{1}{N} |\{1\leq n\leq N : x'(Rn+i)=0\}| \to \frac1R\cdot\frac12.
$$
It follows that
\begin{multline*}
\vep\geq d_B(x',y)= \limsup_{N\to\infty}\frac{1}{N}|\{1\leq n\leq N : x'(n)\neq y(n)\}|\\
=\limsup_{N\to\infty}\frac{1}{N} \left|\bigcup_{0\leq i<R} \{1\leq n\leq N : x'(Rn+i)\neq y(i)\}\right|\\
\geq\liminf_{N\to\infty}\frac{1}{N} \left|\bigcup_{0\leq i<R} \{1\leq n\leq N : x'(Rn+i)\neq y(i)\}\right|\\
\geq \sum_{0\leq i<R}\liminf_{N\to\infty}\frac{1}{N} |\{1\leq n\leq N : x'(Rn+i)\neq y(i)\}|\\
\geq (\frac12-\vep)R\cdot \frac{1/2}{R}+(\frac12-\vep)R\cdot\frac{1/2}{R}=\frac12-\vep.
\end{multline*}
This yields a contradiction with our choice of $\vep$.
\end{Example}
\begin{Remark}
\label{rem:131}
Notice that for $x$ as in \cref{ex1} the density of the set $\{n\in \N : x(n)=1\}$ equals $\frac{1}{2}$.
This shows that there are sets that have density, are not rational but are rational along some subsequence. In fact, one can show that $x$ is a generic point for the measure $\frac12(\delta_{101010\ldots}+\delta_{010101\ldots})$.
\end{Remark}

\begin{Example}\label{ex1-continued}
It was shown in Subsection \ref{seq_applications-B-free} (see \cref{c_rat-0}) that for any set $\sB\subset\N\setminus\{1\}$ the set of $\sB$-free numbers $\cf_\sB$ is rational if and only if $\sB$ is Besicovitch. In particular, if $\sB$ is not Besicovitch then $\cf_\sB$ is not rational. However, if $\Nk$ is an increasing sequence such that $\underline{d}(\cf_\sB)=\lim_{k\to\infty}\frac{|\cf_{\sB}\cap \{1,\ldots,\N_k\}|}{N_k}$, then it follows from \cref{daer} that $\cf_{\sB}$ is rational along $\Nk$.
\end{Example}

\subsection{Generalizing \cref{thm_orb-clos-of-rat}}
\label{sec_inv-meas}


Let $\ca$ be a finite alphabet and let $(X,S)$ be a subshift of the full shift $(\ca^\Z,S)$.
Denote by $\mathcal{P}(X,S)$  the set of all $S$-invariant Borel probability measures on $X$ and by $\mathcal{P}^e(X,S)$ the subset of $\mathcal{P}(X,S)$ of ergodic measures.
By the Krylov-Bogolyubov theorem \cite{MR1503326}, $\mathcal{P}(X,S)$ is non-empty.

Given an increasing sequence $(N_k)_{k\geq 1}$ of natural numbers, we say that $x\in \ca^\N$ is \emph{quasi-generic for} $\mu\in\mathcal{P}(X,S)$ \emph{along} $(N_k)_{k\geq 1}$ if
$$
\lim_{k\to\infty}\frac{1}{N_k}\sum_{n=0}^{N_k-1}f(S^n\tilde x)= \int_X f\, d\mu
$$
for all continuous functions $f\in C(\ca^\Z)$, where
$\tilde x\in\ca^\Z$ is any two-sided sequence that \emph{extends} $x$, i.e., $\tilde x(n)=x(n)$ for all $n\in\N$. 
If $N_k=k$ for all $k$ then $x$ is \emph{generic} for $\mu$ (as was defined in \cref{sec_intro}).
Note that if $x$ is quasi-generic for $\mu$ then
$\mu\in \mathcal{P}(X_x,S)$.

The goal of this section is to show that for any RAP sequence $x\in\ca^\N$ there exists a measure $\mu$ for which $x$ is generic and the corresponding measure preserving system $(X_x,\mu,S)$ has rational discrete spectrum. As a matter of fact, we will prove a slightly more general theorem, see \cref{t:Benji} below.

A partly related result was proved by Iwanik in \cite[Theorem 2]{MR980516}.
He showed that any sequence $x\in\ca^\N$ in the class of so-called \emph{Weyl almost periodic sequences} (which is a class that contains all WRAP sequences) is generic for a measure $\mu$ and the measure preserving system $(X_x,\mu,S)$ has discrete spectrum (but not necessarily rational discrete spectrum). We refer the reader to \cite{MR980516} for the definitions of Weyl almost periodic sequences.
Our variant of Iwanik's result regarding RAP sequences seems to be new.
The authors would like to thank B.\ Weiss for fruitful discussions on the subject.

\begin{Th}\label{t:Benji}
Let $x\in \ca^{\N}$ be RAP along $\Nk$. There exists a measure $\mu$ for which $x$ is quasi-generic along $\Nk$ and the corresponding measure preserving system $(X_x,\mu,S)$ is ergodic and has rational discrete spectrum.
\end{Th}

Note that \cref{thm_orb-clos-of-rat} follows immediately
from \cref{t:Benji}, if we put $N_k=k$ for all $k\geq 1$.

\begin{Remark}
Some special cases of \cref{t:Benji} are known.
It is shown
in \cite{MR3055764,sarnak-lectures}
-- in the context of the squarefree subshift $(X_\sqfree,S)$ -- that
the characteristic function of the set of squarefree numbers
is generic for a measure which yields a measure preserving system with rational discrete spectrum.
This result has been generalized to arbitrary
sets of $\sB$-free numbers $\cf_\sB$
(see \cite{Ab-Le-Ru,Ba-Ka-Ku-Le}).
Also, for regular Toeplitz sequences it is shown in \cite[Section 4]{MR0255766} that the corresponding Toeplitz system has rational discrete spectrum (with respect to the unique invariant measure). 
\end{Remark}

The proof of \cref{t:Benji} hinges on four lemmas.
The first two lemmas, namely \cref{l:asympdenswords}
and \cref{qge2}, are needed to prove that any RAP
sequence is generic for an invariant probability measure.
\cref{p:Benji1} shows that the measure obtained this way is ergodic. Finally, \cref{l:benmar1} proves that the corresponding system has rational discrete spectrum.
We conclude this subsection by combining these four lemmas
to give a proof of \cref{t:Benji}.

Given $(\alpha_1,\ldots,\alpha_\ell)\in\ca^\ell$ and $n_1<\ldots<n_\ell$, we define the corresponding cylinder set
$$
C=C^{\alpha_1,\ldots,\alpha_\ell}_{n_1,\ldots,n_\ell}:=\{x\in\ca^{\Z} : x(n_j)=\alpha_j\text{ for }j=1,\ldots,\ell\}.$$
Each cylinder set is a clopen subset of $\ca^\Z$ and the family of cylinder sets forms a basis of topology on $\ca^{\Z}$. More generally, for every subshift $(X,S)$
a basis of topology is given by the clopen sets of the from $C^{\alpha_1,\ldots,\alpha_\ell}_{n_1,\ldots,n_\ell}\cap X$, $(\alpha_1,\ldots,\alpha_\ell)\in\ca^\ell$ and $n_1<\ldots<n_\ell\in\Z$.

\begin{Lemma}\label{l:asympdenswords}
Let $x,y\in \ca^{\N}$ and let $C=C^{\alpha_1,\ldots,\alpha_\ell}_{n_1,\ldots,n_\ell}$, where
$n_1,\ldots,n_\ell\in\Z$ and $\alpha_1,\ldots,\alpha_\ell \in \ca$. Then, for any two-sided sequences $\tilde{x},\tilde{y}\in\ca^\Z$ extending $x$ and $y$, we have
\begin{eqnarray}
\label{eqn:12-Nk}
\limsup_{k\to\infty}\frac{1}{N_k}
\sum_{n=1}^{N_k} \big|\1_C(S^n \tilde x)-\1_C(S^n\tilde y)\big|
&\leq & \ell d_B^{(N_k)}(x,y).
\end{eqnarray}
\end{Lemma}

\begin{proof}
Let $\ca^\Delta:=\{(a,a):a\in\ca\}$.
%
For any $k\geq1$,
$$
\frac1{N_k}\big|\{1\leq n\leq N_k :  x(n)\neq y(n)\}\big|= \frac1{N_k}\sum_{n=1}^{N_k}\rho(x(n),y(n)).$$
It follows that
\begin{equation}
\label{eq:disc-markov}
\limsup_{k\to\infty}\frac1{N_k}\sum_{n=1}^{N_k} \1_{\ca^2\setminus\ca^\Delta}(x(n),y(n))\leq d_B^{(N_k)}(x,y).
\end{equation}
In view of \eqref{eq:disc-markov}, the left hand side of \eqref{eqn:12-Nk} can be estimated by
\begin{eqnarray*}
\limsup_{k\to\infty}\frac1{N_k}\sum_{n=1}^{N_k} \big|\1_A(S^n \tilde{x})-\1_A(S^n\tilde{y})\big|
&\leq & \limsup_{k\to\infty}\frac1{N_k}\sum_{n=1}^{N_k} \sum_{i=1}^\ell
\1_{\ca^2\setminus\ca^\Delta}\Big(x(n+n_i),y(n+n_i)\Big)\\
&\leq &  \sum_{i=1}^\ell\limsup_{k\to\infty}\frac1{N_k}\sum_{n=1}^{N_k}  \1_{\ca^2\setminus\ca^\Delta}\Big(x(n+n_i),y(n+n_i)\Big)\\
&\leq &  \sum_{i=1}^\ell d_B^{(N_k)}(x,y).
\end{eqnarray*}
This completes the proof of \eqref{eqn:12-Nk}.
\end{proof}

\begin{Lemma}\label{qge2}
Let $x,x_n\in\ca^\N$, $n\in\N$,
and suppose $\lim_{n\to\infty}d_B^{(N_k)}(x_n,x)=0$.
If $x_n$ is quasi-generic along $\Nk$ for all $n\in\N$, then
$x$ is quasi-generic along $\Nk$.
\end{Lemma}

\begin{proof}
To show that $x$ is quasi-generic, it suffices to show
that for all continuous functions $f\colon\ca^\Z\to\R$ the limit
\begin{equation}
\label{eq:denart-1}
\lim_{k\to\infty}\frac{1}{N_k}\sum_{n=1}^{N_k}f(S^n\tilde x)
\end{equation}
exists, where $\tilde{x}\in\ca^Z$ is any two-sided sequence extending $x\in\ca^\N$.
Note that any continuous function $f\colon\ca^\Z\to\R$ can be
approximated uniformly by linear combinations of characteristic functions of cylinder sets $A=C^{\alpha_1,\ldots,\alpha_\ell}_{n_1,\ldots,n_\ell}$.
Hence, we can assume without loss of generality that the function $f$
in~\eqref{eq:denart-1} is given by the indicator function of such a cylinder set.

Let $\vep>0$ be arbitrary and pick $m\geq 1$
such that $d_B^{(N_k)}(x_m,x)<\vep$. Let
$\tilde{x}_m\in\ca^\Z$ be any two-sided sequence extending $x_m\in\ca^\N$.
Then, from \cref{l:asympdenswords}, we deduce that the difference
$$
\limsup_{k\to\infty}\frac{1}{N_k}\sum_{n=1}^{N_k}f(S^n\tilde x)-\liminf_{k\to\infty}\frac{1}{N_k}\sum_{n=1}^{N_k}f(S^n\tilde x)
$$
is bounded from above by
$$
\limsup_{k\to\infty}\frac{1}{N_k}\sum_{n=1}^{N_k}f(S^n \tilde x_m)-\liminf_{k\to\infty}\frac{1}{N_k}\sum_{n=1}^{N_k}f(S^n \tilde x_m)
+ 2\ell d_B^{(N_k)}(x,x_m).
$$
But $x_m$ is quasi-generic along $\Nk$ and  therefore
$$
\limsup_{k\to\infty}\frac{1}{N_k}\sum_{n=1}^{N_k}f(S^n\tilde x_m)=\liminf_{k\to\infty}\frac{1}{N_k}\sum_{n=1}^{N_k}f(S^n\tilde x_m).
$$
This implies that
$$
\limsup_{k\to\infty}\frac{1}{N_k}\sum_{n=1}^{N_k}f(S^n\tilde x)-\liminf_{k\to\infty}\frac{1}{N_k}\sum_{n=1}^{N_k}f(S^n \tilde x)\leq 2\ell \epsilon.
$$
Since $\vep>0$ was arbitrary, this shows that the limit in \eqref{eq:denart-1} exists.
\end{proof}

Next, we need a slight generalization of
\cite[Proposition 4.6]{MR1727510}.
We include the proof for the convenience of the reader.
\begin{Lemma} [see {\cite[Proposition 4.6]{MR1727510}} for the case $N_k=k$]
\label{p:Benji1}
Suppose $x,x_j\in\ca^\N$, $j\in\N$,
and $\lim_{j\to\infty}d_B^{(N_k)}(x_j,x)=0$.
If each $x_j$ is quasi-generic along $\Nk$ for an ergodic measure, then $x$ is quasi-generic along $\Nk$ for an ergodic measure.
\end{Lemma}

\begin{proof}
By \cref{qge2}, we know that $x$ is quasi-generic along $\Nk$
for an invariant measure $\mu\in \mathcal{P}(\ca^\Z,S)$.
It only remains to show that $\mu$ is ergodic. It suffices
to show that for all
$f,g\in L^\infty(\ca^\Z,\mu)$, one has
\begin{equation}
\label{eqn:Benj-eq1}
\lim_{N\to\infty}\frac1N\sum_{n=1}^N \int_{\ca^\Z} S^n f \cdot g\,d\mu =
\int_{\ca^\Z} f\, d\mu \int_{\ca^\Z} g\,d\mu.
\end{equation}
Similarly to the argument in the proof of \cref{qge2},
it suffices to prove \eqref{eqn:Benj-eq1}
for the special case where $f$ and $g$
are the indicator functions of cylinder sets. In other words, we can assume without loss of generality that $f=\1_A$ and $g=\1_B$, where $A=C^{\alpha_1,\ldots,\alpha_\ell}_{n_1,\ldots,n_\ell}$ and $B=C^{\beta_1,\ldots,\beta_r}_{m_1,\ldots,m_r}$.

Fix $\vep>0$. Let
$$
C_n:=S^{-n}A\cap B=\{z\in \ca^\Z:z(n+n_i)=\alpha_i,z(m_j)=\beta_j\text{ for }i=1,\ldots,\ell,j=1,\ldots,r\}.
$$
Then \eqref{eqn:Benj-eq1} can be rewritten as
\begin{equation}
\label{eqn:Benj-eq2}
\lim_{N\to\infty}\frac1N\sum_{n=1}^N \mu(C_n)=
\mu(A) \mu(B).
\end{equation}
The sequence $x$ is quasi-generic for $\mu$ along $(N_k)$, so
$$
\mu(C_n)=\lim_{k\to\infty}\frac1{N_k}\sum_{i=1}^{N_k} \1_{C_n}(S^i\tilde x),
$$
for all $\tilde{x}\in\ca^\Z$ that extend $x\in\ca^\N$.
Fix $\vep>0$. In view of \cref{l:asympdenswords}, for  $j$ sufficiently large, we obtain
\begin{equation}
\label{eqn:Benj-eq3}
\limsup_{k\to\infty}\frac1{N_k}\sum_{i=1}^{N_k}
|\1_{C_n}(S^i\tilde x)-\1_{C_n}(S^i \tilde x_j)|\leq \vep,
\end{equation}
for all $\tilde{x}_j\in\ca^\Z$ that extend $x_j\in\ca^\N$.
Here, it is important that $\epsilon$ appearing in~\eqref{eqn:Benj-eq3} does not depend on $n$.
Denote by $\mu_j$ the measure for which the sequence $x_j$ is quasi-generic along $\Nk$.
It then follows from \eqref{eqn:Benj-eq3}
that $|\mu(C_n)-\mu_j(C_n)|\leq \vep$ for all $n$.
A similar argument shows that if $j$ is sufficiently large then
$|\mu(A)-\mu_j(A)|\leq \vep$
and $|\mu(B)-\mu_j(B)|\leq \vep$.

Now, since $\mu_j$ is ergodic,
we have that \eqref{eqn:Benj-eq2} holds with $\mu$
replaced by $\mu_j$. Then, using the triangle inequality and the
fact that $\mu_j(C_n)$, $\mu_j(A)$ and $\mu_j(B)$ are $\vep$-close
to $\mu(C_n)$, $\mu(A)$ and $\mu(B)$, respectively, we obtain that
$$
\limsup_{N\to\infty}\left|\frac1N\sum_{n=1}^N \mu(C_n)-\mu(A)\mu(B)\right|
\leq 3\vep.
$$
Since $\vep$ was chosen arbitrarily, this proves \eqref{eqn:Benj-eq2}.
\end{proof}

For the statement of the next lemma, we need to recall
the definition of a joining.
Let $\xbmt$ and $(Y,\mathcal{C},\nu,R)$ be ergodic measure preserving systems and let $\lambda$ be a $(T\times R)$-invariant measure on $(X\times Y,\mathcal{B}\otimes \mathcal{C})$. We say that $(X\times Y,\mathcal{B}\otimes \mathcal{C},\lambda, T\times R)$ is a \emph{joining} of $\xbmt$ and $(Y,\mathcal{C},\nu,R)$ if
$\lambda|_X=\mu$ and $\lambda|_Y=\nu$~\cite{MR0213508}. We will write $J((X,\cb,\mu,T),(Y,\mathcal{C},\nu,R))$ for the set of all joinings of $\xbmt$ and $(Y,\mathcal{C},\nu,R)$. The subset of ergodic joinings will be denoted by $J^e((X,\cb,\mu,T),(Y,\mathcal{C},\nu,R))$.
The definition of a joining extends naturally to any finite or countably infinite family of systems.

\begin{Lemma} \label{l:benmar1} Assume that $x,x^{(n)}\in \ca^{\N}$ are quasi-generic along $\Nk$ for ergodic measures $\mu$, $\mu_n$ ($n\geq1$), respectively. Assume moreover that $x^{(n)}\to x$ in $d_B^{(N_k)}$. Then $(X_x,\mu,S)$ is a factor of $\left((\ca^\Z)^{\times\infty},\nu, S^{\times\infty}\right)$ for some $\nu\in J^e\big((\ca^\Z,\mu_1,S),(\ca^\Z,\mu_2,S),\ldots\big)$.
\end{Lemma}
\begin{proof}
 Consider
$$
z:=(x,x^{(1)},x^{(2)},\ldots)\in \ca^{\N}\times (\ca^\N)^{\times \infty}.
$$
Then $z$ is quasi-generic along a subsequence of $\Nk$ for an invariant measure $\ov{\nu}$, i.e.\ for an increasing sequence $(k_\ell)_{\ell\geq 1}$, we have
\begin{equation}\label{benmar1}
\int_{\ca^{\Z}\times (\ca^\Z)^{\times\infty}}f\ d\ov{\nu}=\lim_{\ell\to\infty} \frac1{N_{k_\ell}}\sum_{n=1}^{N_{k_\ell}}f((S\times(S^{\times\infty}))^n(\tilde{z}))
\end{equation}
for all $f\in C(\ca^{\Z}\times (\ca^\Z)^{\times \infty})$ and all $\tilde{z}\in \ca^{\Z}\times (\ca^\Z)^{\times \infty}$ that extend $z$.
Using the assumption of quasi-genericity along $(N_k)$, we have
$$
\ov{\nu} \in J\big((\ca^\Z,\mu,S),(\ca^\Z,\mu_1,S),(\ca^\Z,\mu_2,S),\ldots\big).
$$
We use $\cb(\ca^\Z)$ and $\cb((\ca^\Z)^{\times \infty})$
to denote the Borel $\sigma$-algebra on $\ca^\Z$
and $(\ca^\Z)^{\times \infty}$, respectively.
We claim now that $(\ca^{\Z},\mu,S)$ is a factor of $((\ca^\Z)^{\times \infty},\ov{\nu}|_{(\ca^\Z)^{\times\infty}},S^{\times\infty})$, i.e., that
up to $\ov{\nu}$-measure zero sets
the $\sigma$-algebra $\cb(\ca^{\Z})\otimes\{\emptyset,(\ca^\Z)^{\times \infty}\}$ is contained
in the $\sigma$-algebra $\{\emptyset,\ca^{\Z}\}\ot\cb((\ca^\Z)^{\times \infty})$.
Notice first that it is enough to show that for each
$\alpha\in \ca$ , we have
\begin{equation}\label{benmar2}
C^\alpha_0=\{u\in \ca^{\Z}: u(0)=\alpha\}\in \{\emptyset,\ca^{\Z}\}\ot\cb((\ca^\Z)^{\times \infty})\bmod\ov{\nu},
\end{equation}
as $\{C_0^{\alpha} :  \alpha\in \ca\}$
is a generating partition.
To obtain~\eqref{benmar2}, we note first that for each $n\geq1$, we have
\begin{multline}
\begin{split}
\label{benmar3}
\ov{\nu}\Big(\big(C_0^{\alpha}\times (\ca^\Z)^{\times \infty}\big)\triangle \big(\ca^{\Z}\times
(\ca^{\Z}\times\ldots \times \ca^{\Z}\times \underbrace{C_0^{\alpha}}_{n\text{-th position}}\times \ca^{\Z}\times\ldots)\big)\Big)\\
\leq d_B^{(N_k)}(x,x^{(n)}).
\end{split}
\end{multline}
Indeed, if $\pi_{0,n}\colon\ca^{\Z}\times (\ca^\Z)^{\times \infty}\to
\ca^{\Z}\times\ca^{\Z}$ denotes the projection
$$
\pi_{0,n}\big(y,y^{(1)},\ldots,y^{(n-1)},y^{(n)},y^{(n+1)},\ldots\big):= (y,y^{(n)})
$$
and $\ov{\nu}_{0,n}$ denotes the push-forward
of $\ov{\nu}$ under $\pi_{0,n}$, then we obtain
\begin{align*}
\ov{\nu}&((C_0^{\alpha}\times (\ca^\Z)^{\times \infty})\triangle (\ca^{\Z}\times (\ca^{\Z}\times\ldots \times \ca^{\Z}\times \underbrace{C_0^{\alpha}}_n\times \ca^{\Z}\times\ldots)))\\
&=\ov{\nu}_{0,n}((C_0^{\alpha}\times \ca^{\Z})\triangle (\ca^{\Z}\times C_0^{\alpha}))=\ov{\nu}_{0,n}(C_0^{\alpha}\times (C_0^{\alpha})^c)+\ov{\nu}_{0,n}((C_0^{\alpha})^c\times C_0^{\alpha})\\
&=\lim_{\ell\to \infty}\frac1{N_{k_\ell}}|\{0\leq t\leq N_{k_\ell}-1:(S\times S)^t(x,x^{(n)})\in (C_0^{\alpha}\times(C_0^{\alpha})^c)\cup  ((C_0^{\alpha})^c\times C_0^{\alpha})\}|\\
&\leq d_{B}^{(N_k)}(x,x^{(n)})\to 0\text{ when }n\to\infty.
\end{align*}
(We used the fact that the sets under consideration are clopen and hence~\eqref{benmar1} applies.) Since~\eqref{benmar3} holds, also~\eqref{benmar2} holds as any (complete) $\sigma$-algebra is closed in the metric $\ov{\nu}(\cdot\triangle\cdot)$.

We have shown that $(\ca^{\Z},\mu,S)$ is a measure-theoretic factor of the system
$$
\mathcal{Z}:=((\ca^\Z)^{\times \infty},\ov{\nu}|_{(\ca^\Z)^{\times \infty}},S^{\times\infty}),
$$
i.e.\ $(\ca^{\Z},\mu,S)$ is represented by an $S^{\times\infty}$-invariant sub-$\sigma$-algebra $\mathcal{C}$ of $(\ca^\Z)^{\times \infty}$. Consider the ergodic decomposition of $\mathcal{Z}$:
$$
\ov{\nu}|_{(\ca^\Z)^{\times \infty}}=\int \kappa\, dQ(\kappa).
$$
After the restriction to $\mathcal{C}$, we obtain
$$
\mu=(\ov{\nu}|_{(\ca^\Z)^{\times \infty}})|_\mathcal{C}=\int \kappa|_\mathcal{C}\, dQ(\kappa).
$$
Since, by \cref{p:Benji1}, the system $(\ca^{\Z},\mu,S)$ is ergodic, it follows by the uniqueness of ergodic decomposition that $\kappa|_\mathcal{C}=\mu$ for $Q$-a.e.\ $\kappa$. In other words, $(\ca^{\Z},\mu,S)$ is a measure-theoretic factor of almost every ergodic component of $\mathcal{Z}$.
Moreover, because of the ergodicity of $\mu_n$, $n\geq1$, such an ergodic component is an ergodic joining of the family $\{(\ca^\Z,\mu_n,S)\}_{n\in\N}$. It follows that a typical ergodic component $\nu$ satisfies the assertion of the lemma.
\end{proof}

\begin{proof}[Proof of \cref{t:Benji}]
Suppose $x\in\ca^\N$ is RAP along
$\Nk$. By definition, we can find periodic points
$x_n\in\ca^\N$, $n\in\N$, such that $x_n$ converges to $x$ in
the $d_B^{(N_k)}$ pseudo-metric.

Each $x_n$ is generic for a cyclic rotation
and hence, by \cref{qge2}, $x$ is quasi-generic along $\Nk$
for some invariant measure $\mu\in\mathcal{P}(X_x,S)$.
Following \cref{p:Benji1}, we deduce that $(X_x,\mu,S)$
is ergodic.

Finally, any ergodic joining of cyclic rotations
exhibits rational discrete spectrum and therefore
any factor of an ergodic joining of cyclic rotations
also has rational discrete spectrum. However,
it follows from \cref{l:benmar1} that $(\ca^{\Z},\mu,S)$, which is isomorphic to $(X_x,\mu,S)$, is a factor of a system given by such a joining, hence it has rational
discrete spectrum.
\end{proof}

It is natural to inquire whether \cref{t:Benji} characterizes RAP sequences.
The answer is negative. In the following example we construct a subshift of $\{0,1\}^\Z$
containing a  point that is not RAP but is transitive and generic for an ergodic measure yielding a dynamical system with rational discrete spectrum.

\begin{Example}\label{prz:m1} We will define a uniquely ergodic model for the cyclic rotation on two points (in such a model each point is generic for the unique invariant measure). The subshift $X\subset\{0,1\}^{\Z}$ will consists of three orbits:
\begin{itemize}
\item the periodic point $a=\ldots01.010101\ldots$;
\item the orbit of the point $b$ which arises from the point $a$ by erasing one ``1'';
\item the orbit of the point $c$ which arises from the point $a$ by erasing infinitely many ``1''s so that the distance between the consecutive erased ``1''s goes to infinity.
\end{itemize}
It follows that in the orbit of $c$ longer and longer (``periodic'') words $0101\ldots01$ are approaching the periodic orbit of $a$ from odd and even positions -- this makes the point $c$ non-rational (cf.\ \cref{ex1}). Since $c$ is generic for the measure given by $a$, our claim follows.
\end{Example}

Even though the system constructed in \cref{prz:m1} contains
a transitive point that is not RAP,
it also contains an abundance of transitive points that are in fact RAP.
As the following proposition shows, this is not by coincidence:
\begin{Prop}\label{p:marjo1}
Let $\nu\in\mathcal{P}^e(\{0,1\}^\Z,S)$ be such that
$(\{0,1\}^\Z,\nu,S)$ has rational discrete spectrum. Then $\nu$-a.e.\ $x\in \{0,1\}^\Z$ is RAP.
Moreover, if $X$ is the topological support of $\nu$, then there exists a transitive point $\eta\in X$ for which $\eta|_{[1,\infty)}$ is RAP.
\end{Prop}


\begin{proof}[Proof (cf. {\cite[Theorem 3.19]{MR773063}})]
By assumption, the spectrum of $(\{0,1\}^\Z,\nu,S)$ consists of roots of unity of degree $n_t$, with $n_t \divides n_{t+1}$, $t\geq 0$. Note that if $f_t\circ S =e^{2\pi i /n_t}\cdot f_t$, then $f_t^{n_t}\circ S=f_t^{n_t}$ and hence, by ergodicity, we can assume that $f_t$ takes its values in the group $\{e^{2\pi ij / n_t} : j=0,\dots,n_t-1\}$. By setting
$
D_0^{t}:=\{x\in \{0,1\}^\Z : f_t(x)=1\},
$
we obtain the partition
$$
D_t:=\{D_0^t,SD_0^t,\dots,S^{n_t-1}D_0^t\}
$$
of the space $\{0,1\}^\Z$. Since $\{f_t : t\geq 0\}$ forms an orthonormal basis of $L^2(\nu)$, it follows that for each $t\geq 0$ there is a partition $Q^t=\{E_0^t,E_1^t\}$ of $\{0,1\}^\Z$ such that $E_i^t$ is a union of elements of the partition $D_t$ ($i=0,1$) satisfying
$$
\nu(C^0_0\triangle E^t_0)+\nu(C^1_0\triangle E^t_1)\to0\text{ when }t\to\infty.
$$
Since $S^{n_t}D_0^t=D_0^t$, the sequence $(\1_{E_1^t}(S^kx))_k$ is periodic of period $n_t$, for $\nu$-a.e.\ $x\in \{0,1\}^\Z$. Moreover, $\nu$-a.e.\ point $x\in \{0,1\}^\Z$ satisfies the ergodic theorem for all sets $C_0^i\triangle E_i^t$, $t\geq 0$, $i=0,1$. Hence, the first part of the assertion follows from the pointwise ergodic theorem.

If, additionally, $\nu$ has full topological support then the orbit of $\nu$-a.e.\ point has to intersect any open set belonging to a countable basis of open sets. In other words, $\nu$-a.e.\ point is transitive, so the the second assertion follows from the first one.
\end{proof}

\begin{Remark}\label{topmixing}
If a subshift $(X,S)$ is a strictly ergodic and topologically mixing model of an odometer (note that such models exist due to \cite{MR890422}), then by \cref{p:marjo1} a.e.\ $x\in X$ is a RAP point generating a topologically mixing subshift. Hence, RAP points can generate topologically mixing systems.
\end{Remark}

\subsection{Revisiting \cref{thm_poly-multi-rec-along-rat}}
\label{sec_revisit}

The purpose of this subsection is to give a proof
of the following generalized form of \cref{thm_poly-multi-rec-along-disc-rat-spec}
and discuss some applications thereof (see \cref{c_ave-rec-for-B-free-along-subseq}).

\begin{Def}
\label{def:ave-rec-subseq}
A set $R\subset \N$ is called an \emph{averaging set of polynomial multiple recurrence along $\Mk$} if for all invertible measure preserving systems
$\xbmt$, $\ell\in\N$, $A\in\mathcal{B}$ with $\mu(A)>0$
and for all polynomials $p_i\in\Q[t]$, $i=1,\ldots,\ell$,
with $p_i(\Z)\subset\Z$ and $p_i(0)=0$ for $i\in\{1,\ldots,\ell\}$, one has
$$
\lim_{k\to\infty}\frac{1}{M_k}\sum_{n=1}^{M_k}
\1_R(n)\mu\Big(A\cap T^{-p_1(n)}A\cap\ldots\cap T^{-p_\ell(n)}A\Big)
>0.
$$
\end{Def}

\begin{Th}
\label{thm_poly-multi-rec-along-disc-rat-spec-Nk}
Let $R\subset \N$ with $d^{(N_k)}(R)>0$ and suppose $\eta:=\1_R$
is quasi-generic along $(N_k)$ for a measure $\nu$
such that
$(X_\eta,\nu,S)$ has rational discrete spectrum.
Then there exists a subsequence $(M_k)_{k\geq1}$ of $\Nk$ such that the following are equivalent:
\begin{enumerate}[(I)~~]
\item\label{itm:pmradrs-a}
$R$ is divisible along $(M_k)_{k\geq1}$, that is,
$$
d^{(M_k)}(R\cap u\N):=\lim_{k\to\infty}\frac{|R\cap u\N
\cap \{1,\ldots,M_k\}|}{M_k}>0
$$ for all $u\in \N$.
\item\label{itm:pmradrs-c}
$R$ is an averaging set of polynomial multiple recurrence along
$\Mk$.
\end{enumerate}
\end{Th}

\begin{Remark}
\label{rem:12}
If $R$ is RAP along $\Nk$ then it follows from the proof of \cref{thm_poly-multi-rec-along-disc-rat-spec-Nk} given below that in the statement of the theorem one can take $M_k=N_k$ for all $k\in\N$.
On the other hand, if $R$ is not RAP along $\Nk$ then this is not necessarily true.
For instance, take $R\subset\N$ such that the sequence $\1_R$ equals the sequence $x$ from \cref{ex1}.
Then $\1_R$ is generic for a measure $\nu$ such that
$(X_\eta,\nu,S)$ has rational discrete spectrum (see \cref{rem:131}). However, the set $R$ is not divisible, since $d(R\cap 2\Z)$ does not exist. For the same reason, $\1_R$ will not be a good weight for polynomial multiple convergence (see \cref{def:goodweights}). 
Therefore it is indeed necessary to pass to a subsequence $(M_k)_{k\geq1}$ of $\Nk$ in \cref{thm_poly-multi-rec-along-disc-rat-spec-Nk}.
\end{Remark}

For the proof of
\cref{thm_poly-multi-rec-along-disc-rat-spec-Nk},
we need the following variant of \cref{l1}.

\begin{Lemma}\label{l1-2u}
Let $R\subset\N$ and suppose $\1_{R}=\eta$ is quasi-generic along $\Nk$ for a measure $\nu_\eta$ on $X_\eta\subset \{0,1\}^\Z$ such that $(X_\eta,\nu_\eta,S)$ has rational discrete spectrum. Let $(X,T)$ be a topological dynamical system and let $\mu\in\mathcal{P}^e(X,T)$ be a measure with a generic point $x\in X$. If $(X,\mu,T)$ is totally ergodic then
$$
\lim_{k\to\infty}\frac1{N_k}\sum_{n=1}^{N_k}\1_{R}(n)f(T^{n}x)=d^{(N_k)}(R)\int_Xf\,d\mu
$$
for each $f\in C(X)$.
\end{Lemma}
\begin{proof}
Consider $(\tilde\eta,x)\in X_\eta\times X$, where $\tilde\eta\in\{0,1\}^\Z$
is defined as $\tilde\eta(n)=\eta(n)$ for all $n\in\N$ and $\tilde\eta(n)=0$ for all $n\in\Z\setminus\N$.
Since $(X,\mu,T)$ is totally ergodic and
$(X_\eta,\nu_\eta,S)$ has rational discrete spectrum,
it follows that $(X,\mu,T)$ and $(X_\eta,\nu_\eta,S)$ are
spectrally disjoint \cite{Ha-Pa}. In particular,
the only joining of these two
systems is given by the product measure $\nu_\eta\otimes\mu$, i.e.,
$J((X_\eta,\nu_\eta,S),(X,\mu,T))=\{\nu_\eta\otimes\mu\}$.
It follows that
\begin{equation}\label{rrr1-2u}
\mbox{$(\tilde\eta,x)$ is quasi-generic along $\Nk$ for the product measure
$\nu_\eta\ot\mu$.}
\end{equation}
Fix $f\in C(X)$ and let $F\colon X_\eta\to\{0,1\}$ be given by $F(z)=z(0)$. Then $F\in C(X_\eta)$ and, since $\eta$ is quasi-generic along $\Nk$ for $\nu_\eta$, we obtain
$$
\int_{X_\eta}F\,d\nu_\eta=\lim_{k\to\infty}\frac1{N_k} \sum_{n=1}^{N_k}F(S^n\tilde \eta)=
\lim_{k\to\infty}\frac1{N_k} \sum_{n=1}^{N_k}\1_R(n)= d^{(N_k)}(R).
$$
In view of~\eqref{rrr1-2u}, we have
$$
\lim_{k\to\infty}\frac1{N_k}\sum_{n=1}^{N_k}F\ot f\big((S\times T)^{n}(\tilde\eta,x)\big)= \int F\ot f\,d(\nu_\eta\ot\mu)=\int_{X_\eta}F\,d\nu_\eta~\int_Xf\,d\mu.
$$
Finally, one only needs to observe that
$$
\frac1{N_k}\sum_{n=1}^{N_k}F\ot f\big((S\times T)^{n}(\tilde\eta,x)\big)=\frac1{N_k}\sum_{n=1}^{N_k}\1_{R}(n)f(T^{n}x)
$$
and the proof is complete.
\end{proof}

\begin{Lemma}\label{irg-l1}
Let $R\subset\N$ and suppose $\1_{R}=\eta$ is quasi-generic along $\Nk$ for a measure $\nu$ on $X_\eta\subset \{0,1\}^\Z$ such that $(X_\eta,\nu,S)$ has rational discrete spectrum.
For $u\in \N$ and $j\in\{0,1,\ldots,u-1\}$ let $(R-j)/u$ denote the set $\{n\in\N: nu+j\in R\}$. Then there exists a subsequence $(M_k)_{k\geq 1}$ of $\Nk$ with the property that for every $u\in\N$ and $j\in\{0,1,\ldots,u-1\}$ the point $\eta_{u,j}:=\1_{(R-j)/u}$ is quasi-generic along $(M_k/u)_{k\geq 1}$ for a measure $\nu_{u,j}$ such that $(X_{\eta_{u,j}},\nu_{u,j},S)$ has rational discrete spectrum.
\end{Lemma}

\begin{proof}
By applying a standard diagonalization method, choose a subsequence $(M_k)_{k\geq 1}$ of $\Nk$ such that for every $u\in\N$ the point $\eta$ is quasi-generic along $(M_k/u)_{k\geq 1}$ for a measure $\mu_u$ with respect to the transformation $S^u$. In other words, for every $u\in\N$ and every continuous function $f\in C(\{0,1\}^\Z)$, the limit
$$
\lim_{k\to\infty}\frac{u}{M_k}\sum_{n=1}^{M_k/u} f(S^{un}\tilde\eta)
$$
exists and equals $\int f~d\mu_u$, where $\tilde\eta\in\{0,1\}^\Z$ is any two sided sequence that extends $\eta\in\{0,1\}^\N$. Define $\mu_{u,j}:=S^j\mu_u$ and note that $\mu_{u,j}$ is $S^u$-invariant.
Since
\begin{eqnarray*}
\frac{1}{u}\sum_{j=0}^{u-1}\int f~d\mu_{u,j} &=&
\frac{1}{u}\sum_{j=0}^{u-1}\int S^jf~d\mu_u
\\
&=&
\frac{1}{u}\sum_{j=0}^{u-1}\lim_{k\to\infty}\frac{u}{M_k}\sum_{n=1}^{M_k/u} f(S^{un+j}\tilde\eta)
\\
&=&
\lim_{k\to\infty}\frac{1}{M_k}\sum_{n=1}^{M_k} f(S^{n}\tilde\eta)
\\
&=&\int f~d\nu,
\end{eqnarray*}
we deduce that
\begin{equation}
\label{eq:rw1--2}
\frac{1}{u}\sum_{j=0}^{u-1}\mu_{u,j}=\nu.
\end{equation}
In particular, we have that for each Borel-measurable function $g$ on $X_\eta$,
\begin{equation}
\label{eq:rw1-1}
\|g\|_{L^2(\nu)}^2=\frac{1}{u}\sum_{j=0}^{u-1}\|g\|_{L^2(S^j\mu_u)}^2\geq \max_{0\leq j<u}\frac{1}{u}\|g\|_{L^2(\mu_{u,j})}^2.
\end{equation}

We deduce from \eqref{eq:rw1--2} that $\mu_{u,j}$ is absolutely continuous with respect to $\nu$, that is, any set that has zero measure with respect to $\nu$ also has zero measure with respect to $\mu_{u,j}$.
Therefore, any eigenfunction of the system $(X_\eta,\nu,S)$ with eigenvalue $\lambda$ is an eigenfunction of the system $(X_\eta, \mu_{u,j}, S^u)$ with eigenvalue $\lambda^u$.
The system $(X_\eta,\nu,S)$ has rational discrete spectrum and so the span of eigenfunctions with rational eigenvalue is dense in $L^2(\nu)$. However, if a class of bounded measurable functions is dense in $L^2(\nu)$, then, by \eqref{eq:rw1-1}, it is also dense in $L^2(\mu_{u,j})$. Hence, the span of eigenfunctions with rational eigenvalue is dense in $L^2(\mu_{u,j})$, which proves that $(X_\eta, \mu_{u,j}, S^u)$ has rational discrete spectrum.

Let $\Phi:\{0,1\}^\Z\to\{0,1\}^\Z$ denote the map defined by the rule
$\Phi(x)(n)=x(un+j)$ for all $x\in\{0,1\}^\Z$.
It is straightforward to verify that $\Phi(X_\eta)=X_{\eta_{u,j}}$ and that
$\Phi$ satisfies
\begin{equation}
\label{law-irg-1}
\Phi\circ S^u = S\circ\Phi.
\end{equation}
Let $\nu_{u,j}$ denote the push-forward of $\mu_{u,j}$ under $\Phi$.
Since $S^j\eta$ is quasi-generic along $(M_k/u)_{k\geq 1}$ for $\nu_{u,j}$ under the transformation $S^u$, it follows from \eqref{law-irg-1} that $\eta_{u,j}$ is generic for the measure $\nu_{u,j}$ along $(M_k/u)_{k\geq 1}$.
Finally, observe that $(X_{\eta_{u,j}},\nu_{u,j},S)$ has rational discrete spectrum because $(X_\eta,\mu_{u,j}, S^u)$ has rational discrete spectrum.
\end{proof}

\begin{Th}
\label{thm_poly-multi-conv-along-disc-rat-spec-Nk}
Let $R\subset \N$ and suppose $\eta:=\1_R$ is quasi-generic along $(N_k)$ for a measure $\nu$ such that $(X_\eta,\nu,S)$ has rational discrete spectrum. Then there exists a subsequence $(M_k)_{k\geq1}$ of $\Nk$ such that $R$ is an averaging set of polynomial multiple convergence along $\Mk$, that is, for all invertible measure preserving systems $\xbmt$, $\ell\in\N$, $A\in\mathcal{B}$ with $\mu(A)>0$ and for all polynomials $p_i\in\Q[t]$, $i=1,\ldots,\ell$,
with $p_i(\Z)\subset\Z$ for $i\in\{1,\ldots,\ell\}$, the limit
\begin{equation}
\label{eq:irg-1-2}
\lim_{k\to\infty}\frac{1}{M_k}\sum_{n=1}^{M_k}
\1_R(n)\mu\Big(A\cap T^{-p_1(n)}A\cap\ldots\cap T^{-p_\ell(n)}A\Big)
\end{equation}
exists (cf. \cref{def:goodweights}).
\end{Th}

\begin{proof}
By applying \cref{irg-l1} we can find a subsequence $(M_k)_{k\geq 1}$ of $\Nk$ such that for every $u\in\N$ and every $j\in\{0,1,\ldots,u-1\}$ the point $\eta_{u,j}:=\1_{(R-j)/u}$ is quasi-generic along $(M_k/u)_{k\geq 1}$ for a measure $\nu_{u,j}$ such that $(X_{\eta_{u,j}},\nu_{u,j},S)$ has rational discrete spectrum.
Let $\xbmt$, $\ell\in\N$, $A\in\mathcal{B}$ with $\mu(A)>0$ and $p_i\in\Q[t]$, $i=1,\ldots,\ell$,
with $p_i(\Z)\subset\Z$ for $i\in\{1,\ldots,\ell\}$ be arbitrary.
Define
$$
\varphi(n):=\mu\Big(A\cap T^{-p_1(n)}A\cap\ldots\cap T^{-p_\ell(n)}A\Big).
$$
In view of \cref{thm_4.1}, we can find for every $\epsilon >0$ a basic nilsequence $(f(T_g^n x))$, where
$T_g$ is an ergodic nilrotation on some nilmanifold $X=G/\Gamma$, $f\in C(X)$ and $x\in X$, such that
$$
\limsup_{N\to\infty} \frac{1}{N}\sum_{n=1}^N
\vert \varphi(n)- f(T_g^n x)\vert \leq \epsilon.
$$
It thus suffices to show that the limit
\begin{equation}
\label{eq:irg-1-3}
\lim_{k\to\infty}\frac{1}{M_k}\sum_{n=1}^{M_k}
\1_R(n)f(T_g^n x)
\end{equation}
exists, because from this it follows that
$$
\limsup_{k\to\infty} \frac{1}{M_k}\sum_{n=1}^{M_k}
\1_R(n)\varphi(n) -\liminf_{k\to\infty} \frac{1}{M_k}\sum_{n=1}^{M_k}
\1_R(n)\varphi(n)\leq 2\epsilon,
$$
from which we can deduce that the limit in \eqref{eq:irg-1-2} exists, as $\epsilon$ was chosen arbitrarily.

Let $X_0,X_1,\ldots,X_{u-1}$ denote the connected components of the nilmanifold $X$. Since $T_g$ is ergodic, it cyclically permutes the connected components of $X$. We can therefore assume without loss of generality that $T_gX_{j}=X_{j+1\bmod u}$. In particular, $T_g^u X_j=X_j$ and, according to \cref{prop_2.1}, the nilsystem $(X_j,\mu_{X_j}, T_g^u)$ is totally ergodic.
Note that
$$
\lim_{k\to\infty}\frac{1}{M_k}\sum_{n=1}^{M_k}
\1_R(n)f(T_g^n x)=
\sum_{j=0}^{u-1}\lim_{k\to\infty}\frac{u}{M_k}\sum_{n=1}^{M_k/u}
\1_{(R-j)/u}(n)f(T_g^{un+j} x),
$$
where the limit on the left hand side in the above equation exists if all the limits for $j=0,1,\ldots,u-1$ on the right hand side exist.
It remains to show that for every $j\in\{0,1,\ldots,u-1\}$ the limit
$$
\lim_{k\to\infty}\frac{u}{M_k}\sum_{n=1}^{M_k/u}
\1_{(R-j)/u}(n)f(T_g^{un+j} x)
$$
exists.
Suppose $T_g^jx\in X_{j_0}$ for some $j_0\in\{0,1,\ldots,u-1\}$. Since $(X_{j_0},\mu_{X_j}, T_g^u)$ is totally ergodic (and uniquely ergodic), it follows from \cref{l1-2u} that 
$$
\lim_{k\to\infty}\frac{u}{M_k}\sum_{n=1}^{M_k/u}
\1_{(R-j)/u}(n)f(T_g^{un+j} x)
=d^{(M_k)}((R-j)/u)\int f\,d\mu_{X_{j_0}}.
$$
This finishes the proof.
\end{proof}

The proof of \cref{thm_poly-multi-rec-along-disc-rat-spec-Nk} hinges on \cref{l1-2u} and \cref{irg-l1} and it is a modification of the proof of \cref{thm_poly-multi-rec-along-rat} given in \cref{sec_multi-rec}.

\begin{proof}[Proof of \cref{thm_poly-multi-rec-along-disc-rat-spec-Nk}]
Let $\xbmt$ be an invertible measure preserving system, let  $R\subset \N$ with $d^{(N_k)}(R)>0$ and assume $\eta:=\1_R$
is quasi-generic along $(N_k)$ for a measure $\nu$
such that $(X_\eta,\nu,S)$ has rational discrete spectrum.
Take any $A\in\cb$ with $\mu(A)>0$ and let $p_1,\ldots,p_\ell\in\Q[t]$ with $p_i(\Z)\subset\Z$, $p_i(0)=0$, $i=1,\ldots,\ell$, be arbitrary.
Choose a subsequence $(M_k)_{k\geq 1}$ of $\Nk$ such that the conclusion of both \cref{thm_poly-multi-conv-along-disc-rat-spec-Nk} and \cref{irg-l1} hold.
We will show that
\begin{equation}\label{rr20-4}
\lim_{k\to\infty}\frac{1}{M_k}\sum_{n=1}^{M_k}
\1_R(n)\varphi(n)>0,
\end{equation}
where $\varphi(n)=\mu\big(A\cap T^{-p_1(n)}A\cap\ldots\cap T^{-p_\ell(n)}A\big)$.
The existence of the limit in 
\eqref{rr20-4} follows from \cref{thm_poly-multi-conv-along-disc-rat-spec-Nk}. It remains to show that the limit in \eqref{rr20-4} is positive.

Arguing as in the proof of \cref{p5new},
we can assume without loss of generality that
$(\varphi(n))$ is a nilsequence.
By \cref{Cor-uniformity} (see the appendix), there exists $\delta>0$ such that
\begin{equation}\label{rr20-1-4}
\lim_{N\to\infty}\frac{1}{N}\sum_{n=1}^N \varphi(un)>\delta
\text{ for all }u\in\N.
\end{equation}
We can approximate $(\varphi(n))$
by a basic nilsequence $(f(T_g^n x))$, where
$T_g$ is a  nilrotation on some nilmanifold $X=G/\Gamma$, $f\in C(X)$ and $x\in X$, such that
$\vert \varphi(n)- f(T_g^n x)\vert \leq \delta/4 $ for all $n\in\N$.

Using \cref{prop_2.1}, we can find $u\in\N$ and a sub-nilmanifold $Y\subset X$ containing $x$
such that $(Y,\mu_Y,T_{g^u})$ is totally ergodic.
It follows from \cref{irg-l1} that
$\1_{R/u}$ is quasi-generic along $(M_k/u)_{k\in\N}$ for a measure $\nu'$ such that the system $(\{0,1\}^\Z,\nu',S)$ has rational discrete spectrum.
It now follows from \cref{l1-2u} that
\begin{equation}\label{rr21-1-4}
\lim_{k\to\infty}\frac{u}{M_k}\sum_{n=1}^{M_k/u}\1_{R/u}(n)f(T_{g^u}^{n}x)=
d^{(M_k/u)}(R/u)\int_Y f\,d\mu_Y.
\end{equation}

Finally, combining \eqref{rr20-1-4} and \eqref{rr21-1-4} and $\vert \varphi(un)- f(T_{g^u}^n x)\vert \leq \delta/4$, we obtain
\begin{eqnarray*}
\lim_{k\to\infty}\frac1{M_k}\sum_{n=1}^{M_k} \1_{R}(n)\psi(n)
&\geq&
\lim_{k\to\infty}\frac1{M_k}\sum_{n=1}^{M_k}\1_{R\cap u\N}(n)\psi(n)
\\
&=&
\frac{1}{u}\left(
\lim_{k\to\infty}\frac{u}{M_k}\sum_{n=1}^{M_k/u}\1_{R/u}(n)\psi(un)\right)
\\
&\geq&
\frac{1}{u}\left(
\lim_{k\to\infty}\frac{u}{M_k}\sum_{n=1}^{M_k/u} \1_{R/u}(n) f(T_{g^u}^{n} x)-
\frac{\delta }{4}d^{(M_k/u)}(R/u)\right)
\\
&\geq&  \frac{1}{u}\left(\frac{3\delta }{4}d^{(M_k/u)}(R/u) - \frac{\delta }{4}d^{(M_k/u)}(R/u)\right)~>~0.
\end{eqnarray*}
This completes the proof.
\end{proof}

As an application of \cref{thm_poly-multi-rec-along-disc-rat-spec-Nk} together with \cref{rem:12},
we obtain a strengthening of \cref{c_ave-rec-for-B-free}
that also applies to $\sB$-free numbers for a general set $\sB\subset\N\setminus\{1\}$.

\begin{Th}
\label{c_ave-rec-for-B-free-along-subseq}
Suppose $\sB\subset\N\setminus\{1\}$. Then there exist an increasing
sequence of positive integers $\Nk$
and a set $D\subset \cf_\sB$
with $d^{(N_k)}(\cf_\sB\setminus D)=0$ such that for all
$r\in\N$ the following are equivalent:
\begin{itemize}
\item
$r\in D$;
\item
$\cf_\sB-r$ is divisible along $\Nk$;
\item
$\cf_\sB-r$ is an averaging set of polynomial multiple recurrence
along $\Nk$.
\end{itemize}
\end{Th}

\begin{proof}[Proof of \cref{c_ave-rec-for-B-free-along-subseq}]
Let $\sB\subset \N$ be arbitrary.
If $\sB$ is Behrend then
$\cf_\sB$ has zero density and we can put $D=\emptyset$.
Thus, let us assume that $\sB$ is not Behrend. Therefore,
the logarithmic density of $\bdelta(\cf_\sB)$ is positive. Moreover, by \cref{daer},
$\bdelta(\cf_\sB)=d^{(N_k)}(\cf_\sB)$ for some increasing sequence $\Nk$. We now repeat word for word the proof of \cref{c_ave-rec-for-B-free} with density and divisibility replaced by density along $\Nk$ and divisibility along $\Nk$, respectively.
\end{proof}

\section{Applications to Combinatorics}
\label{seq_applications-combinatorics}

In this section we show how the results obtained in the previous sections
allow us to derive new refinements of the polynomial Szemer{\'e}di theorem.
In particular, we give a proof of
\cref{c:fc-rat-without-intersectivitiy} and of \cref{c:fc-rat}.

First, let us recall Furstenberg's correspondence principle:

\begin{Prop}[Furstenberg correspondence principle, see \cite{Bergelson87,Bergelson96}]
\label{prop:fc}
Let $E\subset \N$ be a set with positive upper density $\overline{d}(E)>0$.
Then there exist an invertible measure preserving system $\xbmt$
and a set $A\in \cb$ with $\mu(A)\geq \overline{d}(E)$ such that
for all $n_1,\ldots,n_\ell \in \N$, one has
\begin{equation}
\label{eq:fc}
\overline{d}\left(E\cap  (E-n_1)\cap \ldots\cap (E-n_\ell)\right)\geq
\mu\left(A\cap  T^{-n_1}A\cap \ldots\cap T^{-n_\ell} A\right).
\end{equation}
\end{Prop}

We have now the following result regarding averaging sets of polynomial multiple recurrence along $\Nk$.


\begin{Prop}
\label{c:fc-w-i}
Let $\Nk$ be an increasing sequence and let $R\subset \N$ be an averaging set of polynomial multiple recurrence along $\Nk$.
Then for any set $E\subset \N$ with $\overline{d}(E)>0$
and any polynomials $p_1,\ldots,p_\ell \in\Q[t]$, which satisfy
$p_i(\Z)\subset\Z$ and $p_i(0)=0$ for all $i\in\{1,\ldots,\ell\}$,
there exists $\beta>0$ such that the set
$$
\left\{n\in R:\overline{d}\Big(
E\cap (E-p_1(n))\cap \ldots\cap(E-p_\ell(n))
\Big)>\beta \right\}
$$
has positive lower density (with respect to $\Nk$).
\end{Prop}

\begin{Prop}
\label{c:fc}
Let $R\subset \N$ be an averaging set of polynomial multiple recurrence
(along $\Nk$).
Then for any $E\subset \N$ with $\overline{d}(E)>0$
and any polynomials $p_1,\ldots,p_\ell\in\Q[t]$, which satisfy
$p_i(\Z)\subset\Z$ and $p_i(0)=0$ for all $i\in\{1,\ldots,\ell\}$,
there exists a subset $R'\subset R$ satisfying $\overline{d}(R')>0$ such that for any finite subset
$F\subset R'$ we have
$$
\overline{d}\left(\bigcap_{n\in F}\Big(
E\cap \big(E-p_1(n)\big)\cap \ldots\cap\big(E-p_\ell(n)\big)
\Big)\right)>0.
$$
\end{Prop}

By combining \cref{c:fc-w-i}
with \cref{thm_poly-multi-rec-along-rat},
we immediately obtain a proof of \cref{c:fc-rat-without-intersectivitiy}.
Likewise, by combining \cref{c:fc}
with \cref{thm_poly-multi-rec-along-rat},
we immediately obtain a proof of \cref{c:fc-rat}.

We can also get a slight generalization of
\cref{c:fc-rat-without-intersectivitiy}:
we can replace the notions `rational' and `divisible'
with `rational along $\Nk$' and `divisible along $\Nk$'
for any increasing sequence $\Nk$ and, in virtue of
\cref{thm_poly-multi-rec-along-disc-rat-spec-Nk},
the statement of \cref{c:fc-rat-without-intersectivitiy}
remains valid.

\cref{c:fc-w-i} is an immediate consequence of
Furstenberg's correspondence principle and of the definition
of an averaging sets a polynomial multiple recurrence.

For the proof of \cref{c:fc} we need the
following theorem.

\begin{Th}[see {\cite[Theorem 1.1]{Bergelson85}}]
\label{l:VB-thm1.1}
Let $\xbm$ be a probability space and suppose $A_n\in\cb$,
$\mu(A_n)\geq \delta>0$, for $n=1,2,\ldots$.
Then there exists a set $P\subset \N$ with
$\overline{d}(P)\geq \delta$
such that for any finite subset $F\subset P$, we have
$$
\mu\left( \bigcap_{n\in F}A_n\right)>0.
$$
\end{Th}

\begin{proof}[Proof of \cref{c:fc}]
Let $R\subset \N$ be an averaging set of polynomial multiple recurrence
along $\Nk$.
Let $E\subset \N$ with $\overline{d}(E)>0$
and  let $p_i\in\Q[t]$, $i=1,\ldots,\ell$,
with $p_i(\Z)\subset\Z$ and $p_i(0)=0$, for all $i\in\{1,\ldots,\ell\}$.

By applying 
\cref{prop:fc}, we can find an invertible measure preserving system 
$\xbmt$ and a set $A\in \cb$ with $\mu(A)\geq \overline{d}(E)$ such that
\eqref{eq:fc} is satisfied. Next, since $R$ is an averaging set of polynomial multiple recurrence along $\Nk$,
we can find some $\delta>0$ such that the set
$$
D:=\left\{n\in R: \mu\left( A\cap T^{-p_1(n)}A\cap \ldots\cap T^{-p_\ell(n)}A\right)>\delta\right\}
$$
has positive lower density, i.e.,
$\underline{d}(D)=\liminf_{N\to\infty}\frac{|D\cap\{1,\ldots,N\}|}{N}>0$. Let $n_1,n_2,n_3,\ldots$ be an enumeration of
$D$ and let $A_i\in\cb$ denote the set
$$
A_i:=A\cap T^{-p_1(n_i)}A\cap \ldots\cap T^{-p_\ell(n_i)}A.
$$
Then, according to \cref{l:VB-thm1.1}, we can find a set
$P\subset \N$ with $\overline{d}(P)\geq \delta$
such that for any finite subset $F\subset P$, we have
\begin{equation}
\label{eq:intersective}
\mu\left( \bigcap_{n\in F}A_n\right)>0.
\end{equation}
Let $R':=\{n_i: i\in P\}$. Then $R'\subset R$ and it is straightforward to
show that $\overline{d}(R')>0$. Moreover,
combining~\eqref{eq:intersective} with~\eqref{eq:fc}, for
any finite subset $\{n_1,\ldots,n_k\}\subset R'$, we obtain
$$
\overline{d}\left(\bigcap_{i=1}^r\Big(
E\cap (E-p_1(n_i))\cap \ldots\cap(E-p_\ell(n_i))
\Big)\right)>0.
$$
From this the claim follows immediately.
\end{proof}

For the special case of $\sB$-free numbers, we have the following combinatorial corollary of the above results.

\begin{Th}
\label{th:b-free-corollary-combi}
For $\sB\subset\N\setminus\{1\}$ let $\cf_{\sB}$ denote the set of
$\sB$-free numbers and let $\Nk$
be any sequence of increasing positive integers
such that $d^{(N_k)}(\cf_\sB)$ exists
and is positive.
Then there exists a set $D\subset \cf_\sB$ with
$d^{(N_k)}(\cf_\sB\setminus D)=0$
and such that for all $r\in D$,
for all $E\subset \N$ with $\overline{d}(E)>0$
and any polynomials $p_i\in\Q[t]$, $i=1,\ldots,\ell$,
which satisfy
$p_i(\Z)\subset\Z$ and $p_i(0)=0$, for all $i\in\{1,\ldots,\ell\}$,
there exists $\beta>0$ such that the set
$$
\left\{n\in \cf_{\sB}-r:\overline{d}\Big(
E\cap (E-p_1(n_i))\cap \ldots\cap(E-p_\ell(n_i))
\Big)>\beta \right\}
$$
has positive lower density with respect to $\Nk$.
If, additionally, $\sB$ is taut then one can take $D=\cf_{\sB}$.
\end{Th}

A proof of \cref{th:b-free-corollary-combi} follows
immediately by combining \cref{c:fc-w-i} and
\cref{c_ave-rec-for-B-free-along-subseq}.

\section{Rational sequences and Sarnak's conjecture }\label{SAS}

\cref{SAS} is divided into two subsections. In Subsection \ref{SAS1} we give a proof of \cref{synch-ratio}, which states that any automatic sequence generated by a synchronized automaton is WRAP. 
In Subsection \ref{SAS2} we use \cref{synch-ratio} to
strengthen a result obtained by Deshouillers, Drmota and M{\"u}llner in \cite{De-Dr-Mu}, which states that sequences given by synchronized automata satisfy Sarnak's conjecture. 

\subsection{Synchronized automata and substitutions}\label{SAS1}

We begin with a proof of
\cref{synch-ratio}. For the convenience of the reader, we restate the proposition here.

\begin{named}{\cref{synch-ratio}}{}
Each automatic sequence given by a synchronized automaton $\Automaton$ is WRAP.
\end{named}

\begin{proof}
Let $\Automaton=(Q,\alphB,\delta,q_0,\tau)$ be a synchronized complete deterministic automaton with set of states $Q:=\{q_0,\ldots,q_r\}$, input alphabet $\alphB:=\{0,1,\ldots,k-1\}$, finite output alphabet $\ca$, transition function $\delta\colon Q\times \alphB \to Q$, initial state $q_0$ and output mapping $\tau:Q\to\ca$.
For $n\in \N$ let $[n]_k\in\alphB^*$ be defined as in \cref{sec:drs.e}.
Let  $a(n)=\tau(\delta(q_0,[n]_k))$, $n\in\N$, denote the automatic sequence generated by the synchronized automaton $\Automaton$.

Fix $\vep>0$. Let $n_1$ be such that at least $k^{n_1}(1-\vep)$ words of length $n_1$ are synchronizing (see \cref{def_synchro}). In other words, if we set
$$
K:=\{0\leq m < k^{n_1} : [m]_k \text{ is synchronizing}\}
$$
then we have $|K|\geq k^{n_1}(1-\vep)$ (note that if $w$ is a synchronizing word then so is every one of its extensions). Notice that
\begin{equation}\label{eq4}
a(n)=a(m) \text{ whenever }n\equiv m \bmod k^{n_1} \text{ for some }m\in K
\end{equation}
as $[n]_k$ and $[m]_k$ share the last $k^{n_1}$ digits. Consider $a'$ given by
$$
a'(n):=\begin{cases}
a(n) & \text{if }n\bmod k^{n_1}\text{ belongs to } K,\\
0 &\text{otherwise.}
\end{cases}
$$
Notice that $a'$ is periodic of period $k^{n_1}$: for $0\leq m<K$, $j\geq 0$, we have
$$
a'(m+jk^{n_1})=\begin{cases}
a(m),& \text{ if }m\in K,\\
0,& \text{ otherwise}.
\end{cases}
$$
Moreover, using~\eqref{eq4}, we obtain
\begin{multline*}
d_W(a,a')=\limsup_{N\to\infty}\sup_\ell\frac1N\left|\{1\leq n\leq N: a_{n+\ell}\neq a'_{n+\ell}\}\right|\\
\leq \limsup_{N\to\infty}\sup_\ell\frac{1}{N}|\{1\leq n\leq N : n+\ell\bmod k^{n_1}\not\in K\}|=\frac{k^{n_1}-|K|}{k^{n_1}}\leq \vep
\end{multline*}
and the result follows.
\end{proof}

\subsection{Orthogonality of RAP and WRAP sequences to the M\"obius function}
\label{SAS2}
\label{secSar}


Let $(X,T)$ be a topological system, that is, $X$ is a compact metric space and $T\colon X\to X$ a homeomorphism. Let $\mob$ denote the classical \emph{M{\"o}bius function}, i.e., for all $n\in\N$,
$$
\mob(n)=
\begin{cases}
(-1)^k,&\text{if there exist $k$ distinct prime numbers $p_1,\ldots,p_k$}
\\
&\text{such that $n=p_1\cdot\ldots\cdot p_k$;}
\\
0,&\text{otherwise.}
\end{cases}
$$
We write $(X,T)\perp\boldsymbol{\mu}$
whenever $\lim_{N\to\infty}\frac1N\sum_{n=1}^{N}f(T^nx)\boldsymbol{\mu}(n)=0$  for all $f\in C(X)$ and $x\in X$. Sarnak's conjecture \cite{sarnak-lectures} states that
\begin{equation}\label{sar1}
\mbox{$(X,T)\perp\boldsymbol{\mu}$ whenever the topological entropy of $T$ is zero.}\end{equation}

If $x\in \ca^\N$ is an automatic sequence 
generated by a synchronized automata
then its sub-word complexity is at most linear (see, e.g.\ Thm. 10.3.1 in~\cite{MR1997038}), which implies
that the entropy of the dynamical system
$(X_x,S)$ is zero.
It is therefore natural to ask if systems
generated by such automatic sequences
satisfy Sarnak's conjecture.
This question was answered affirmatively in \cite{De-Dr-Mu}.

The next theorem states that any $W$-rational system satisfies Sarnak's conjecture. In view of \cref{synch-ratio}, our result can be viewed as an extension of the main result in \cite{De-Dr-Mu}.

\begin{Th}\label{Wshort}
Let $x\in\ca^\N$ be WRAP. Then for all $f\in C(X_x)$ and $z\in X_x$, we have
\begin{equation}
\label{eq:sc-ep-3}
\lim_{N\to\infty}\frac1N\sum_{n=1}^N f(S^nz)\mob(n)=0.
\end{equation}
Equivalently, $(X_x,S)\perp \boldsymbol{\mu}$.
\end{Th}

For the proof of \cref{Wshort} we need two lemmas. The first lemma is a slight modification of \cref{l:asympdenswords} involving the Weyl pseudo-metric $d_W$ instead of the Besicovitch pseudo-metric $d_B$.

\begin{Lemma}\label{lem:sc-ep-0}
Let $x,y\in \ca^{\N}$,
$n_1,\ldots,n_\ell\in\Z$ and $\alpha_1,\ldots,\alpha_\ell \in \ca$. Then for $C=C^{\alpha_1,\ldots,\alpha_\ell}_{n_1,\ldots,n_\ell}$ we have
$$
\limsup_{H\to\infty}\sup_{m\in\N}\frac{1}{H}
\sum_{m\leq h<m+H} \big|\1_C(S^h \tilde x)-\1_C(S^h\tilde y)\big|
~\leq~ \ell d_W(x,y),
$$
where $\tilde{x},\tilde{y}\in\ca^\Z$ are any two-sided sequences extending $x$ and $y$, respectively.
\end{Lemma}

The proof of \cref{lem:sc-ep-0} is very similar to the proof of \cref{l:asympdenswords} and is omitted.

The next lemma, which is also needed for the proof of \cref{Wshort}, states that RAP sequences are orthogonal to the M{\"o}bius function $\mob$.

\begin{Lemma}\label{lem:sc-ep-W}
Suppose $x\in\ca^\N$ is RAP and $f\in C(\ca^\Z)$. Then
\begin{equation}
\label{eq:sc-ep}
\lim_{N\to\infty}\frac1N\sum_{n=1}^N f(S^n \tilde x)\mob(n)~=~0.
\end{equation}
\end{Lemma}

\begin{proof}
Since any continuous function $f\in C(X_x)$ can be approximated uniformly by cylinder sets $C=C^{\alpha_1,\ldots,\alpha_\ell}_{n_1,\ldots,n_\ell}$, it suffices to show \eqref{eq:sc-ep} for the special case where $f=\1_C=\1_{C^{\alpha_1,\ldots,\alpha_\ell}_{n_1,\ldots,n_\ell}}$ for any $n_1,\ldots,n_\ell\in\Z$ and $\alpha_1,\ldots,\alpha_\ell\in\ca$. 

Hence, let $\ell\in\N$, $n_1,\ldots,n_\ell\in\Z$ and $\alpha_1,\ldots,\alpha_\ell\in\ca$ be arbitrary.
Fix $\epsilon>0$. Since $x$ is RAP we can find a periodic sequence $y\in\ca^\N$ such that $d_B(x,y)\leq \epsilon/\ell$. Let $\tilde y\in\ca^\Z$ be a two-sided periodic sequence that extends $y$.
Then, using \cref{l:asympdenswords}, we get
\begin{equation}
\label{eq:sc-ep-7}
\limsup_{N\to\infty}\left|\frac1N\sum_{n=1}^N \1_C(S^n \tilde x)\mob(n)
~-~
\frac1N\sum_{n=1}^N \1_C(S^n \tilde y)\mob(n)\right| ~\leq~\ell d_B(x,y)~=~\epsilon.
\end{equation}
It is a well-known fact that Dirichlet's prime number theorem along arithmetic progressions is equivalent to the assertion that for any periodic sequence $a(n)$ one has
$
\lim_{N\to\infty}\frac1N\sum_{n=1}^N a(n)\mob(n) =0
$. In particular, $a(n)=\1_C(S^n \tilde y)$ is a periodic sequence and hence 
$$
\lim_{N\to\infty}\frac1N\sum_{n=1}^N \1_C(S^n \tilde y)\mob(n) =0.
$$
Therefore, \eqref{eq:sc-ep-7} simplifies to 
$$
\limsup_{N\to\infty}\left|\frac1N\sum_{n=1}^N \1_C(S^n \tilde x)\mob(n)
\right| ~\leq~\epsilon.
$$
Since $\epsilon>0$ was chosen arbitrarily, the proof of \eqref{eq:sc-ep} is completed.
\end{proof}

\begin{proof}[Proof of \cref{Wshort}]
Let $x\in\ca^\N$ be WRAP and let $f\in C(X_x)$ and $z\in X_x$ be arbitrary. 
It follows from \cref{l:flo} that $z|_\N$ is WRAP and therefore $z|_\N$ is also RAP. Hence \eqref{eq:sc-ep-3} follows directly from \eqref{eq:sc-ep}.
\end{proof}

In light of \cref{Wshort} it is natural to inquire about the behavior of averages of the from
\begin{equation}
\label{eq:sc-ep-7-a}
\frac1H\sum_{m\leq h<m+H} f(S^h z)\mob(n)
\end{equation}
for large values of $H$ and arbitrary $m\in\N$. 
It is believed that the expression in \eqref{eq:sc-ep-7-a} does not converge to $0$ (as $H$ approaches $\infty$) uniformly in $m$.\footnote{Indeed, by Chowla's conjecture \cite{Chowla} (see also \cite{2014arXiv1410,sarnak-lectures}) it follows that for every word $w\in\{0,1\}^H$ that appears in $\mob^2=\1_Q$, where $Q$ denotes the set of squarefree numbers, all words in $v\in \{-1,0,1\}^H$ with $v^2=w$ must appear in $\mob$. In particular, (assuming Chowla's conjecture) for every $H\geq 1$ there is $m\geq 1$ such that $\1_Q(h)=\mob(m+h)$ for all $h\in [1,H]$ and therefore \eqref{eq:sc-ep-7-a} with $f=1$ is close to $\frac{6}{\pi^2}$.}
Nonetheless, using recent results of Matomaki, Radziwi\l \l \ and Tao \cite{Ma-Ra-Ta}, we will show that for large $H$ and ``typical'' $m\in\N$ the averages in \eqref{eq:sc-ep-7-a} are small.
Such averages of $\mob$ (or, more generally, of bounded multiplicative functions) over ``short intervals'' have also been considered in \cite{Ma-Ra, Ma-Ra-Ta, Ab-Le-Ru, Wa}.
We obtain the following result in this direction.

\begin{Th}\label{cor:b-1a}
Let $x\in\ca^\N$ be WRAP, let $f\in C(X_x)$ and let $z\in X_x$.
Then for every $\delta>0$ there exists $H_0\in\N$ such that for all $H\geq H_0$ the set of all $m\in\N$ for which
\begin{equation}
\label{eq:sc-ep-4-a}
\left|\frac1H\sum_{m\leq h<m+H}f(S^hz)\boldsymbol{\mu}(h)\right|< \delta.
\end{equation}
has lower density $\geq 1-\delta$.
\end{Th}

It is not clear if \cref{Wshort} can be derived quickly from \cref{cor:b-1a}. However, we will see that  \cref{cor:b-1a} is a corollary of a stronger result which is a strengthening of \cref{Wshort} and which we state next.

\begin{Th}\label{Wshort-2}
Let $x\in\ca^\N$ be WRAP. Then for all $f\in C(X_x)$ and $z\in X_x$,
\begin{equation}
\label{eq:sc-ep-4}
\lim_{H\to\infty\atop \frac{H}{M}\to0} \frac1M\sum_{M\leq m<2M}\Big|\frac1H\sum_{m\leq h<m+H}f(S^hz)\boldsymbol{\mu}(h)\Big|=0.
\end{equation}
\end{Th}

Before providing a proof of \cref{Wshort-2}, let us show that \cref{Wshort-2} implies both \cref{Wshort} and \cref{cor:b-1a}.
We will need the following standard lemma, the proof of which is included for the convenience of the reader.

\begin{Lemma}\label{lem:b-1}
For every $H\in\N$ let $x_H\colon\N\to\C$ be a sequence bounded in modulus by $1$.
If
$$
\lim_{H\to\infty\atop \frac{H}{M}\to0}
\frac1M\sum_{M\leq m<2M}x_H(m) = 0,
$$
then
$$
\lim_{H\to\infty\atop \frac{H}{N}\to0}
\frac1N\sum_{n=1}^N x_H(n) = 0.
$$
\end{Lemma}

\begin{proof}
Let $H_k$ and $N_k$ be two sequences such that $\lim_{k\to\infty}H_k=\infty$ and $\lim_{k\to\infty}\tfrac{H_k}{N_k}=0$.
Let $\ell\in\N$ be arbitrary. We have
\begin{eqnarray*}
\left|\frac1{N_k}\sum_{n=1}^{N_k}x_{H_k}(n)\right|
&= &\left|\sum_{1\leq j\leq \log_2(N_k)}\frac{1}{2^j}\left(\frac{2^j}{N_k}\sum_{\tfrac{N_k}{2^j}\leq m < \tfrac{N_k}{2^{j-1}}}x_{H_k}(m)\right)+\frac{1}{N_k}x_{H_k}(N_k)\right|
\\
&\leq&\sum_{1\leq j\leq \ell}\frac{1}{2^j}\left|\frac{2^j}{N_k}\sum_{\tfrac{N_k}{2^j}\leq m < \tfrac{N_k}{2^{j-1}}}x_{H_k}(m)\right|+\frac{1}{2^\ell}+\frac{1}{N_k},
\end{eqnarray*}
whenever $\ell\leq\log_2(N_k)$. Note that
$$
\lim_{k\to\infty}\frac{2^j}{N_k}\sum_{\tfrac{N_k}{2^j}\leq m < \tfrac{N_k}{2^{j-1}}}x_{H_k}(m)=0,
$$
because $\lim_{k\to\infty} \tfrac{H_k}{N_k/2^j}=0$ for all $j\in\{1,\ldots,\ell\}$. Hence,
$$
\limsup_{k\to\infty}\frac1{N_k}\left|\sum_{n=1}^{N_k}x_{H_k}(n)\right|\leq \frac{1}{2^\ell}.
$$
Since $\ell\in\N$ was arbitrary, this finishes the proof.
\end{proof}

\begin{proof}[Proof that \cref{Wshort-2} implies \cref{Wshort}]
Define
$$
b_m(H):=\frac1H\left|\sum_{m\leq h<m+H}f(S^hz)\mob(h)\right|.
$$
First, we observe that according to \cref{lem:b-1} we have that $\frac1M\sum_{M\leq m<2M}b_m(H)\xrightarrow[]{H\to\infty,\frac{H}{M}\to0} 0$
implies $\frac1N\sum_{n=1}^{N}b_n(H)\xrightarrow[]{H\to\infty,\frac{H}{N}\to0} 0$.

Let $\epsilon>0$ be arbitrary, let $H_k$ and $N_k$ be two sequences such that $\lim_{k\to\infty}H_k=\infty$ and $\lim_{k\to\infty}\tfrac{H_k}{N_k}=0$ and take $J_k:=\{1\leq n\leq N_k: b_n(H_k)\leq \epsilon^2\}$.
It follows from $\lim_{k\to\infty} \frac{1}{N_k}\sum_{n=1}^{N_k}b_n(H_k)=0$ that for sufficiently large $k$ we have $\tfrac{|J_k|}{N_k}\geq 1-\epsilon$. For $t\in\{0,1,\ldots,H_k-1\}$ define $J_{k,t}:=J_k\cap (H_k\Z+t)$.
Then for some $r\in\{0,1,\ldots,H_k-1\}$ we must have
$$
\frac{|J_{k,r}|}{|(H_k\Z+r)\cap\{1,\ldots,N_k\}|}\geq 1-\epsilon.
$$
We get
\begin{eqnarray*}
\left|\frac1{N_k}\sum_{n=1}^{N_k}f(S^hz)\mob(h)\right|
&\leq &
\left|\frac{H_k}{N_k}\sum_{n\in(H_k\Z+r)\cap\{1,\ldots,N_k\}}\frac1{H_k}\sum_{h=n}^{n+H_k-1}f(S^hz)\mob(h)\right|
+\frac{H_k}{N_k}
\\
&\leq &
\frac{H_k}{N_k}\sum_{n\in(H_k\Z+r)\cap\{1,\ldots,N_k\}}b_n(H_k)
+\frac{H_k}{N_k}
\\
&\leq &
\frac{H_k}{N_k}\sum_{n\in J_{k,r}}b_n(H_k)
+\epsilon
+\frac{H_k}{N_k}\leq \frac{H_k|J_{k,r}|\epsilon^2}{N_k}+\epsilon
+\frac{H_k}{N_k}.
\end{eqnarray*}
As $k\to\infty$ the expression $\frac{H_k|J_{k,r}|\epsilon^2}{N_k}+\epsilon
+\frac{H_k}{N_k}$ converges to $\epsilon^2+\epsilon$. Since $\epsilon$ is arbitrary, this finishes the proof.
\end{proof}

\begin{proof}[Proof that \cref{Wshort-2} implies \cref{cor:b-1a}]
We present a proof by contradiction. Assume there exists some $\delta>0$ such that one can find an increasing sequence $(H_k)_{k\geq 1}$ with the property that for every $k$ the set
$$
D_k:=\left\{m\in\N: \left|\frac{1}{H_k}\sum_{m\leq h<m+H_k}f(S^hz)\boldsymbol{\mu}(h)\right|\geq \delta \right\}
$$
satisfies $\overline{d}(D_k)\geq \delta$.
Since $\overline{d}(D_k)\geq \delta$, we can find $N_k\in\N$ such that $H_k\leq \tfrac{N_k}{k}$ and such that
$$
\frac{1}{N_k}\sum_{m=1}^{N_k}\left|\frac{1}{H_k}\sum_{m\leq h<m+H_k}f(S^hz)\boldsymbol{\mu}(h)\right|\geq \delta^2.
$$
This contradicts the fact that according to \cref{Wshort-2} and \cref{lem:b-1},
$$
\lim_{k\to\infty}\frac{1}{N_k}\sum_{m=1}^{N_k}\left|\frac{1}{H_k}\sum_{m\leq h<m+H_k}f(S^hz)\boldsymbol{\mu}(h)\right|=0.
$$
\end{proof}
\begin{Remark}
From \cref{cor:b-1a} it follows that for all $x\in\ca^\N$ that are WRAP, $f\in C(X_x)$ and $z\in X_x$, we have
$$
\lim_{H\to\infty}\limsup_{N\to\infty}\frac{1}{N}\sum_{n=1}^N\left|\frac1H\sum_{h=1}^H f(S^{n+h}z)\boldsymbol{\mu}(n+h)\right|~=~0.
$$
\end{Remark}

The remainder of this section is dedicated to proving \cref{Wshort-2}.
The following lemma (which is a variant of \cref{lem:sc-ep-W}) will be useful for the proof of \cref{Wshort-2}.

\begin{Lemma}\label{lem:sc-ep-W-2}
Suppose $x\in\ca^\N$ is WRAP and $f\in C(X_x)$. Then
\begin{equation}
\label{eq:sc-ep-2}
\lim_{H\to\infty\atop \frac{H}{M}\to0}\frac1M\sum_{M\leq m<2M}\Big|\frac1H\sum_{m\leq h<m+H}f(S^h\tilde x)\boldsymbol{\mu}(h)\Big|=0.
\end{equation}
where $\tilde x\in\ca^\Z$ is any two-sided sequence extending $x\in\ca^\N$.
\end{Lemma}

\begin{proof}
Since any continuous function $f\in C(X_x)$ can be approximated uniformly by cylinder sets $C=C^{\alpha_1,\ldots,\alpha_\ell}_{n_1,\ldots,n_\ell}$, it suffices to show \eqref{eq:sc-ep-2} for indicator functions of cylinder sets. 

Let $\ell\in\N$, $n_1,\ldots,n_\ell\in\Z$ and $\alpha_1,\ldots,\alpha_\ell\in\ca$ be arbitrary. Fix $\epsilon>0$ and let $y\in\ca^\N$ be a periodic sequence such that $d_W(x,y)\leq \epsilon/\ell$. Let $\tilde y\in\ca^\Z$ be a two-sided periodic sequence that extends $y$.
From \cref{lem:sc-ep-0} it follows that
\begin{equation}
\label{eq:sc-ep-9}
\limsup_{H\to\infty}\sup_{m\in\N}\left|\frac{1}{H}
\sum_{m\leq h<m+H} \big(\1_C(S^n \tilde x)-\1_C(S^n \tilde y)\big)\mob(n)\right| \leq\ell d_W(x,y)\leq\epsilon.
\end{equation}
By a recent result of Matomaki, Radziwi\l \l \ and Tao \cite{Ma-Ra-Ta}, we have that for each periodic sequence $a(n)$:
\begin{equation*}\label{sarSHORT}
\frac1M\sum_{M\leq m<2M}\left|\frac1H\sum_{m\leq h<m+H}a(h)\boldsymbol{\mu}(h)\right|\to0\text{ as }H\to\infty,H/M\to0.
\end{equation*}
Choosing $a(n)=\1_C(S^n \tilde y)$ we thus get
\begin{equation}\label{sarSHORT-2}
\frac1M\sum_{M\leq m<2M}\left|\frac1H\sum_{m\leq h<m+H}\1_C(S^h \tilde y)\boldsymbol{\mu}(h)\right|\to0\text{ as }H\to\infty,H/M\to0.
\end{equation}
Combining \eqref{eq:sc-ep-9} and \eqref{sarSHORT-2}, we obtain
\begin{equation*}
\begin{split}
\lim_{H\to\infty\atop \frac{H}{M}\to0}&\frac1M\sum_{M\leq m<2M} \Big|\frac1H\sum_{m\leq h<m+H}\1_C(S^h\tilde x)\boldsymbol{\mu}(h)\Big|
\\
&~\leq~
\lim_{H\to\infty\atop \frac{H}{M}\to0}\frac1M\sum_{M\leq m<2M}\Big|\frac1H\sum_{m\leq h<m+H}\1_C(S^h\tilde y)\boldsymbol{\mu}(h)\Big| +\epsilon ~=~\epsilon.
\end{split}
\end{equation*}
Since $\epsilon$ is arbitrarily, the proof of \eqref{eq:sc-ep-2} is completed.
\end{proof}

\begin{proof}[Proof of \cref{Wshort-2}]
The following argument is analogous to the one used in the proof of \cref{Wshort}: 
Let $x\in\ca^\N$ be WRAP and let $f\in C(X_x)$ and $z\in X_x$ be arbitrary. 
It follows from \cref{l:flo} that $z|_\N$ is WRAP. Therefore, equation \eqref{eq:sc-ep-4} follows from \eqref{eq:sc-ep-2}.
\end{proof}

\appendix
\section{Uniformity of polynomial multiple recurrence}
In this appendix we derive a uniform version of the following polynomial multiple recurrence theorem obtained in \cite{MR1325795}:

\begin{Th}[see {\cite[Theorem A]{MR1325795}}]
\label{thm:mtomrvl}
Let $\ell,u\in\N$ and let $p_{i,j}\in\Q[t]$ be polynomials satisfying $p_{i,j}(\Z)\subset\Z$ and $p_{i,j}(0)=0$, $i=1,\ldots,\ell$, $j=1,\ldots,u$. Then for any probability space $\xbm$, any $u$-tuple of commuting invertible measure preserving transformations $T_1,\ldots,T_u$ on $\xbm$ and any $A\in\cb$ with $\mu(A)>0$ one has
\[
\liminf_{N\to\infty} \frac1N\sum_{n=1}^{N} \mu\left(A\cap\prod_{j=1}^u T_j^{-p_{1,j}(n)}A\cap \prod_{j=1}^u T_j^{-p_{2,j}(n)}A\cap\ldots\cap\prod_{j=1}^u T_j^{-p_{\ell,j}(n)}A\right)>0.
\]
\end{Th}

The uniform version in question is given by the following theorem (a special case of it was used in the proofs of Theorems \ref{p5new} and \ref{thm_poly-multi-rec-along-disc-rat-spec-Nk}).

\begin{Th}\label{Cor-uniformity}
For all $\ell,d\in\N$ and all $\vep>0$ there exists $\delta>0$ such that the following holds: For any $u\in \N$, for any polynomials $p_{i,j}\in\Q[t]$, $i=1,\ldots,\ell$, $j=1,\ldots,u$, satisfying $\deg(p_{i,j})\leq d$, $p_{i,j}(\Z)\subset\Z$, $p_{i,j}(0)=0$, for any probability space $\xbm$, for any $u$-tuple of commuting invertible measure preserving transformations $T_1,\ldots,T_u$ on $\xbm$, for any $A\in\cb$ with $\mu(A)\geq\vep$ and for any $s\in\N$ one has
\begin{equation}
\label{eq:ameq-1}
\begin{split}
\lim_{N-M\to\infty} \frac1{N-M}\sum_{n=M}^{N-1} \mu\Bigg(A\cap\prod_{j=1}^u T_j^{-p_{1,j}(sn)}A\cap \prod_{j=1}^u T_j^{-p_{2,j}(sn)} A & \cap\ldots
\\
\ldots & \cap \prod_{j=1}^u T_j^{-p_{\ell,j}(sn)}A\Bigg)>\delta.
\end{split}
\end{equation}
\end{Th}

We remark that a slightly less general version of \cref{Cor-uniformity}
is stated in \cite[Theorem 4.1]{FHK13} without a proof.

In the course of proving \cref{Cor-uniformity} we will make use of the following equivalent combinatorial form of \cref{thm:mtomrvl}.

\begin{Th}[see {\cite[Theorem 3.2]{MR1784213}}]\label{Th-finite-PSz}
Let $\ell,u\in\N$, let $\vep>0$ and let $p_{i,j}\in\Z[t]$ be polynomials satisfying $p_{i,j}(0)=0$, $i=1,\ldots,\ell$, $j=1,\ldots,u$. Then there exists a positive integer $N=N(\ell,u,\vep,p_{i,j})$ such that for all sets $A\subset \Z^d$ with
\[
\frac{\vert A\cap [1,N]^u\vert }{N^u}>\vep
\]
there exist $n\in\N$ and $a\in A$ such that $a+(p_{i,1}(n),\ldots,p_{i,u}(n))\in A$ for all $i\in\{1,2,\ldots,\ell\}$.
\end{Th}

We will need the following theorem, which is of independent interest and can be interpreted as a polynomial extension of Theorem F2 in \cite{MR1784213}.

\begin{Th}\label{Cor-uniformity-1}
For every $\ell,d\in\N$ and every $\vep>0$ there exist $K\in\N$ and $\beta>0$ such that for any probability space $\xbm$, any commuting invertible measure preserving transformations $T_{i,j}$, $1\leq i\leq \ell$ and $1\leq j\leq d$, and any $A\in\cb$ with $\mu(A)\geq\vep$ there exists $n\in\{1,\ldots,K\}$ such that
\begin{equation}
\label{eq:asbr-0}
\mu\left(A\cap\prod_{j=1}^d T_{1,j}^{-n^j}A\cap \prod_{j=1}^d T_{2,j}^{-n^j}A\cap\ldots\cap\prod_{j=1}^d T_{\ell,j}^{-n^j}A\right)>\beta.
\end{equation}
Moreover,
\begin{equation}
\label{eq:asbr-1}
\lim_{N-M\to\infty}\frac{1}{N-M}\sum_{n=M}^{N-1} \mu\left(A\cap\prod_{j=1}^d T_{1,j}^{-n^j}A\cap \prod_{j=1}^d T_{2,j}^{-n^j}A\cap\ldots\cap\prod_{j=1}^d T_{\ell,j}^{-n^j}A\right)\geq \frac{\beta}{K^2}.
\end{equation}
\end{Th}

\begin{proof}
Let $u:=d\ell$ and, for $1\leq i\leq \ell$ and $1\leq t\leq u$, define
\begin{equation}
\label{eq:asbr-2}
p_{i,t}(n)=
\begin{cases}
n^j,&\text{if}~t=(i-1)d+j~\text{with}~1\leq i\leq \ell~\text{and}~1\leq j\leq d;\\
0,&\text{otherwise}.
\end{cases}
\end{equation}
Let $K=N(\ell,u,\vep/2,p_{i,t})$ as guaranteed by \cref{Th-finite-PSz}.
For the remainder of this proof let us call a set of the form $\{a\}\cup \{a+(p_{i,1}(n),\ldots,p_{i,u}(n)): 1\leq i\leq \ell\}$ for some
$a=(a_1,\ldots,a_u)\in \N^u$ and $n\in\N$ a {\em basic arrangement}.
Let $J$ denote the collection of all basic arrangements contained in $\{1,\ldots,K\}^u$.
Set $\beta:=\frac{\vep}{4|J|}$.
We claim that \eqref{eq:asbr-0} and \eqref{eq:asbr-1} are satisfied with this choice of $K$ and $\beta$.

Let $\xbm$ be an arbitrary probability space, let $T_{i,j}$, $1\leq i\leq \ell$ and $1\leq j\leq d$, be commuting invertible measure preserving transformations on $X$ and let $A\in\cb$ with $\mu(A)\geq\vep$.
For $1\leq t\leq u$ let $S_t:= T_{i,j}$ where $(i,j)\in \{1,\ldots,\ell\}\times\{1,\ldots,d\}$ is such that $t=(i-1)d+j$. 
It thus follows from \eqref{eq:asbr-2} that
\begin{equation}
\label{eq:asbr-3}
\prod_{j=1}^d T_{i,j}^{n^j}
=
\prod_{t=1}^u S_t^{p_{i,t}(n)}.
\end{equation}

Define
$$
f(x):= \frac{1}{K^u}\sum_{(n_1,\ldots,n_u)\in [1,K]^u} \1_A\left(\prod_{t=1}^u S_t^{n_t} (x)\right).
$$
Clearly, $f$ is a non-negative function and $\int_X f\,d\mu \geq \vep$. Therefore the set $B:=\{x\in X: f(x)\geq \vep/2\}$ satisfies $\mu(B)\geq \vep/2$. Also, for every $x\in B$ the set
$$
E_x:=\left\{(n_1,\ldots,n_u)\in[1,K]^u:\prod_{t=1}^u S_t^{n_t} (x)\in A\right\}
$$
has density at least $\vep/2$ in $[1,K]^u$, i.e.\ $|E_x|\geq (\vep/2) K^u$.
By our choice of $K$, we are guaranteed to find at least one basic arrangement contained in $E_x$.

We have shown that for every $x\in B$ there exists a
basic arrangement contained in $E_x\subset[1,K]^u$. Since there are $|J|$-many basic arrangements in $[1,K]^u$, by the pigeonhole principle there exists a set $C\subset B$ with $\mu(C)\geq \frac{\vep}{2|J|}$ such that $E_x$ contains the same basic arrangement for every $x\in C$. Suppose this basic arrangement
is given by $\{(a_1,\ldots,a_u)\}\cup \{(a_1,\ldots,a_u)+(p_{i,1}(n),\ldots,p_{i,u}(n)): 1\leq i\leq \ell\}$.
Let $C':=\prod_{t=1}^u S_t^{a_t} C$. Then for any $x'\in C'$ and any $i\in\{1,\ldots,\ell\}$, if $x:= \prod_{t=1}^u S_t^{-a_t} (x')$ then by \eqref{eq:asbr-3} and the definition of $E_x$ we have
$$
\prod_{j=1}^d T_{i,j}^{n^j}(x')
~=~
\prod_{t=1}^u S_t^{p_{i,t}(n)}(x')
~=~
\prod_{t=1}^u S_t^{a_t + p_{i,t}(n)}(x) ~\in A.
$$
This shows that $C'$ is contained in the intersection $A\cap\prod_{j=1}^d T_{1,j}^{-n^j}A\cap \prod_{j=1}^d T_{2,j}^{-n^j}A\cap\ldots\cap\prod_{j=1}^d T_{\ell,j}^{-n^j}A$. Since $\mu(C')=\mu(C)>\beta$, this finishes the proof of \eqref{eq:asbr-0}.

Next, we give a proof of \eqref{eq:asbr-1}.
Let $M\geq 1$ be arbitrary. Note that for all $m$ with $M(K-1)<m\leq MK$ and all $k$ with $1\leq k\leq K$ the products $mk$ are pairwise distinct. 
For $1\leq j\leq d$, $1\leq m\leq M$ and $1\leq i \leq \ell$ define $R_{i,j,m}:=T_{i,j}^{m^j}$.
It follows that
\begin{equation*}
\begin{split}
&\frac{1}{M K^2}\sum_{n=1}^{MK^2}\mu\left(A\cap\prod_{j=1}^d T_{1,j}^{-n^j}A\cap \prod_{j=1}^d T_{2,j}^{-n^j}A\cap\ldots\cap\prod_{j=1}^d T_{\ell,j}^{-n^j}A\right)
\\
&
\geq \frac1{MK^2}\sum_{m=M(K-1)+1}^{MK}~~\sum_{k=1}^K\mu\left(A\cap\prod_{j=1}^d T_{1,j}^{-(mk)^j}A\cap \prod_{j=1}^d T_{2,j}^{-(mk)^j}A\cap\ldots\cap\prod_{j=1}^d T_{\ell,j}^{-(mk)^j}A\right)
\\
&
= \frac1{MK^2}\sum_{m=M(K-1)+1}^{MK}~~\sum_{k=1}^K\mu\left(A\cap\prod_{j=1}^d R_{1,j,m}^{-k^j}A\cap \prod_{j=1}^d R_{2,j,m}^{-k^j}A\cap\ldots\cap\prod_{j=1}^d R_{\ell,j,m}^{-k^j}A\right).
\end{split}
\end{equation*}
In light of \eqref{eq:asbr-0} we have
$$
\sum_{k=1}^K\mu\left(A\cap\prod_{j=1}^d R_{1,j,m}^{-k^j}A\cap \prod_{j=1}^d R_{2,j,m}^{-k^j}A\cap\ldots\cap\prod_{j=1}^d R_{\ell,j,m}^{-k^j}A\right)>\beta
$$
for all $1\leq m\leq M$.
Therefore, 
\begin{equation*}
\begin{split}
\frac1{MK^2}\sum_{m=M(K-1)+1}^{MK}~~\sum_{k=1}^K&\mu\left(A\cap\prod_{j=1}^d R_{1,j,m}^{-k^j}A\cap \prod_{j=1}^d R_{2,j,m}^{-k^j}A\cap\ldots\cap\prod_{j=1}^d R_{\ell,j,m}^{-k^j}A\right)
\\
&>
\frac1{MK^2}\sum_{m=M(K-1)+1}^{MK}\beta
\\
&=\frac{\beta}{K^2}.
\end{split}
\end{equation*}
This proves that
\begin{equation}
\label{eq:bmcp-tr}
\liminf_{M\to\infty}
\frac{1}{M K^2}\sum_{n=1}^{MK^2}\mu\left(A\cap\prod_{j=1}^d T_{1,j}^{-n^j}A\cap \prod_{j=1}^d T_{2,j}^{-n^j}A\cap\ldots\cap\prod_{j=1}^d T_{\ell,j}^{-n^j}A\right)\geq \frac{\beta}{K^2}.
\end{equation}
Finally, it follows from the results in \cite{Walsh12} that the limits on the left hand side of \eqref{eq:bmcp-tr} and on the left hand side of \eqref{eq:asbr-1} exist and are equal. This finishes the proof of \eqref{eq:asbr-1}.
\end{proof}

\begin{proof}[Proof of \cref{Cor-uniformity}]
Depending only on $\ell,d\in\N$ and $\epsilon>0$, choose $\beta>0$ and $K\geq 1$ as guaranteed by \cref{Cor-uniformity-1}. Note that coefficients of integer polynomials of degree $d$ can be written as fractions with denominator $q:=d!$. Define $b:=q!$ and pick any $\delta>0$ such that $\delta<\frac{\beta}{bK^2}$. We claim that \eqref{eq:ameq-1} holds with this choice of $\delta$.

Let $u,s\in \N$ and let $p_{i,j}\in\Q[t]$, $i=1,\ldots,\ell$, $j=1,\ldots,u$, with $\deg(p_{i,j})\leq d$, $p_{i,j}(\Z)\subset\Z$, $p_{i,j}(0)=0$ and such that the denominators of the coefficients of $p_{i,j}$ (when written as reduced fractions) are at most $q$. Furthermore, let $T_1,\ldots,T_u$ be commuting invertible measure preserving transformations on a probability space $\xbm$ and let $A\in\cb$ with $\mu(A)\geq\vep$.
It follows from \cite{Walsh12} that the limit on the left hand side of \eqref{eq:ameq-1} exists and is equal to 
\begin{equation}
\label{eq:ameq-2}
\lim_{N\to\infty} \frac1{N}\sum_{n=1}^{N} \mu\Bigg(A\cap\prod_{j=1}^u T_j^{-p_{1,j}(sn)}A\cap \prod_{j=1}^u T_j^{-p_{2,j}(sn)} A \cap
\ldots \cap \prod_{j=1}^u T_j^{-p_{\ell,j}(sn)}A\Bigg).
\end{equation}
It thus suffices to show that \eqref{eq:ameq-2} is bigger than $\delta$.

For $1\leq i\leq \ell$ and $1\leq j\leq u$ find $a_{i,j}^{(1)},\ldots,a_{i,j}^{(d)}$ such that
$$
p_{i,j}(n)=a_{i,j}^{(1)}n+a_{i,j}^{(2)}n^2+\ldots+a_{i,j}^{(d)}n^d.
$$
By assumption, $b a_{i,j}^{(k)}\in \Z$ and hence $s^kb^k a_{i,j}^{(k)}\in \Z$ for all $s\in \N$. Define
$$
R_{i,k}:=\prod_{j=1}^{u}T_j^{s^k b^k a_{i,j}^{(k)}}, \qquad 1\leq k\leq d.
$$
Clearly,
$$
\prod_{j=1}^u T_j^{p_{i,j}(bsn)}=\prod_{j=1}^d R_{i,j}^{n^j},\qquad\forall n\in\N.
$$
We thus have
\begin{equation*}
\begin{split}
\frac1{N}\sum_{n=1}^{N} \mu\left(A\cap\prod_{j=1}^u T_j^{-p_{1,j}(sn)}A\cap \prod_{j=1}^u T_j^{-p_{2,j}(sn)}A\cap\ldots\cap\prod_{j=1}^u T_j^{-p_{\ell,j}(sn)}A\right)
\\
\geq\frac1{N}\sum_{n=1}^{\lfloor N/b\rfloor } \mu\left(A\cap\prod_{j=1}^u T_j^{-p_{1,j}(bsn)}A\cap \prod_{j=1}^u T_j^{-p_{2,j}(bsn)}A\cap\ldots\cap\prod_{j=1}^u T_j^{-p_{\ell,j}(bsn)}A\right)
\\
=\frac1{N }\sum_{n=1}^{\lfloor N/b\rfloor } \mu\left(A\cap\prod_{j=1}^d R_{1,j}^{-n^j}A\cap \prod_{j=1}^d R_{2,j}^{-n^j}A\cap\ldots\cap\prod_{j=1}^d R_{\ell,j}^{-n^j}A\right).
\end{split}
\end{equation*}
From \eqref{eq:asbr-1} it follows that
$$
\lim_{N\to\infty}\frac1{N}\sum_{n=1}^{\lfloor N/b\rfloor } \mu\left(A\cap\prod_{j=1}^d R_{1,j}^{-n^j}A\cap \prod_{j=1}^d R_{2,j}^{-n^j}A\cap\ldots\cap\prod_{j=1}^d R_{\ell,j}^{-n^j}A\right)
\geq \frac{\beta}{bK^2}>\delta.
$$
Therefore,
$$
\lim_{N\to\infty} \frac1N\sum_{n=1}^{N} \mu\left(A\cap\prod_{j=1}^u T_j^{-p_{1,j}(sn)}A\cap \prod_{j=1}^u T_j^{-p_{2,j}(sn)}A\cap\ldots\cap\prod_{j=1}^u T_j^{-p_{\ell,j}(sn)}A\right)>\delta.
$$
\end{proof}

\small
\bibliography{BibAF}

\bigskip
\footnotesize

\noindent
Vitaly Bergelson\\
\textsc{Department of Mathematics, Ohio State University, Columbus, OH 43210, USA}\par\nopagebreak
\noindent
\href{mailto:vitaly@math.ohio-state.edu}
{\texttt{vitaly@math.ohio-state.edu}}

\medskip

\noindent
Joanna Ku\l aga-Przymus\\
\textsc{Faculty of Mathematics and Computer Science, Nicolaus Copernicus University, Chopina 12/18, 87-100 Toru\'{n}, Poland}\par\nopagebreak
\noindent
\href{mailto:joanna.kulaga@gmail.com}
{\texttt{joanna.kulaga@gmail.com}}

\medskip

\noindent
Mariusz Lema\'nczyk\\
\textsc{Faculty of Mathematics and Computer Science, Nicolaus Copernicus University, Chopina 12/18, 87-100 Toru\'{n}, Poland}\par\nopagebreak
\noindent
\href{mailto:mlem@mat.umk.pl}
{\texttt{mlem@mat.umk.pl}}

\medskip

\noindent
Florian K.\ Richter\\
\textsc{Department of Mathematics, Ohio State University, Columbus, OH 43210, USA}\par\nopagebreak
\noindent
\href{mailto:richter.109@osu.edu}
{\texttt{richter.109@osu.edu}}
\end{document}